\documentclass[11pt,english]{article}
\usepackage[english]{babel}

\usepackage{graphicx,subfigure}

\usepackage[usenames, dvipsnames]{xcolor}
\usepackage[utf8x]{inputenc}
\usepackage[T1]{fontenc}
\usepackage{tikz}
\usepackage{pgfplots}

\usepackage{amsmath}
\usepackage{amsfonts}
\usepackage{amssymb}
\usepackage{amsthm}

\usepackage{esint}

\usepackage[colorlinks=true,urlcolor=blue,
citecolor=red,linkcolor=blue,linktocpage,pdfpagelabels,
bookmarksnumbered,bookmarksopen,backref=page]{hyperref}

\usepackage{cite} 

\usepackage{fullpage}

\DeclareMathOperator{\sech}{sech}

 \newcommand{\RE}{\mathrm{Re}\,}
 
 \DeclareMathOperator{\sn}{sn}
 \DeclareMathOperator{\diag}{diag}
 
\newtheorem{theorem}{Theorem}[section]
 
 \newtheorem{lemma}[theorem]{Lemma}
 \newtheorem{proposition}[theorem]{Proposition}
 \theoremstyle{definition}
 \newtheorem{definition}[theorem]{Definition}
 \theoremstyle{remark}
 \newtheorem{remark}[theorem]{Remark}
 
 \numberwithin{equation}{section}

\newcommand{\R}{\mathbb R}

\newcommand{\N}{\mathbb N}
\newcommand{\C}{\mathbb C}

\newcommand{\cG}{\mathcal{G}}
\newcommand{\cE}{\mathcal{E}}
\newcommand{\cF}{\mathcal{F}}
\newcommand{\cA}{\mathcal{A}}

\newcommand{\cS}{\mathcal{S}}

\newcommand{\cn}{\mathrm{cn}}
\newcommand{\dn}{\mathrm{dn}}

\newcommand{\sac}{\sigma_\mathrm{\tiny{ac}}}
\newcommand{\sess}{\sigma_\mathrm{\tiny{ess}}}
\newcommand{\ptsp}{\sigma_\mathrm{\tiny{pt}}}

\begin{author}
{Jaime Angulo Pava$^{1}$ $\;$ and Ram\'on G. Plaza$^{2, }$\thanks{Corresponding author.}}
\end{author}

\begin{title}{Dynamics of the sine-Gordon equation on tadpole graphs}
\end{title}
\begin{document}
\maketitle

\centerline{$^1$ Department of Mathematics,
IME-USP}
 \centerline{Rua do Mat\~ao 1010, Cidade Universit\'aria, CEP 05508-090,
 S\~ao Paulo, SP (Brazil)}
 \centerline{\tt angulo@ime.usp.br}
 \centerline{ $^2$ Instituto de Investigaciones en Matem\'aticas Aplicadas
   y en Sistemas,}
 \centerline{Universidad Nacional Aut\'{o}noma de M\'{e}xico,  Circuito Escolar s/n,}
 \centerline{Ciudad Universitaria, C.P. 04510 Cd. de M\'{e}xico (Mexico)}
 \centerline{\tt  plaza@aries.iimas.unam.mx}

\begin{abstract}
This work studies the dynamics of solutions to the sine-Gordon equation posed on a tadpole graph, namely, 
a graph consisting of a  circle with a half-line attached at a single vertex,  and endowed with boundary conditions at the vertex of $\delta$-type. The latter  generalize conditions of Neumann-Kirchhoff type. The purpose of this analysis is to establish the existence and instability  of stationary solutions which we have  called \emph{single-lobe kink state profiles}, which consist of a symmetric stationary solution of subluminal-type profile in the finite ring of the tadpole, coupled with a decaying kink-profile at the infinite edge of the graph. It is proved that such stationary profile solutions are linearly (and nonlinearly) unstable under specific restriction and by  the flow of the sine-Gordon model on the graph. The extension theory of symmetric operators, Sturm-Liouville oscillation results and a splitting eigenvalue method are fundamental ingredients in the stability analysis. The local well-posedness of the sine-Gordon model in an appropriate energy space is also established. The theory developed in this investigation constitutes the first steps  by understading the dynamics to the sine-Gordon equation on a tadpole graph.
\end{abstract}

\textbf{Mathematics  Subject  Classification (2020).} Primary 35Q51, 35J61, 35B35; Secondary 47E05.\\

\textbf{Key  words.} sine-Gordon model,  tadpole graph,  single-lobe kink solutions, $\delta$-type interaction,  extension theory of symmetric operators, Sturm-Liouville theory, instability.


\section{Introduction}


%
%

Recently, nonlinear models posed on metric graphs, such as the nonlinear Schrödinger equation, the sine-Gordon model, and the Korteweg-de Vries equation, have garnered considerable attention (for an abridged list of related works the reader is referred to \cite{AdaNoj14, AngGol18a, AngGol18b, AC, AC1, AP1, AP2, AP3, AST15, KP, KPG, KMPX, KNP22, Noj14}). Real systems can exhibit strong inhomogeneities due to varying nonlinear coefficients across different regions of the spatial domain or due to the specific geometry of the domain itself. Addressing these challenges is difficult due to the complexity of both the equations of motion and the graph topology. A particular issue complicating the analysis is the behavior at the vertices, where a soliton profile traveling along one edge may experience complex phenomena like reflection and radiation emission. This makes it hard to track how energy propagates through the network. Consequently, studying soliton propagation in networks presents significant obstacles. The mechanisms governing the existence and stability (orbital or spectral) of soliton profiles remain uncertain for many types of graphs and models. This project intends to contribute with the exploration of these topics in the particular case of the  sine-Gordon equation,
\begin{equation}
\label{sG0}
\partial_t ^2 \textbf{U}-  \textbf{c}^2\Delta \textbf{U}+ \sin \textbf{U} = 0.
\end{equation}
To explore key features of the sine-Gordon model on metric graphs we therefore choose a relatively simple metric graph domain known as the \emph{tadpole graph} (see Figure 1), a study that has not been addressed in the literature, as far as we know.

We recall that a metric graph $\mathcal{G}$ is a structure represented by a finite number of vertices $V=\{\nu_i\}$ and a set of adjacent edges at the vertices $E=\{e_j\}$ (for further details, see \cite{BK}). Each edge $e_j$ can be identified with a finite or infinite interval of the real line, $I_e$. The notation $e \in E$ will be used to indicate that $e$ is an edge of $\mathcal{G}$. This identification introduces the coordinate $x_e$ along the edge $e$. A quantum graph is a metric graph $\mathcal{G}$ equipped with a differential (or pseudo-differential) operator acting on functions on the graph. A tadpole graph is a graph composed of a circle and a  infinite half-line attached to a common vertex $\nu = L$. If the circle is identified with the interval $[-L, L]$ and the half-line with $[L, \infty)$, we obtain a particular metric graph structure that we will denote again by $\mathcal{G}$, represented by the edge-set $E = \{[-L, L], [L, \infty)\}$ and the unique vertex $V = \{L\}$ (see Figure \ref{tadp}). We use the subscripts $1$ and $2$ to refer to the edges $e_1 = [-L, L]$ and $e_2 = [L, \infty)$.
\begin{figure}[h]
 	\centering
\includegraphics[angle=0,scale=0.5]{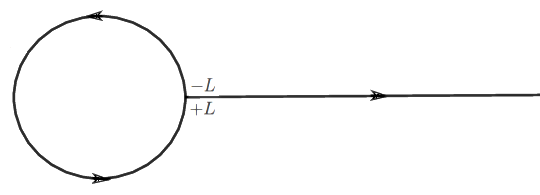}
	\caption{\small{Tadpole graph $\cG$ with vertex at $\nu = L$.} \label{tadp}}
\end{figure}

A wave function $\textbf{U}$ defined on the tadpole graph $\mathcal{G}$ will be understood as an pair of functions $\textbf{U} = (u_1, u_2)$, where $u_1$ is defined on $e_1$ and $u_2$ on $e_2 $.  Thus, we define the action of the Laplacian operator $\Delta$ on the tadpole graph $\mathcal{G}$   by
\begin{equation}
\label{Laplaciantadpole}
-\Delta \textbf{U} = ( -u''_1, -u''_2).
\end{equation}
Thus, from \eqref{sG0}, for $\mathbf{U}(t)=(u_1(\cdot, t), u_2(\cdot, t))$ and   the nonlinearity $ \sin \mathbf{U}$ acting componentwise, i.e.,
$ \sin \bold U(t) = (\sin u_1(\cdot, t), \sin u_2(\cdot, t))$, we obtain that for $\textbf{c}^2=(c_1^2, c_2^ 2)$ the model in \eqref{sG0} can be written in  the  following vectorial form on a tadpole graph $\mathcal{G}$,
\begin{equation}\label{sG}
\partial_t ^2u_j- c_j^2 \partial_x ^2 u_j+ \sin u_j =0,\qquad x\in e_j,\quad c_j>0, \quad j= 1, 2.
\end{equation}

Posing the sine-Gordon equation  on a metric graph comes out naturally from practical applications. For example, since the phase-difference in a long (infinite) Josephson junction obeys the classical sine-Gordon model on the line (cf. \cite{BEMS,Jsph65,SvG05}), the coupling of three  or more Josephson junctions forming a network can be effectively modeled by the vectorial sine-Gordon model in \eqref{sG0} on a graph. The sine-Gordon equation was first conceived on a $\mathcal{Y}$-shaped Josephson junction by Nakajima \emph{et al.} \cite{NakO76,NakO78} as a prototype for logic circuits. Recently, Angulo and Plaza \cite{AP1, AP2, AP3} studied the sine-Gordon model on $\mathcal{Y}$-junctions  by obtaining an  comprehensive study of  the linear (and nonlinear) instability  of static, kink and kink/anti-kink soliton profile solutions.

The goal of this paper is to shed new light on the study of the dynamics of the  sine-Gordon vectorial model in \eqref{sG0} in the case of a tadpole graph. Hence, we will consider the sine-Gordon model \eqref{sG} as a first order system on the bounded interval and on the half-line, respectively,
\begin{equation}
\label{sg2}
\begin{cases}
\partial_t u_j = v_j\\
\partial_t v_j =c_j^2 \partial_x^2 u_j - \sin u_j,
\end{cases}
\qquad x \in e_j, \;\,  t > 0, \;\, 1 \leqq j \leqq 2
\end{equation}
with $u_j=u_j(x,t)$,  $v_j=v_j(x,t)$ and $u_1, v_1:  [-L,L]\to \mathbb R$,  $u_2, v_2:  [L,\infty) \to \mathbb R$. Here, we are interested in the dynamics generated by the flow of the sine-Gordon  on the tadpole around solutions of static or stationary type,
\begin{equation}\label{trav20}
u_j (x,t) = \phi_j(x), \qquad v_j(x,t) = 0,
\end{equation}
for all $j = 1,2$, and $x \in e_j$, $t > 0$, where each of the profile functions $\phi_j$ satisfies the classical kink-equation
\begin{equation}
\label{trav21}
-c_j^2 \phi''_j + \sin \phi_j=0, 
\end{equation}
on each edge $e_j$ and for all $j$, as well as the boundary conditions at the vertex $\nu = L$ being the so-called  $\delta$-coupling or generalized Neumann-Kirchhoff boundary conditions, namely, $(\phi_1, \phi_2)\in  D_Z$
\begin{equation}
\label{Domain0}
D_Z =\{(f,g)\in H^2(-L,L)\times H^2(L, \infty): f(L)=f(-L)=g(L), \, f'(L)-f'(-L)=g'(L)+Z g(L)\},
\end{equation}   
  where  for any $n\geqq 0$,
 $$
 H^n(\mathcal G)=H^ n(-L, L)\oplus  H^ n( L, \infty).
 $$
The parameter $Z\in \mathbb R$ is a coupling constant between the disconnected loop and the half-line. The values of $f$, $g$ and their derivatives at the vertex (and at $-L$) are understood as the appropriate one-sided limits (observe that, as  $(f,g)\in H^2(-L,L)\times H^2(L, \infty)$,  $f'(\pm L)$ and $g'(L+)$ are well defined). We note that $(-\Delta,   D_Z)_{Z\in \mathbb R}$ represents a one-parameter family of self-adjoint operators on a tadpole graph via the extension theory of symmetric operators by Krein and von Neumann (see Theorem A.6 in \cite{Angtad}). 

The delicate point about the existence  of  stationary solutions solving \eqref{trav21} is given by the component $\phi_1$ on $[-L,L]$. The profile $\phi_2$ on $[L,\infty)$ is given by  the well-known kink-soliton profile  on the full real line modulo translation (see \cite{Dra83,SCM}), 
\begin{equation}
\label{trav22}
\phi_2(x)\equiv \phi_{2,a}(x) = 4 \arctan \big[e^{-\frac{1}{c_2}(x-L+a)}\big], \quad\;\; x \in [L, \infty),
\end{equation}
where  we are interested for the case $c_2>0$ ($\lim_{x\to +\infty} \phi_2(x)=0$) and the shift parameter $a\in \mathbb R$ will be determined by the condition $(\phi_1, \phi_2)\in D_Z$ (see Figure \ref{fig2}).
 
%
\begin{figure}[h]
 \centering
\includegraphics[angle=0,scale=0.55]{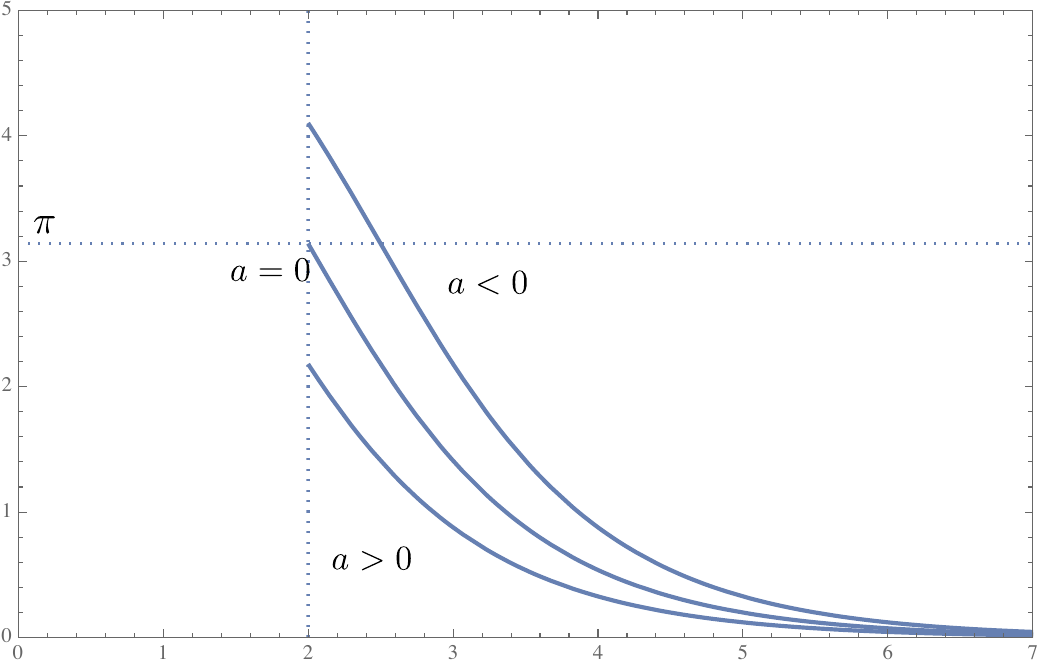}
\caption{\small{Graph of the kink-soliton profile (in solid blue line) for $x \in [L, \infty)$ for different values of $a$. Here $L =2$ and $c_2 = 1$ (color online).} \label{fig2}} 
\end{figure}

Among all stationary solutions for \eqref{trav21}, specifically on $[-L, L]$, we are particularly interested in the so-called \emph{single-lobe kink states}, an example of which is shown on Figure \ref{fig3}. More precisely, we have the following definition.

\begin{figure}[h]
 \centering
\includegraphics[angle=0,scale=0.55]{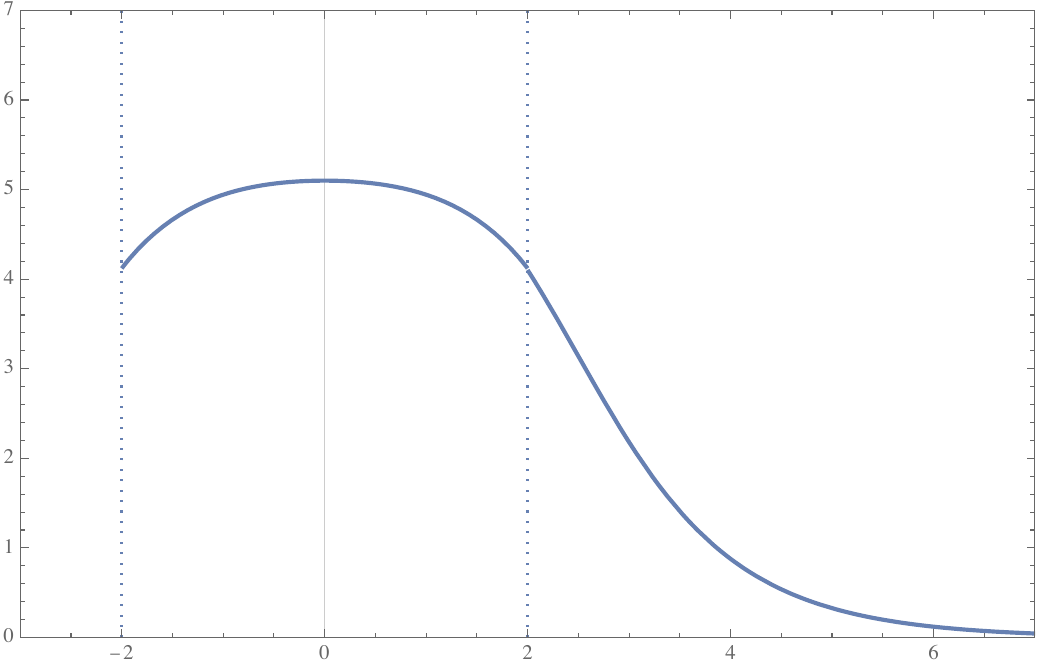}
\caption{\small{An example of a single-lobe kink state profile for the sine-Gordon model according to Definition \ref{singlelobe}.} \label{fig3}} 
\end{figure}
 

\begin{definition}
\label{singlelobe}
The stationary profile solution $\Theta = (\phi_1, \phi_2) \in D_Z$ is said to be a \emph{single-lobe kink state} for \eqref{trav21} if every component is positive on each edge of the tadpole graph $\cG$, the maximum of $\phi_1$ is achieved at a single internal point, symmetric on $[-L,L]$ and monotonically decreasing on $[0,L]$. Moreover, $\phi_2$ is strictly decreasing on $[L,\infty)$.
\end{definition}

 In  Figure \ref{anti} (see formula \eqref{profile}), we show other profile type for $\phi_1$ solution of \eqref{trav21} on $[-L, L]$. 
 
The existence and dynamics  of positive single-lobe kink states  for the sine-Gordon model  on a tadpole graph have not been addressed  in the literature and, consequently, one of the purposes of this article is to take the first steps towards a better understanding of the dynamics of these profiles by the flow of the sine-Gordon model. In particular, in Proposition \ref{Z2} below, we show  that there are  single-lobe kink states $(\phi_1,\phi_2)\in D_Z$ provided that $Z\in (-\infty, \frac{2}{\pi c_2})$.

For the convenience of the reader we give a brief description of our main results. Initially, in Propositions \ref{1profile} and \ref{2profile} below, we establish the profiles of a two-families of single-lobe states for \eqref{trav21} on $[-L,L]$ depending on  an {\it a priori} sign of the kink-shift, namely, $a<0$ or $a>0$ (we note that no necessarily all these profiles will be  the first component for a  single-lobe kink state $(\phi_1,\phi_2)\in D_Z$, such as is established in  Propositions \ref{exis1} and \ref{exis3} below).  In Proposition \ref{1profile} (for {\it a priori} $a<0$),  the following family of solutions  $ \phi_{1,k}$ for equation in  \eqref{trav21} will be constructed based on the subluminal periodic type solution for the sine-Gordon (see Figure \ref{fig4}) with a single-lobe profile on $[-L, L]$  given by,
 \begin{equation}\label{formula1a}
\phi_{1,k}(x)= 2\pi - \arccos\Big[-1+ 2k^2 \sn^2\Big(\frac{x}{c_1} +K(k); k\Big)\Big],\quad x\in [-L,L],  
\end{equation}
 where $\sn(\cdot; k)$ is the Jacobian elliptic function of snoidal-type and  the elliptic modulus $k$ satisfying $K(k)>\frac{L}{c_1}$, where $K$ is the Legendre's complete elliptic integral of the first type (see \eqref{K} and Byrd and Friedman \cite{ByFr71}).
 More precisely, for $\frac{L}{c_1}\leqq \frac{\pi}{2}$, we have the curve of solutions $k\in (0,1)\to \phi_{1,k}$, and for $\frac{L}{c_1}> \frac{\pi}{2}$,  the curve $k\in (k_0,1) \to \phi_{1,k}$, with  $K(k_0)=\frac{L}{c_1}$. The profile $\phi_{1,k}$ satisfies $\phi''_{1,k}(x)<0$ for all $x\in [-L,L]$ (see Figure \ref{fig6a}). In this way, in Proposition \ref{exis1} below we give a complete restrictions on the parameters $L, c_1, c_2, k$ and $Z$ for which there are (or there are not) single-lobe kink states $(\phi_{1,k}, \phi_{2,a(k)})\in D_Z$ with a first-component being $\phi_{1,k}$ in \eqref{formula1a}. Similarly, for $a>0$, in Proposition \ref{2profile} we establish also a family of solutions  $\phi_{1,k}$ for equation in  \eqref{trav21} (see \eqref{La2}) but they are not strictly concave (see Figure \ref{fig7}). in Proposition \ref{exis3},  we also give sufficient conditions on  $L, c_1, c_2, k$ and $Z\leqq 0$ for the subluminal profile in \eqref{La2} to be the first-component of a single-lobe kink state. In Proposition  \ref{3profile} , we establish the existence of degenerated single-lobe kink state $(\pi, \phi_{2,0})\in D_Z$ if and only if $Z=\frac{2}{\pi c_2}$ (see Figure \ref{fig8}).

\begin{figure}[h]
\centering
\includegraphics[angle=0,scale=0.55]{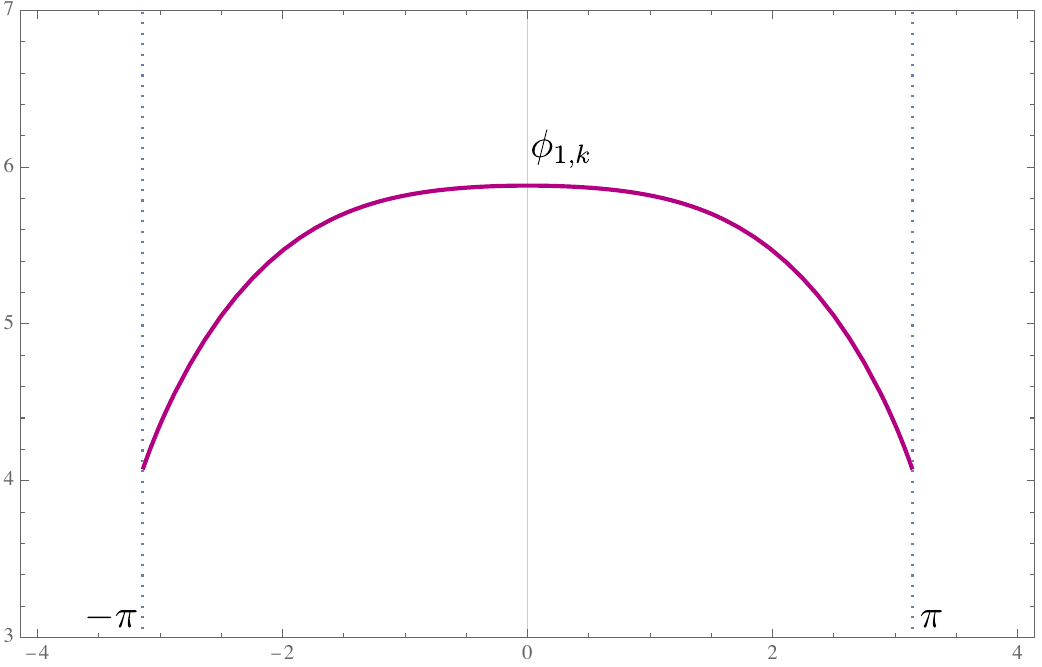}
\caption{\small{Single-lobe solution  $\phi_{1, k}$ in \eqref{formula1a} with $L=\pi$, $c_1=1$ and $k^2 = 0.98$.} \label{fig6a}} 
\end{figure}

The second goal of the present analysis is associated with the dynamics of the sine-Gordon flow around the  single-lobe kink states obtained in Propositions  \ref{3profile}, \ref{exis1} and  \ref{exis3}; more specifically, we are interested on its stability properties. The stability of these static configurations is an important property from both the mathematical and the physical points of view. Stability can predict whether a particular state can be observed in experiments or not. Unstable configurations are rapidly dominated by dispersion, drift, or by other interactions depending on the dynamics, and they are practically undetectable in applications. 

In the forthcoming stability analysis, for instance, by considering  the single-lobe kink states $(\phi_1, \phi_2)\equiv (\phi_{1,k}, \phi_{2,a(k)})\in D_Z$ determined by Proposition \ref{exis1} (with $\phi_{1,k}$ in \eqref{formula1a}), the following  family of linearized operators around of $(\phi_1, \phi_2)$ plays a fundamental role,
\begin{equation}\label{trav23}
\mathcal{L}_Z \bold{v}=\Big (\Big(-c_j^2\frac{d^2}{dx^2}v_j + \cos (\phi_j)v_j
\Big)\delta_{j,k} \Big ),\quad \l1\leqq j, k\leqq 2,\;\;\bold{v}= (v_j)_{j=1}^2,\;\; c_j>0,
\end{equation}
where $\delta_{j,k}$ denotes the Kronecker symbol, and defined on domains with $\delta$-type interaction at the vertex $\nu = L$, $D(\mathcal{L}_Z)=D_Z$ in \eqref{Domain0}. An interesting characteristic of the spectrum structure associated with operators in \eqref{trav23} on a tadpole graph is that they have a non-trivial Morse index (larger than or equal to 1, in general) which makes the stability study not so immediate. Here we use a novel linear instability criterion for stationary solutions of evolution models on metric graphs (see Theorem \ref{crit} below)  developed by Angulo and Cavalcante \cite{AC1}  and adapted to the case of the sine-Gordon model on tadpole graphs, as well as a {\it {splitting eigenvalue method}} introduced by Angulo \cite{Angtad} (see also Lemma \ref{split} below and Angulo \cite{Angloop}). This method is applied to the operator $\mathcal{L}_Z \equiv \diag (\mathcal{L}_1, \mathcal{L}_a)$ on $D_Z$. More precisely, we reduce the eigenvalue problem associated to $\mathcal{L}_Z$ in \eqref{trav23}  into two classes of eigenvalue problems, one for  $\mathcal{L}_1$ with periodic boundary conditions on $[-L, L]$ and the second one for $\mathcal{L}_a$ with $\delta$-type boundary condition on the half-line $[L, \infty)$. Hence, under the restrictions on $L, c_1, c_2, Z$ and $k$ given in Theorems \ref{exemplos}, \ref{ker} and \ref{MK}, we show that $\ker(\mathcal{L}_Z )=\{\bold{0}\}$ and that the Morse index of $\mathcal{L}_Z$ is equal to one. Therefore, as a  consequence of the linear instability criterion in Theorem \ref{crit} we deduce that  these single-lobe kink states are linearly unstable by the sine-Gordon flow on a tadpole graph (see our main instability result in Theorem \ref{unstable} below).

Our last main result is that the linear instability of the  single-lobe kink states  implies also its nonlinear instability (Theorem \ref{unstable} below). The strategy to prove this connection  is to use the fact that the mapping data-solution associated to the sine-Gordon model is of class $C^2$ on the energy space $\mathcal E(\mathcal G)\times L^2(\mathcal G)$, with
 \begin{equation}\label{senergy}
 \mathcal E(\mathcal G)=\{(f,g)\in H^1(\mathcal G): f(-L)=f(L)=g(L)\},
\end{equation}
$ \mathcal E(\mathcal G)$ will be also called the {\it closed continuous subspace at the vertex $\nu=L$ of $H^1(\mathcal G)$}. 
In this way, to carry out this strategy we  establish a complete theory on the local well-posedness problem for the sine-Gordon model on $\mathcal E(\mathcal G)\times L^2(\mathcal G)$ in Section \ref{seccauchy}.

For completeness and for the convenience of the reader, in the Appendix  we formulate some tools from the extension theory of symmetric operators by Krein and von Neumann which are suitable for our needs. In particular, we establish a Perron-Frobenius result which is used in the accurate Morse index estimate $n(\mathcal L_Z)\leqq 1$ for $\mathcal L_Z$ in \eqref{trav23} and $(\phi_1, \phi_2)$  being {\it a priori}  single-lobe kink state (see Theorem \ref{morse1}).

\subsection*{Notation}

Let $A$ be a  closed densely defined symmetric operator in a Hilbert space $H$. The domain of $A$ is denoted by $D(A)$. The deficiency indices of $A$ are denoted by  $n_\pm(A):=\dim  \ker (A^*\mp iI)$, with $A^*$ denoting the adjoint operator of $A$.  The number of negative eigenvalues counting multiplicities (or Morse index) of $A$ is denoted by  $n(A)$. For any $-\infty\leq a<b\leq\infty$, we denote by $L^2(a,b)$ the Hilbert space equipped with the inner product $(u,v)=\int_a^b u(x)\overline{v(x)}dx$. By $H^m(a,b)$  we denote the classical  Sobolev spaces on $(a,b)\subseteq \mathbb R$ with the usual norm.   We denote by $\mathcal{G}$ the tadpole graph parametrized by the edges $e_1 = [-L, L]$, $e_2 = [L,\infty)$,  attached to a common vertex $\nu=L$. On the graph $\mathcal{G}$ we define the classical spaces $L^p(\mathcal{G})=L^p(-L, L)  \oplus L^p(L, \infty)$, $p>1$, and 
  \begin{equation*}  
 \quad H^m(\mathcal{G})=H^m(-L, L) \oplus H^m(L, \infty), \quad m\in \mathbb N,
 \end{equation*}   
with the natural norms. Also, for $\mathbf{u}= (u_j)_{j=1}^2$, $\mathbf{v}= (v_j)_{j=1}^2 \in L^2(\mathcal{G})$, the inner product is defined by
$$
\langle \mathbf{u}, \mathbf{v}\rangle= \int_{-L}^L u_1(x) \overline{v_1(x)} \, dx  +  \int_L^{\infty}  u_2(x) \overline{v_2(x)} \, dx.
$$
Depending on the context we will use the following notations for different objects. By $\|\cdot \|$ we denote  the norm in $L^2(\mathbb{R})$ or in $L^2(\mathcal{G})$. By $\| \cdot\| _p$ we denote  the norm in $L^p(\mathbb{R})$ or in $L^p(\mathcal{G})$.

\vskip0.2in

 \section{Well-posedness for the sine-Gordon on a tadpole graph}
 \label{seccauchy}
  
  In this section we show the local well-posedness for the sine-Gordon model \eqref{sg2} on a  tadpole graph $\cG$ in the energy-space $\mathcal E(\cG)\times L^2(\cG)$, with
 \begin{equation}\label{senergy}
 \mathcal{E}(\cG)=\{(f,g)\in H^1(\cG): f(-L)=f(L)=g(L)\}.
\end{equation}
The space $\mathcal E(\mathcal G)$ represents the {\it closed continuous subspace at the vertex $\nu=L$ of $H^1(\mathcal G)$}.

 Next,  if we denote a  wave function $\bold U$ on   the tadpole graph $\mathcal G$  by $\bold{W}=(\boldsymbol{u}, \boldsymbol{v})^\top$  with $\boldsymbol{u}=(u_1, u_2)^\top$,  $\boldsymbol{v}=(v_1, v_2)^\top$, we get the following matrix-form for \eqref{sg2},
\begin{equation}
\label{sg3}
\bold {W}_t=JE \bold {W}+ F(\bold {W})
\end{equation}
with
\begin{equation}
\label{sg4}
J=\left(\begin{array}{cc}0 & I_2 \\-I_2 & 0\end{array}\right), \; E=\left(\begin{array}{cc}\mathcal F & 0  \\0 & I_2\end{array}\right),\; F(\bold {W})=\left(\begin{array}{c}0 \\0 \\- \sin u_1 \\- \sin u_2 \end{array}\right),
\end{equation}
$I_2$ is the identity matrix of order 2 and $\mathcal F$ is the diagonal matrix,
\begin{equation}
\label{sg5}
\mathcal F=\Big (\Big(-c_j^2\frac{d^2}{dx^2}\Big )\delta_{j,k}\Big),\quad 1\leqq j,k\leqq 2.
\end{equation}

The main result of this section can be stated as follows.
\begin{theorem}
\label{IVP}   
The Cauchy problem associated to the sine-Gordon model \eqref{sg3} on a tadpole graph $\cG$ is locally well-posed on  the energy space $\mathcal{E}(\cG)\times L^2(\cG)$. More precisely, for any $\boldsymbol{\Psi} \in \cE(\cG)\times L^2(\cG)$ there exists $T > 0$ such that the sine-Gordon system \eqref{sg3} has a unique solution,
\[
\boldsymbol{W} \in C([0,T]; \cE(\cG)\times L^2(\cG)),
\]  
with initial condition $\boldsymbol{W}(0) = \boldsymbol{\Psi}$. Furthermore, for every $T_0 \in (0,T)$ the data-solution map,
\[
\boldsymbol{\Psi} \in \cE(\cG)\times L^2(\cG) \mapsto \boldsymbol{W} \in C([0,T_0]; \cE(\cG)\times L^2(\cG))
\]
is, at least, of class $C^2$. 
\end{theorem}

The proof of Theorem \ref{IVP} follows the script established in Angulo and Plaza  \cite{AP1} (see Section 2)  in the case of the sine-Gordon model on a $\mathcal Y$-junction with a $\delta$-interaction. Thus, we start by describing the spectrum of the one parameter family of self-adjoint operators $(-\Delta, D_Z)$ for $Z \in \R$ defined in \eqref{Laplaciantadpole} and \eqref{Domain0} (for a related result in the context of tadpole graphs see Noja and Pelinovsky \cite{NoPel20}, Proposition A.1). In order to simplify the notation we write $\cF_Z$ to denote the operator $-\Delta$ acting on $L^2(\cG)$ with domain $D_Z$, for each $Z \in \R$.

\begin{theorem}
\label{spectrumDeltaZ}
For every $Z \in \R$ the essential spectrum of the self-adjoint operator $\cF_Z$ is purely absolutely continuous and $\sess(\cF_Z) = \sac(\cF_Z) = [0,\infty)$. If $Z > 0$ then $\cF_Z$ has exactly one negative eigenvalue, i.e., its point spectrum is $\ptsp(\cF_Z) = \{ - \varrho_{_Z}^2 \}$, where $\varrho_{_Z} > 0$ is the only positive root of the trascendental equation
\[
\varrho \left( \frac{2}{c_1} \tanh \Big( \frac{\varrho L}{c_1}\Big) + \frac{1}{c_2} \right) - Z = 0,
\]
and with eigenfunction
\[
\boldsymbol{\Phi}_Z = \begin{pmatrix}
\cosh (\frac{\varrho_{_Z} x}{c_1} ) \\
\\ e^{- \frac{\varrho_{_Z} x}{ c_2}}
\end{pmatrix}.
\]
If $Z \leqq 0$ then $\cF_Z$ has no point spectrum, $\ptsp(\cF_Z) = \varnothing$.
\end{theorem}
\begin{proof}
First, it is to be observed that from Theorem A.6 in \cite{Angtad} and Proposition \ref{semibounded}, the one parameter family of self-adjoint extensions $(\cF_Z, D_Z)$ satisfies $n(\cF_Z) \leqq 1$ for all $Z \in \R$.

Let us now consider the spectral problem $\cF_Z \mathbf{U} = \lambda \mathbf{U}$ with $\mathbf{U} = (u,v) \in D_Z$. In view of the geometry of the tadpole graph $\cG$ the spectrum of $\cF_Z$ is the union of two sets: the set $\lambda$ for which $v \equiv 0$, and the set of $\lambda$ for which $v \neq 0$. In the first case, we have the point spectral problem
\[
\left\{
\begin{aligned}
-c_1^2 u'' &= \lambda u, \qquad x \in (-L,L)\\
u(L) &= u(-L) = 0,\\
u'(L) &- u'(-L) = 0,
\end{aligned}
\right.
\]
regardless of the value of $Z \in \R$. The eigenvalues of this spectral problem are precisely $\lambda = n^2 \pi^2 c_1^2 / L^2$ with $n \in \N$, and for each $\lambda$ in this set the eigenfuction of $\cF_Z$ is given by
\[
\begin{aligned}
u(x) &= \sin \big( n \pi x /L\big), & x \in (-L,L),\\
v(x) &= 0, & x \in (L,\infty),
\end{aligned}
\]
for every $Z \in \R$.

The second set includes the absolutely continuous spectrum, $\sac(\cF_Z) = [0,\infty)$, and for each $\lambda = k^2 \in [0,\infty)$ the Jost functions of $\cF_Z$ are given by
\[
\begin{aligned}
u(x) &= a(k) \big( e^{ikx/c_1} + e ^{-ikx/c_1}\big), & x \in (-L,L),\\
v(x) &= e^{ikx/c_2} + b(k) e^{-ikx/c_2} , & x \in (L,\infty).
\end{aligned}
\]

The coefficients $a(k), b(k)$ can be obtained from the boundary conditions in \eqref{Domain0}, yielding the equation
\[
\mathbf{M} \begin{pmatrix} 
a(k) \\ b(k)
\end{pmatrix} = \mathbf{g}(k),
\]
where
\[
\mathbf{M} = \begin{pmatrix} 
2 \cos \big( k L / c_1 \big) & - e^{-ikL/c_2} \\
\frac{k}{c_1} \sin \big( k L / c_1 \big) & \big( \frac{ik}{c_2} - Z\big) e^{-ikL/c_2} 
\end{pmatrix}, \qquad \mathbf{g}(k) =  e^{ikL/c_2}\begin{pmatrix} 
1 \\ \big( \frac{ik}{c_2} + Z\big)
\end{pmatrix}.
\]
The determinant of $\mathbf{M}$ is
\[
| \mathbf{M} | = e^{ikL/c_2} \Big[ 2 \Big( \frac{ik}{c_2} - Z\Big) \cos \Big( \frac{kL}{c_1}\Big) +  \frac{k}{c_1} \sin \Big( \frac{kL}{c_1}\Big) \Big].
\]
A straightforward calculation shows that $| \mathbf{M} | = 0$ if and only if $k = 0$ and $Z = 0$ (yielding in that case the constant solutions $u = v= 2a(0)$). Hence for all $k \neq 0$ and all $Z$ the coefficients are given by
\[
\begin{pmatrix} 
a(k) \\ b(k)
\end{pmatrix} = \Big[ 2 \Big( \frac{ik}{c_2} - Z\Big) \cos \Big( \frac{kL}{c_1}\Big) +  \frac{k}{c_1} \sin \Big( \frac{kL}{c_1}\Big) \Big]^{-1} \begin{pmatrix} 
\frac{2ik}{c_2} e^{-ikL/c_2} \\  -\frac{k}{c_1} \sin \Big( \frac{kL}{c_1}\Big) + 2 \Big( \frac{ik}{c_2} + Z\Big) \cos \Big( \frac{kL}{c_1}\Big)
\end{pmatrix}.
\]
Since $a(k)$ and $b(k)$ are bounded and non-zero, there are no spectral singularities in the absolutely continuous spectrum of $\cF_Z$ in $L^2(\cG)$.

Finally, let us examine whether this second set includes isolated negative eigenvalues, $\lambda < 0$, with $v \neq 0$. In such a case, since $v \in H^2([L,\infty))$ implies $v \to 0$ as $x \to \infty$, we consider an eigenfunction of the form
\[
\begin{aligned}
u(x) &= \cosh (\sqrt{|\lambda|} x / c_1), & x \in (-L,L),\\
v(x) &= e^{- \sqrt{|\lambda|} x / c_2 }, & x \in (L,\infty),
\end{aligned}
\]
under the assumption $\lambda < 0$. The boundary conditions in \eqref{Domain0} yield the system
\[
\begin{aligned}
\cosh (\sqrt{|\lambda|} L / c_1) &= e^{- \sqrt{|\lambda|} L / c_2},\\
\frac{2\sqrt{|\lambda|}}{c_1} \sinh (\sqrt{|\lambda|} L / c_1) &= \Big( Z - \frac{\sqrt{|\lambda|}}{c_2}\Big) e^{- \sqrt{|\lambda|} L / c_2}.
\end{aligned}
\]
Upon substitution we obtain the trascendental equation
\[
\sqrt{|\lambda|} \left( \frac{2}{c_1} \tanh \Big( \frac{\sqrt{|\lambda|} L}{c_1}\Big) + \frac{1}{c_2} \right) = Z,
\]
which has a unique real positive solution $\sqrt{|\lambda|} > 0$ only in the case when $Z > 0$. Indeed, consider the function
\[
G(\varrho) := \varrho \left( \frac{2}{c_1} \tanh \Big( \frac{\varrho L}{c_1}\Big) + \frac{1}{c_2} \right) - Z, \qquad \varrho \in [0,\infty),
\]
assuming $Z > 0$. Clearly, $G(0) = -Z < 0$ and $G(\varrho) \to \infty$ as $\varrho \to \infty$. Moreover,
\[
G'(\varrho) = \frac{2}{c_1} \tanh \Big( \frac{\varrho L}{c_1}\Big) + \frac{1}{c_2} + \frac{2 \varrho L}{c_1^2} \frac{1}{\cosh^2 (\varrho L/c_1) } > 0,
\]
for all $\varrho \in (0,\infty)$ and, thus, $G$ is strictly increasing. By the intermediate value theorem there exists a unique value $\varrho_{_Z} > 0$ such that $G(\varrho_{_Z}) = 0$.

From the previous observation we conclude that if $Z > 0$ then $n(\cF_Z)=1$, whereas if $Z \leqq 0$ then $n(\cF_Z) = 0$. The theorem is proved.
\end{proof}

Let us now characterize the resolvent of the operator $\cA = JE$ (defined in \eqref{sg4}) in $H^1(\cG) \times L^2(\cG)$. For that purpose we use the description of the spectrum of $\cF_Z$ from Theorem \ref{spectrumDeltaZ}.

\begin{theorem}
\label{resolvA}
Let $Z \in \R$. For any $\lambda \in \C$ such that $-\lambda^2 \in \rho(\cF_Z)$ we have that $\lambda$ belongs to the resolvent set of $\cA = JE$, with $D(\cA) = D_Z \times L^2(\cG)$; moreover, the resolvent, 
\[
R(\lambda : \cA) = (\lambda I_4 - \cA)^{-1} \, : \, H^1(\cG) \times L^2(\cG) \to D(\cA),
\]
has the following representation for $\mathbf{U} = (\boldsymbol{u}, \boldsymbol{v})$,
\begin{equation}
\label{resoA}
R(\lambda : \cA) \mathbf{U} = \begin{pmatrix}
- R(-\lambda^2 : \cF_Z) (\boldsymbol{v} + \lambda \boldsymbol{u})\\
- \lambda R(-\lambda^2 : \cF_Z) (\boldsymbol{v} + \lambda \boldsymbol{u}) - \boldsymbol{u}
\end{pmatrix},
\end{equation}
where $R(-\lambda^2 : \cF_Z) = (-\lambda^2 I_2 - \cF_Z)^{-1} : L^2(\cG) \to D_Z$.
\end{theorem}
\begin{proof}
The proof is very similar to that of Theorem 2.2 in \cite{AP1} so we omit it.
\end{proof}

It is actually possible to provide some explicit formulae for $R(-\lambda^2 : \cF_Z)$ in the case where $-\lambda^2 \in \rho(\cF_Z)$ and $\lambda \in \R$. For the sake of completeness and for the reader's convenience we now establish those.

\begin{proposition}
\label{propresolvFZ}
Fix $Z \in \R$ and consider $-\lambda^2 \in \rho(\cF_Z)$ with $\lambda \in \R$. For any $\boldsymbol{u} = (u_1,u_2) \in L^2(\cG)$ we have:
\begin{itemize}
\item[(a)] In the case when $Z \leq 0$ and for all $\lambda > 0$ (without loss of generality) such that $-\lambda^2 \in \rho(\cF_Z)$, set $\Phi := (\cF_Z + \lambda^2 I_2)^{-1} \boldsymbol{u}$ to obtain
\begin{equation}
\label{formularesolvZneg}
\begin{aligned}
\Phi_1(x) &= \frac{d_1}{c_1^2} \cosh( \lambda x / c_1) + \frac{1}{2c_1 \lambda} \int_{-L}^L u_1(y) \cosh(\lambda |x-y|/c_1) \, dy, & x \in (-L,L),\\
\Phi_2(x) &= \frac{d_2}{c_2^2} e^{-\lambda x /c_2} + \frac{1}{2c_2 \lambda} \int_{L}^\infty u_2(y) e^{-\lambda |x-y|/c_2} \, dy, & x \in (L,\infty).
\end{aligned}
\end{equation}
\item[(b)] In the case when $Z > 0$ and for all $\lambda > 0$ (without loss of generality) such that $-\lambda^2 \in \rho(\cF_Z)$ and with $-\lambda \neq \varrho_{_Z}^2$, set $\Psi:= (\cF_Z + \lambda^2 I_2)^{-1} \boldsymbol{u}$ to obtain
\begin{equation}
\label{formularesolvZpos}
\begin{aligned}
\Psi_1(x) &= \frac{\cosh(\varrho_z x / c_1)}{\lambda + \varrho_z^2} \big\langle u_1, \cosh(\varrho_z x / c_1)\big\rangle + \Phi_1(x), & x \in (-L,L),\\
\Psi_2(x) &=  \frac{e^{-\varrho_z x / c_2}}{\lambda + \varrho_z^2} \big\langle u_2, e^{-\varrho_z x / c_2}\big\rangle + \Phi_2(x),& x \in (L,\infty).
\end{aligned}
\end{equation}
\end{itemize}
The constants $d_j = d_j(\lambda, \Phi_j)$, $j =1,2$, are chosen such that $\Phi \in D_Z$.
\end{proposition}
\begin{proof}
Formulae \eqref{formularesolvZneg} and \eqref{formularesolvZpos} result from the variation of constants formula: a solution to the homogeneous problem plus a convolution with the Green's function associated to the operator (see formula \eqref{formularesolvZneg} when $Z \leqq 0$). In the case when there exists a negative eigenvalue (that is, when $Z > 0$) one also needs to project onto the associated eigenspace (formula \eqref{formularesolvZpos}). The reader can directly verify that $(\cF_Z + \lambda^2 I_2)\Phi = \boldsymbol{u}$ if $Z \leqq 0$, and that $(\cF_Z + \lambda^2 I_2)\Psi = \boldsymbol{u}$ if $Z > 0$, whenever $-\lambda^2 \in \rho(\cF_Z)$.

\end{proof}

Next, from Theorem \ref{spectrumDeltaZ} we can define the following equivalent $X^1_Z$-norm to $H^1(\cG)$, for $\boldsymbol{v} =(v_j)_{j=1}^2\in H^1(\cG)$
\begin{equation}\label{1norm}
\|\boldsymbol{v} \|_{X^1_Z}^2=\|\boldsymbol{v}'\|^2 _{L^2(\cG)} + (\beta+1)\|\boldsymbol{v} \|^2 _{L^2(\cG)} +Z|v_1(L)|^2,
\end{equation}
where $\beta \geqq 0$ is defined as,
\begin{equation}
\label{defofbeta}
\beta := \begin{cases}
\varrho_{_Z}^2, & \text{if } Z > 0,\\
0, & \text{if } Z \leqq 0.
\end{cases}
\end{equation}

We will denote by $H^1_Z(\cG)$ the space $H^1(\cG)$ with the norm $\|\cdot\|_{X^1_Z}$. Moreover, for $\boldsymbol{u} =(u_j)_{j=1}^2$, the following well-defined inner product in $H^1_Z(\cG)$, 
\begin{equation}
\label{1inner}
\langle \boldsymbol{u} , \boldsymbol{v} \rangle_{1,Z}= \int_{-L}^L u'_1(x)\overline{v'_1(x)}dx +  \int_L^{\infty} u'_2(x) \overline{v'_2(x)}dx +(\beta +1) \langle \boldsymbol{u} , \boldsymbol{v} \rangle + Zu_1(L)\overline{v_1(L)},
\end{equation}
induces the $X^1_Z$-norm above (here we are considering $c_j^2 = 1$, $j=1,2$, without loss of generality).

The following theorem verifies that the operator $\mathcal A = JE$ is indeed the infinitesimal generator of a $C_0$-semigroup upon application of the classical Lumer-Phillips theory (see Pazy \cite{Pa}). 

\begin{theorem}
\label{cauchy1}
Let $Z\in \mathbb R$ and consider the linear operators $J$ and $E$ defined in \eqref{sg4}. Then, $\mathcal A = JE$ with $D(\mathcal A)= D(\mathcal F_Z)\times \mathcal E(\mathcal{\cG})$ is the infinitesimal  generator of a $C_0$-semigroup $\{\mathcal{S}(t)\}_{t\geqq 0}$ on $H^1(\mathcal{\cG})\times L^2(\mathcal{\cG})$. The initial value problem 
\begin{equation}
\label{LW1}
\begin{cases}
\boldsymbol{w} _{t}=\mathcal A\boldsymbol{w},  \\
\boldsymbol{w} (0)=\boldsymbol{w} _0\in D(\mathcal A)=D(\mathcal F_Z)\times \mathcal E(\mathcal{\cG}),
\end{cases}
\end{equation}
has a unique solution $\boldsymbol{w} \in C([0, \infty): D(\mathcal A))\cap C^1([0, \infty): H^1(\mathcal{\cG})\times L^2(\mathcal{\cG}))$ given by $\boldsymbol{w} (t)=\mathcal{S}(t)\boldsymbol{w} _0$, $t\geqq 0$. Moreover, for any $\boldsymbol{u}\in H^1(\mathcal{\cG})\times L^2(\mathcal{\cG})$ and $\theta>\beta+1$, where $\beta \geqq 0$ is given in \eqref{defofbeta}, we have the representation formula
\begin{equation}\label{FRW}
\mathcal{S}(t)\boldsymbol{u} = \frac{1}{2\pi i}\int_{\theta-i\infty}^{\theta+i\infty} e^{\lambda t} R(\lambda: \mathcal A) \boldsymbol{u} \, d\lambda
\end{equation}
where $\lambda\in \rho(\mathcal A)$ with $\RE \lambda = \theta$ and $R(\lambda: \mathcal A)=(\lambda I-\mathcal A)^{-1}$, and for every $\delta>0$, the integral converges uniformly in $t$ for every $t\in [\delta, 1/\delta]$.
\end{theorem}
\begin{proof} By considering the Hilbert space $X_Z\equiv H^1_Z(\mathcal{\cG})\times L^2(\mathcal{\cG})$ with inner product $\langle \cdot, \cdot\rangle_{X_Z}=\langle \cdot, \cdot\rangle_{1,Z} + \langle \cdot, \cdot\rangle$, with $\langle \cdot, \cdot\rangle_{1,Z}$ defined in \eqref{1inner}, we can follow the strategy of proof of Theorem 2.5 in Angulo and Plaza \cite{AP1} and from standard semigroup theory and properties of the Laplace transform (see Pazy \cite{Pa}) we obtain the result.

\end{proof}

Next result simply states the (expected) invariance property of the energy space under the action of the semigroup.
\begin{proposition}
\label{preser}
The subspace $\mathcal E(\cG)\times L^2(\cG)$ is invariant under the semigroup $\{\cS(t)\}_{t\geqq 0}$ defined by formula \eqref{FRW}. Moreover, $\cS(t)(\mathcal E(\cG)\times L^2(\cG))\subset \mathcal E(\cG)\times \mathcal C(\cG)$, $t>0$, where
\begin{equation}\label{0continuity}
\mathcal C(\cG)=\{(v_j)_{j=1}^2 \in  L^2(\cG): v_1(L)=v_2(L)\}.
\end{equation}
\end{proposition}
\begin{proof} By the representation of $\cS(t)$ in \eqref{FRW} it suffices to show that the resolvent operator $R(\lambda:\mathcal A)\Phi\in \mathcal E(\cG)\times L^2(\cG)$ for $\Phi\in \mathcal E(\cG)\times L^2(\cG)$. Indeed, for $\Psi=(\bold u, \bold v)$ we have from formula \eqref{resoA} that $R(-\lambda^2: \mathcal F_Z)(\lambda \bold u+\bold v)\in D(\mathcal F_Z)\subset \mathcal E(\cG)$ and so $R(\lambda:\mathcal A)\Phi\in \mathcal E(\cG)\times \mathcal E(\cG)\subset \mathcal E(\cG)\times \mathcal C(\cG)\subset \mathcal E(\cG)\times L^2(\cG)$.
\end{proof}

We are now ready to prove Theorem \ref{IVP}.

\begin{proof}[Proof of Theorem \ref{IVP}]
Based on Theorem \ref{cauchy1} and Proposition \ref{preser}, the local well-posedness result in $ \mathcal E(\cG)\times  L^2(\cG)$ follows from the Banach fixed point theorem and standard arguments. It is to be noticed that the contraction mapping principle has the advantage that, since the nonlinear term $F(\mathbf W)$ is smooth, then the Implicit Function Theorem implies the smoothness of the data-solution map associated to the sine-Gordon equation (see \cite{AngGol18b,AP1,AP2,AP3} for further information). Lastly, from Proposition \ref{preser}  we obtain that for all $t\in (0, T]$, $\boldsymbol{w}(t) \in \mathcal E(\cG)\times  \mathcal C(\mathcal Y) $. This finishes the proof.
\end{proof}

\section{Linear instability criterion for the sine-Gordon model on a tadpole graph}
\label{seccriterium}

 In this section we establish a linear instability criterion of stationary solutions for the sine-Gordon model \eqref{sg3} on a tadpole graph $\mathcal G$. It is based on the results by Angulo and Cavalcante in \cite{AC, AC1} for stationary solutions of evolution models on metric graphs. By the sake of completeness we establish the main points of our criterion. More importantly, the criterion is so general that  it also applies to any type of stationary solutions independently of the boundary conditions under consideration and can be therefore used to study configurations with boundary rules at the vertex of $\delta'$-interaction type, or with other types of stationary solutions to the sine-Gordon equation.

Let us suppose that $JE$ in \eqref{sg4} on a domain $D(JE)\subset L^2(\mathcal G)$ is the infinitesimal generator of a $C_0$-semigroup on $L^2(\mathcal G)$ and we consider {\it a priori} a  stationary solution $\Phi=(\phi_1(x), \phi_2(x),0,0)$ for the sine-Gordon model on a tadpole graph such that $\Phi \in D(JE)$. Thus, every component satisfies the equation
\begin{equation}\label{statio}
-c^2_j \phi''_j + \sin \phi_j =0,\quad j=1,2.
\end{equation}
 Now, we suppose  that $\bold W$ satisfies formally  equality in \eqref{sg3} and we define the new variable
  \begin{equation}\label{stat3}
 \bold V \equiv \bold W-\Phi.
  \end{equation}
Then, from \eqref{statio} we obtain the following  linearized system for \eqref{sg3} around $\Phi$,
   \begin{equation}\label{stat4}
  \bold V_t=J\mathcal E\bold V,
 \end{equation} 
  with $\mathcal E$ being the $4\times 4$ diagonal-matrix $\mathcal E=\left(\begin{array}{cc} \mathcal L & 0 \\0 & I_2\end{array}\right)$, and 
 \begin{equation}\label{stat5}
\mathcal{L} =\Big (\Big(-c_j^2\frac{d^2}{dx^2}+ \cos \phi_j
\Big)\delta_{j,k} \Big ),\qquad 1\leqq j, k\leqq 2.
\end{equation}

We point out the equality $J\mathcal E=J E+\mathcal{T}$, with 
$$
\mathcal{T} = \left(\begin{array}{cc} 0& 0 \\ \big( - \cos(\phi_j) \, \delta_{j,k} \big)& 0\end{array}\right)
$$
 being a \emph{bounded} operator on $H^1(\mathcal{G})\times L^2(\mathcal{G})$. This  implies that $J\mathcal E$ also generates a  $C_0$-semigroup on  $H^1(\mathcal{G})\times L^2(\mathcal{G})$ (see Pazy \cite{Pa}).
  
In the sequel, our objective is to provide sufficient conditions for the trivial solution $\bold V \equiv 0$ to be unstable by the linear flow \eqref{stat4}. More precisely, we are interested in finding a {\it growing mode solution} of \eqref{stat4} with the form $ \bold V=e^{\lambda t} \Psi$ and $\RE \lambda >0$. In other words, we need to solve the formal system 
 \begin{equation}\label{stat6}
 J\mathcal E \Psi=\lambda \Psi,
\end{equation}
with $\Psi\in D(J\mathcal E)$.  If we denote by $\sigma(J\mathcal E)$ the spectrum  of $J\mathcal E$, it is well known that we can divide this into two disjoint parts $\sigma(J\mathcal E)= \sigma_{\mathrm{pt}}(J\mathcal E)\cup \sigma_{\mathrm{ess}}(J\mathcal E)$, where $\lambda \in  \sigma_{\mathrm{pt}}(J\mathcal E)$ if $\lambda$ is isolated and with finite multiplicity, and $\sigma_{\mathrm{ess}}(J\mathcal E)$ is the essential spectrum for $J\mathcal E$ (see \cite{RS4}). The later discussion suggests the usefulness of the following definition.
\begin{definition}
The stationary vector solution $\Phi \in D(\mathcal E)$    is said to be \textit{spectrally stable} for model sine-Gordon if the spectrum of $J\mathcal E$, $\sigma(J\mathcal E)$, satisfies $\sigma(J\mathcal E)\subset i\mathbb{R}.$
Otherwise, the stationary solution $\Phi\in D(\mathcal E)$   is said to be \textit{spectrally unstable}.
\end{definition}

It is standard to show that $ \sigma_{\mathrm{pt}}(J\mathcal E)$ is symmetric with respect to both the real and imaginary axes and that $ \sigma_{\mathrm{ess}}(J\mathcal E)\subset i\mathbb{R}$, by supposing $J$ skew-symmetric and $\mathcal E$ self-adjoint   (for instance, under Assumption $(S_3)$ below for $\mathcal L$; see \cite[Lemma 5.6 and Theorem 5.8]{GrilSha90}). These cases on $J$ and  $\mathcal E$ will be considered in our theory. Hence  it is equivalent to say that $\Phi\in D(J\mathcal E)$ is  \textit{spectrally stable} if $ \sigma_{\mathrm{pt}}(J\mathcal E)\subset i\mathbb{R}$, and it is spectrally unstable if $ \sigma_{\mathrm{pt}}(J\mathcal E)$ contains a point $\lambda$ with $\RE \lambda>0$.
 
 Next, we establish our theoretical framework and the main assumptions in order to obtain a nontrivial solution to problem \eqref{stat6}:
 \begin{enumerate}
 \item[($S_1$)]  $J\mathcal E$ is the generator of a $C_0$-semigroup $\{S(t)\}_{t\geqq 0}$.  
 \item[($S_2$)] Let $\mathcal L$ be the matrix-operator in \eqref{stat5}  defined on a domain $D(\mathcal L)\subset L^2(\mathcal{G})$ on which $\mathcal L$ is self-adjoint.
 \item[($S_3$)] Suppose $\mathcal L:D(\mathcal L)\to L^2(\mathcal{G})$ is  invertible  with Morse index $n(\mathcal L)=1$ and such that $\sigma(\mathcal L)=\{\lambda_0\}\cup J_0$ with $J_0\subset [r_0, \infty)$, for $r_0>0$ and $\lambda_0<0$.
 \end{enumerate} 
 
Our linear instability criterion of stationary solutions for the sine-Gordon on a tadpole graph  is the following.
\begin{theorem}\label{crit}
Suppose the assumptions $(S_1)$ - $(S_3)$ hold.  Then the operator $J\mathcal E$ has a real positive and a real negative eigenvalue. 
\end{theorem}

The proof of Theorem \ref{crit} can be seen in \cite{AP1} -Theorem3.2.

It is widely known  that the spectral instability of a specific traveling wave solution of an evolution type model is  a key prerequisite to show their nonlinear instability property (see \cite{GrilSha90} and references therein). Thus we have the following definition.
\begin{definition}\label{nonstab}
A stationary vector solution $\Phi \in D(\mathcal E)$ is said to be \textit{nonlinearly unstable} in $X\equiv H^1(\mathcal G)\times L^2(\mathcal G)$-norm for model sine-Gordon \eqref{sg3} if there is $\epsilon>0$, such that for every $\delta>0$ there exist an initial datum $\bold w_0$ with $\|\Phi -\bold w_0\|_X<\delta$ and an instant $t_0=t_0(\bold w_0)$ such that $\|\bold W(t_0)-\Phi \|_X>\epsilon$, where $\bold W=\bold W(t)$ is the solution of the sine-Gordon model with $\bold W(0)=\bold w_0$.
\end{definition}

\section{Existence  of single-lobe kink state  for the sine-Gordon equation on a tadpole graph}
\label{secInst}

In this section we  show the existence  of static solutions of  single-lobe kink type for the sine-Gordon model determined by a $\delta$-interaction type at the vertex  $\nu=L$ of a tadpole graph (see \eqref{Domain0}). We will find that these profiles exist  as long as we have several restrictions involved  the {\it{a priori}} parameters $c_i$ and $Z$. We start with the strength value $Z$.

\subsection{{\it A priori} restrictions on the parameter $Z$} 

We consider a single-lobe kink state $(\phi_1, \phi_2)$  for  \eqref{trav21}. Then, as $\phi_1( L)=\phi_2 (L)>0$ follows that $\phi_1$ satisfies the following condition, 
 \begin{equation}\label{RCBC}
  \phi'_1( L)-  \phi'_1 (-L)=\Big[ \frac{\phi'_2(L)}{\phi_2(L)} +Z\Big]  \phi_1(L)\equiv \alpha_{_{Z}} \phi_1(L).
 \end{equation}
Next, from the fact that $\phi_1$ is even we have $2\phi_1'(L)=\alpha_{Z} \phi_1(L)$, and from $\phi'_1(L)<0$ we get $\alpha_{Z}<0$. Now, from \eqref{trav22} we obtain for $y=e^ {-a/c_2}$
 \begin{equation}\label{Z-condit}
-Z>\frac{\phi'_2(L)}{\phi_2(L)}=-\frac{1}{c_2} \frac{y}{(1+y^ 2)\arctan(y)}\equiv f(y),\;\;\quad y>0.
 \end{equation}
 Thus, we have the following proposition.
 
 \begin{proposition}
 \label{Z2}
 We consider a single-lobe kink state $(\phi_1, \phi_2)\in D_Z$  for  \eqref{trav21}.  Then, the strength value $Z$ satisfies $Z\in (-\infty, \frac{2}{\pi c_2})$.
 \end{proposition}
 \begin{proof}   From \eqref{Z-condit} we get initially that $Z<\frac{1}{c_2} $ because $f(y)>-\frac{1}{c_2}$ for all $y>0$. Now, for $a\leqq 0$ we have $y=e^ {-a/c_2}\geqq 1$. Then $-Z>f(y)\geqq f(1)=-\frac{2}{\pi c_2}$, and so $Z<\frac{2}{\pi c_2}$. For $a>0$, we have  $y<1$ and so $f(y)<f(1)$. Thus, for $a\to 0^+$, $f(y)\to f(1)$, and therefore $f(1)\leqq -Z$. But, for $a=0$, from  \eqref{Z-condit} follows $-Z>f(1)$.
 In conclusion, we need to have the restriction $Z\in (-\infty, \frac{2}{\pi c_2})$.
 \end{proof}  
 
 \subsection{Existence of positive single-lobe states (subluminal librations)} 

In the following we show the existence of positive solutions for the equation
\begin{equation}
\label{1lobe}
-c_1^2 \phi''_1(x) + \sin ( \phi_1(x))=0, \qquad x\in [-L,L],
\end{equation}
with a single-lobe on $[-L,L]$, namely,  $\phi_1$ is even, the maximum of $\phi_1$ is achieved at a single internal point of $[-L,L]$, namely, $x=0$, and $\phi'(x)<0$ on $(0,L]$ (see Figure \ref{fig3} and the phase-plane in Figure \ref{fig4}).
\begin{figure}[h]
 \centering
\includegraphics[angle=0,scale=0.5]{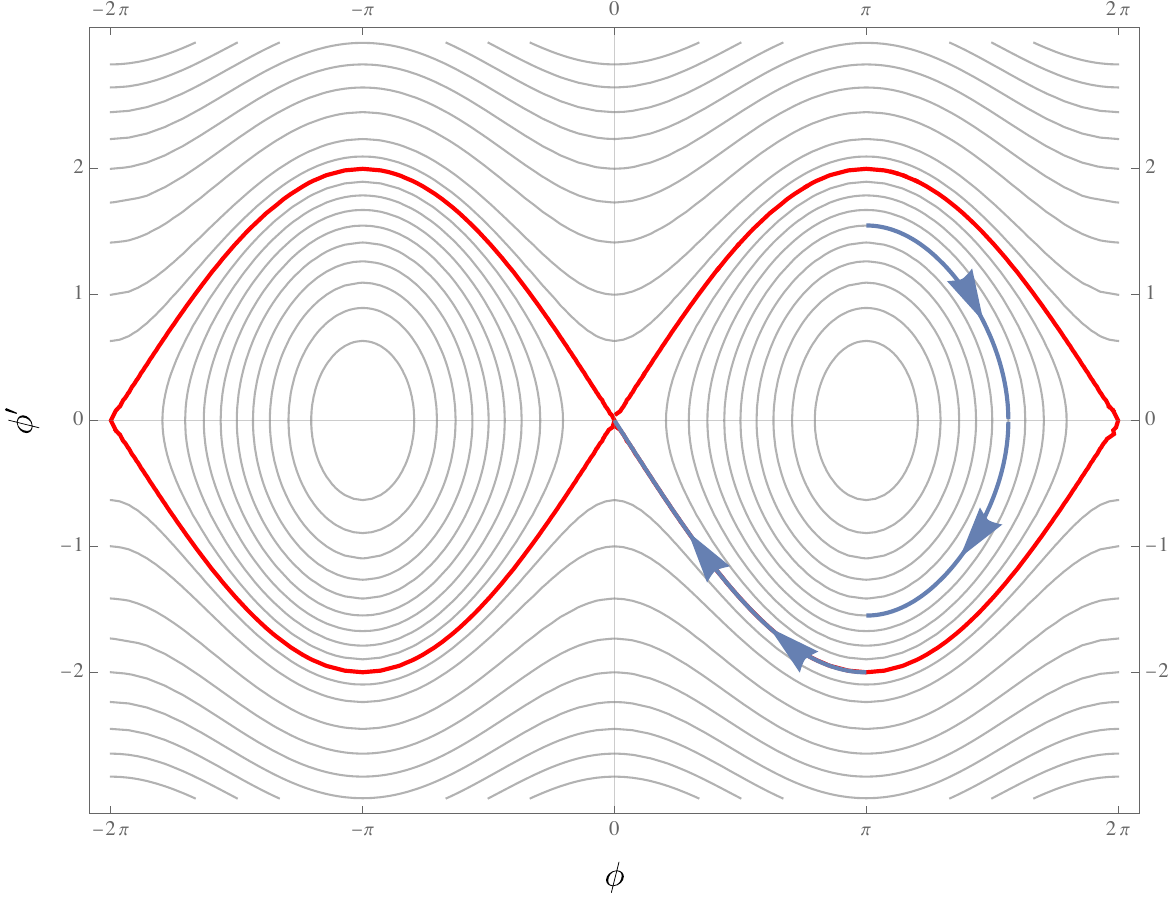}
\caption{\small{Phase-plane of a positive single-lobe kink state for \eqref{trav21}. Here we take $c_1 = c_2 = 1$ and the translation variable $\xi = x - ct$ with $c = 0$ (subluminal waves). The picture shows the phase plane of the solutions $\phi = \phi(\xi)$ for \eqref{trav21}. The separatrices are represented in red solid lines and the librations are the periodic waves inside the separatrices (in light gray). The single-lobe kink profile (depicted with blue thick arrow lines) is composed of a periodic libration, symmetric in $[-L,L]$ and connecting $\phi = \pi$ with itself, glued with a decaying kink connecting $\phi = \pi$ to $\phi = 0$ along a separatrix (color online).} \label{fig4}} 
\end{figure}

The solutions to \eqref{1lobe} of our interest are called {\it subluminal librations}, namely, the profiles inside the separatrix induced by the kink-profile $\phi_2$ in \eqref{trav22} (see Figure \ref{fig4}) which are associated to a traveling speed satisfying $c^2 < c_j^2$, where $c_j$ is the corresponding characteristic speed on each edge $e_j$, which are fixed (see, e.g., Jones \emph{et al.} \cite{JMMP2} for a complete description of subluminal librations). These profiles constitute one of the components of the periodic stationary profiles determined by the sine-Gordon model. Then, by considering the quadratic form associated to \eqref{1lobe} and by integrating by parts once we get from condition $\phi_1'(0)=0$,
\begin{equation}
\label{GF}
-\frac{c_1^2}{2} [\phi'_1(x)]^ 2 - \cos ( \phi_1(x))=- \cos ( \phi_1(0))\Longrightarrow  -\frac{c_1^2}{2} [\phi'_1(x)]^ 2+ 1- \cos ( \phi_1(x))=E, \;\;x\in [-L,L],
\end{equation}
with $E\equiv 1-\cos ( \phi_1(0))$ is so-called energy-level. Thus, being $\phi_1$ a single-lobe profile we need to have  $\phi_1(0)\in (\pi, 2\pi)$ and therefore  $E\in (0, 2)$.

Next, we consider the following change of variables of snoidal-type (see Figure \ref{fig5}),
\begin{equation}
\label{singleL}
 \cos ( \phi_1(x))=-1+\beta \sn^2(\lambda (x-x_0);k),
\end{equation}
where $\text{sn}(\cdot;k)$ represents the Jacobian elliptic function snoidal determined by the parameter $k\in (0,1)$ (called the elliptic modulus) and $x_0, \lambda, \beta \in \mathbb R$ to be determined. Thus, it is not difficult to see from \eqref{GF} that we arrive to the following equality for all $x\in [-L,L]$,
$$
2(E-2)\beta \sn^2+\beta^2[2\beta-(E-2)]\sn^4-\beta^3\sn^6=-2c_1^2\beta^2\lambda^2\sn^2+2\beta^2\lambda^2c_1^2(1+k^2)\sn^4-2\beta^2\lambda^2c_1^2k^2\sn^6,
$$
and therefore $\beta=2-E>0$, $\lambda=\frac{1}{c_1}$, and $k=\sqrt{\frac{2-E}{2}}$.

Now, since $\cos ( \phi_1(0))=1-E$ (see \eqref{GF}) we get  from  \eqref{singleL} that $(2-E)\sn^2(-\frac{x_0}{c_1};k)=2-E$ and therefore
$$
\sn^2(-\frac{x_0}{c_1};k)=1\;\;\text{implies}\;\; x_0=\pm c_1 (2n+1) K(k),\; n\geqq 0,
$$
where $K(k)$ is the complete elliptic integral of the first kind defined by 
\begin{equation}\label{K}
K(k)=\int_0^1 \frac{1}{ \sqrt{(1-t^2)(1-k^2t^2)}} dt,
\end{equation}
with $K'(k)>0$, $K(k)\to \frac{\pi}{2}$ as $k\to 0$ and $K(k)\to \infty$ as $k\to 1$ (see \cite{ByFr71}). We recall that   $\sn(\cdot;k)$ is periodic-odd of period $4K(k)$ and, thus, $\sn^2(\cdot;k)$ is periodic-even of period $2K(k)$. Hence we can choose any sign for defining $x_0$ and  $x_0=c_1K(k)$. Therefore, we get the formula for $\phi_1$ as
\begin{equation}\label{profile}
 \cos ( \phi_1(x))=-1+ 2k^2\sn^2\Big(\frac{x}{c_1} +K(k); k\Big),\quad k\in (0,1). 
 \end{equation}
We note immediately that $ \Phi(x)=\arccos\Big[-1+ 2k^2\sn^2\Big(\frac{x}{c_1} +K(k); k\Big)\Big]$ does not represent the first component of a single-lobe kink profile according to Definition \ref{singlelobe}  (see Figures \ref{anti} and \ref{fig3}). The stability properties of these $\Phi$-profiles and of the anti-kink profiles for the sine-Gordon equation will be studied in a future work.

\begin{figure}[h]
 \centering
\includegraphics[angle=0,scale=0.55]{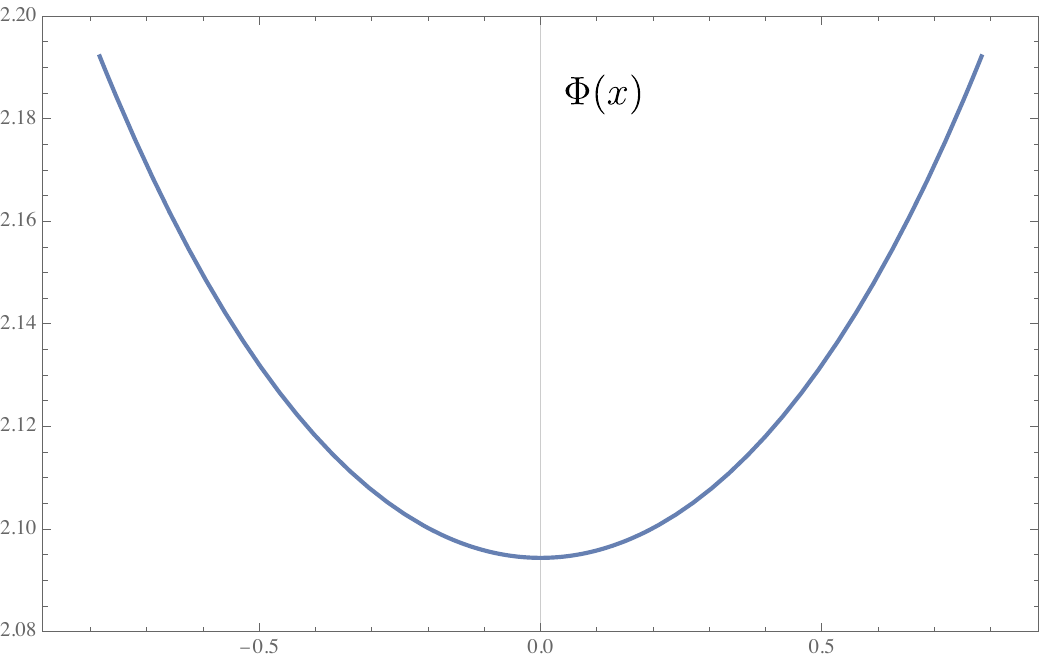}
\caption{\small{Profile for $\Phi = \Phi(x)$  with $x\in [-\frac{\pi}{4}, \frac{\pi}{4}]$, $c_1=1$,  and any $k\in (0,1)$.} \label{anti}} 
\end{figure}

\begin{figure}[h]
\centering
\includegraphics[angle=0,scale=0.55]{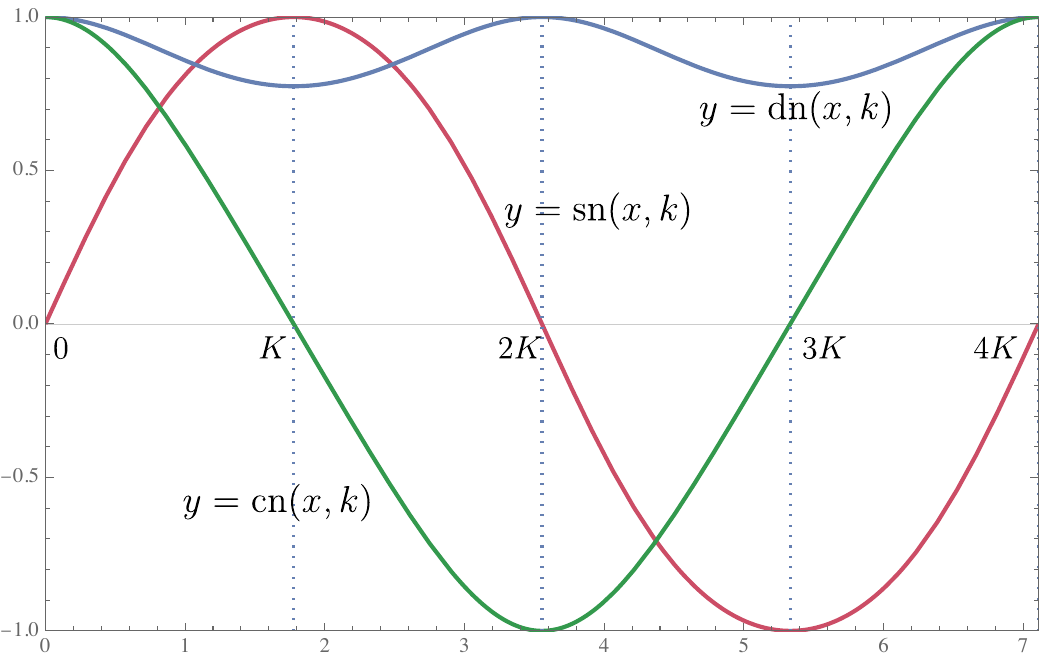}
\caption{\small{Graphs of the elliptic functions $\sn(x,k)$, $\cn(x,k)$ and $\dn(x,k)$ for $x \in [0,4K(k)]$, $k=0.4$.} \label{fig5}} 
\end{figure}

Before establishing our existence result of the single-lobe profile according to Definition \ref{singlelobe}, we obtain some relations between the modulus $k^2=\frac{2-E}{2}$ and the maximum value of {\it a priori} single-lobe $\phi_1$ (we recall $\phi_1(0)\in (\pi, 2\pi)$; see Figure \ref{fig4}):
\begin{enumerate}
\item[(i)] Let $k$ be fixed and  $k^2\leqq \frac12$: $E\in [1, 2)$, $\cos(\phi_1(0))=-1+2k^2\in (-1,0]$ if and only if $\phi_1(0)\in (\pi, \frac{3\pi}{2}]$.
\item[(ii)] For $k$ fixed and  $k^2>\frac12$: $E\in (0, 1)$, $\cos(\phi_1(0))=-1+2k^2\in (0,1)$ if and only if $\phi_1(0)\in (\frac{3\pi}{2}, 2\pi]$.
\end{enumerate}

In the sequel we obtain two explicit formulae for the profile $\phi_1$ in \eqref{profile} being an even single-lobe profile on $[-L,L]$ such that $\phi_1'(x)<0$ for $x\in (0, L)$ and $\phi_1(-L)=\phi_1(L)=\phi_{2,a}(L)$.  As we shall see, these formulae will depend on the position of  $\phi_{2,a}(L)$ with regard to $\pi$, where the shift-parameter $a$ determines the kink-soliton profile in \eqref{trav22} (see Figure \ref{fig3}).

\begin{proposition}
\label{1profile}
 Let  $L$ and $c_1$ be fixed positive constants. We consider {\it a priori} $a<0$ in \eqref{trav22} which determines the kink-soliton profile $\phi_{2,a}$ (therefore $\phi_{2,a}(L)>\pi$).  Suppose that $\phi_1$ satisfies equation in \eqref{1lobe} and $\phi_1(-L)=\phi_1(L)=\phi_{2,a}(L)$. Then, there is a family of solutions  $\phi_1\equiv \phi_{1,k}$ for equation in \eqref{1lobe} of subluminal type  with a single-lobe profile on $[-L, L]$ given by,
 \begin{equation}\label{formula1}
\phi_{1,k}(x)= 2\pi - \arccos\Big[-1+ 2k^2 \sn^2\Big(\frac{x}{c_1} +K(k); k\Big)\Big],\quad x\in [-L,L],  
\end{equation}
 with  modulus $k$ satisfying $K(k)>\frac{L}{c_1}$. More precisely, for $\frac{L}{c_1}\leqq \frac{\pi}{2}$, we have $k\in (0,1)$, and for $\frac{L}{c_1}> \frac{\pi}{2}$, we have $k\in (k_0,1)$, with  $K(k_0)=\frac{L}{c_1}$. The profile $\phi_{1,k}$ satisfies $\phi''_{1,k}(x)<0$ for all $x\in [-L,L]$.
 \end{proposition}
\begin{proof}
Let $a<0$.  From the boundary condition $\phi_1(L)=\phi_{2,a}(L)>\pi$, we have that for all $x\in [-L,L]$, $2\pi\geqq\phi_1(0)\geqq  \phi_1(x)\geqq \phi_1(L)>\pi$. Thus, from  \eqref{profile} we get the subluminal profile in \eqref{formula1}. 

Next, from  condition $\phi_1(x)>\pi$ on $[-L,L]$,  the property $\sn(2K(k);k)=0$ and \eqref{profile}, we need to have $\frac{x}{c_1} +K<2K$ for all $x\in [0, L]$
(see Figure \ref{fig6a} above). Therefore, {\it a priori} we get that  modulus $k$ satisfies $K(k)>\frac{L}{c_1}$.  Next, for $\frac{L}{c_1}\leqq \frac{\pi}{2}$ we obtain immediate a family of single-lobe solutions  $k\to\phi_{1, k}$ for any $k\in (0,1)$.  For $\frac{L}{c_1}> \frac{\pi}{2}$, we consider $k_0\in (0,1)$ such that $K(k_0)=\frac{L}{c_1}$, then we have the family of single-lobe profiles $k\in (k_0, 1)\to \phi_{1, k}$ defined in \eqref{formula1}. Lastly, from \eqref{1lobe} we get immediately $\phi''_{1,k}(x)<0$ for all $x\in [-L,L]$.
\end{proof}

\begin{remark} We consider $L=\pi$, $c_1=1$, then $K(k)=\pi$ for $k\approx 0.9843\equiv k_0$, and so for $k>k_0$ we get the family of subluminal solutions $k\to \phi_{1, k}$  in \eqref{formula1} with a single-lobe profile on $[-L,L]$. Figure \ref{fig6a} shows the profile of $\phi_{1, k}$ for $k^2= 0.98$. 
\end{remark}

\begin{proposition}\label{2profile}
 Let  $L, c_1$ be fixed positive constants. We consider {\it a priori} $a>0$ in \eqref{trav22} which determines the kink-soliton profile $\phi_{2,a}$ (with $\phi_{2,a}(L)<\pi$).  Suppose that $\phi_1$ satisfies equation in \eqref{1lobe} and $\phi_1(-L)=\phi_1(L)=\phi_{2,a}(L)$. Then, there is a family of solutions  $\phi_1\equiv \phi_{1,k}$ for equation in \eqref{1lobe} of subluminal type  with a single-lobe profile on $[-L, L]$  given by
 \begin{equation}\label{La2}
\phi_{1,k}(x)= \left\{ \begin{array}{ll}
2\pi-\arccos\Big[-1+ 2k^2\sn^2\Big(\frac{x}{c_1} +K; k\Big)\Big],\quad\quad x\in [0,b],\\
\arccos\Big[-1+ 2k^2\sn^2\Big(\frac{x}{c_1} +K; k\Big)\Big], \;\qquad\quad\quad x\in [b, L],
  \end{array} \right.
 \end{equation} 
 with $b=b(k)$ being  unique in $ (0, L)$ such that $\phi_{1,k}(b)=\pi$ (indeed, $b(k)=c_1K(k)$), and with  modulus $k$ satisfying $K(k)<\frac{L}{c_1}$. Therefore, we have,
 \begin{enumerate}
 \item[(a)]  for $\frac{L}{c_1}\leqq \frac{\pi}{2}$ there are no single-lobe profiles for \eqref{1lobe} with the profile in \eqref{La2},
 \item[(b)]  for $\frac{L}{c_1}> \frac{\pi}{2}$, $\phi_{1,k}$ in \eqref{La2} is a  single-lobe profile  for \eqref{1lobe} with $k\in (0, k_0)$, and $k_0$ satisfying $K(k_0)=\frac{L}{c_1}$; and,
  \item[(c)] $\phi''_{1,k}(b)=0$, $\phi''_{1,k}(x)<0$ for $x\in [-b,b]$ and $\phi''_{1,k}(x)>0$ for $x\in [-L, L]/[-b, b]$.
 \end{enumerate}  
   \end{proposition}
\begin{proof}
From $a>0$, $\phi_1(0)>\pi$ and the boundary condition $\phi_1(L)=\phi_{2,a}(L)<\pi$, we have a unique $b\in (0, L)$, $b=b(k)$, such that $\phi_1(b)=\pi$, and so for $x\in [0, b]$,  $\phi_1(x)\in [\pi, \phi_1(0)]$, and for $x\in [b, L]$,  $\phi_1(x)\in [\phi_1(L), \pi]$. Thus, from  \eqref{profile} we get the single-lobe  profile in \eqref{La2}. Next, from $\phi_{1,k}(b)=\pi$ and the strictly-decreasing property  of $\phi_{1, k}$ on $[0, L]$ we need to have  $b(k)=c_1K(k)$ (see Figure \ref{fig5} above). Therefore, $K(k)=\frac{b}{c_1}<\frac{L}{c_1}$. Moreover, from \eqref{1lobe} we obtain  $\phi''_1(b)=0$, $\phi''_1(x)<0$ for $x\in [0, b)$ and $\phi''_1(x)>0$ for $x\in [b, L]$ (see Figure \ref{fig7} below). 
\end{proof}

\begin{remark}
We consider $L=\pi$, $c_1=1$, then $K(k)=\pi$ for  $k\approx 0.984432\equiv k_0$,  and so for  $k\in (0, k_0)$ we get the family of subluminal librational solutions $k\to \phi_{1, k}$  in \eqref{La2} with a single-lobe  profile on $[-L,L]$ with $b=K(k)$. Figures \ref{fig8a} and \ref{fig9b}  show the parts of $\phi_{1, k}$ in \eqref{La2}  when they are considered on the complete interval  $[0, \pi]$ with $k=0.5$.  Figure \ref{fig7}  shows  the full profile of $\phi_{1, k}$ in \eqref{La2} with $k=0.5$, $b=K(0.5)$ and it which is based on gluing the two components for $\phi_{1, k}$ in \eqref{La2} from the profiles in Figures \ref{fig8a} and \ref{fig9b}. We note that the ``corners'' (namely, the points where the derivative does not exist) in Figures \ref{fig8a} and \ref{fig9b} occur simultaneously at $x= b=K(0.5)$.
\end{remark}

\begin{figure}[h]
\begin{center}
\subfigure[]{\label{fig8a}\includegraphics[scale=0.45]{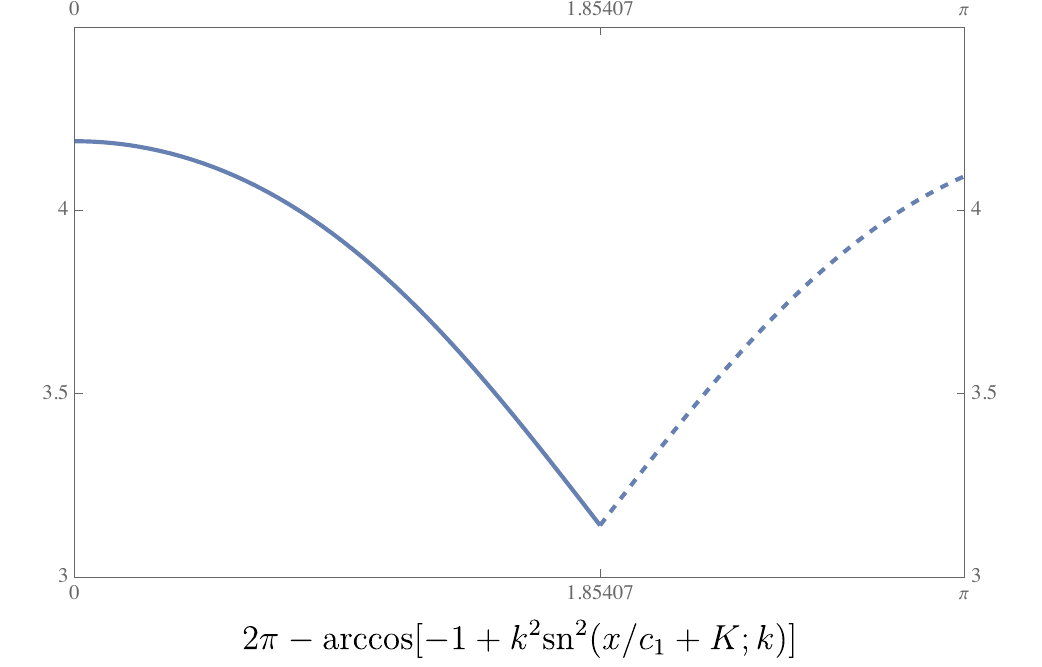}}
\subfigure[]{\label{fig9b}\includegraphics[scale=0.45]{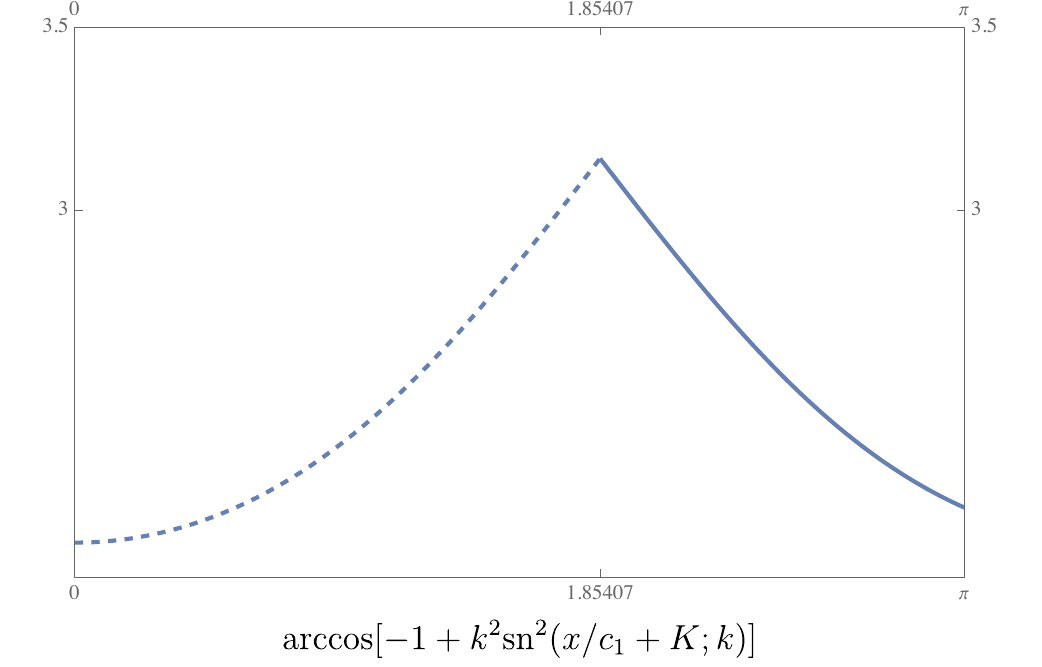}}
\caption{\small{Graph of $\phi_{1, k}$ in \eqref{La2} with $c_1=1$, $k=0.5$, $L=\pi$ and $x\in [0, \pi]$. Here $b \approx 1.85407 \in (0,\pi)$. Panel (a) shows the first part with solid blue line for $x \in [0,b]$. Panel (b) shows the second part with solid blue line for $x \in [b,\pi]$ (color online). }
}
\label{fig89}
\end{center}
\end{figure}


\begin{figure}[h]
 \centering
\includegraphics[angle=0,scale=0.55]{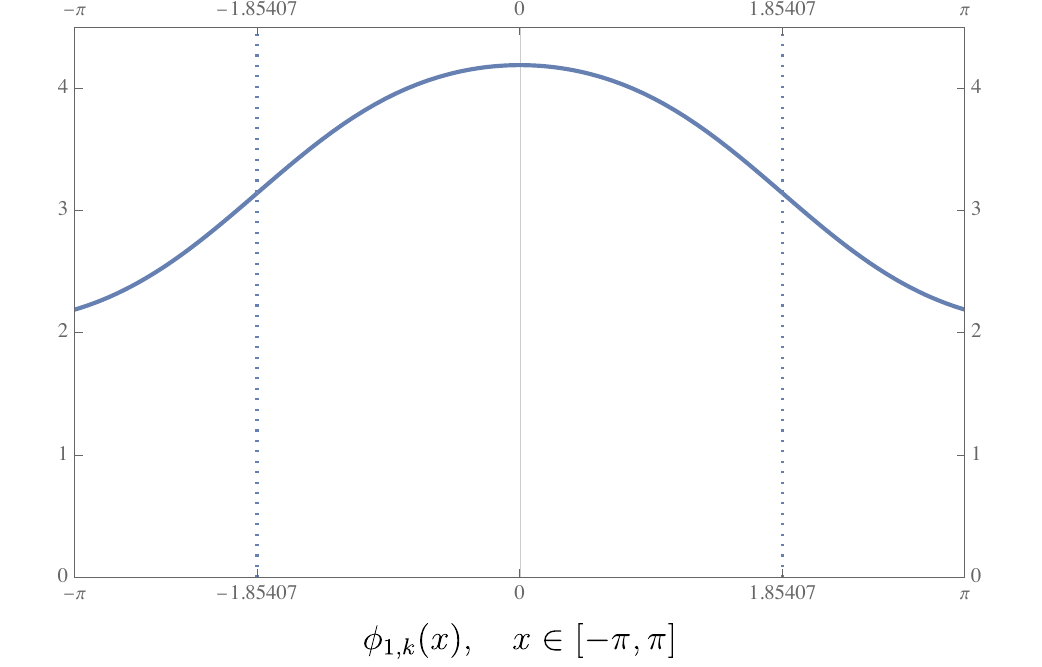}
\caption{\small{Graph of the single-lobe solution (in solid blue line), composed by the extension of $\phi_{1, k}$ in \eqref{La2} with $c_1=1$, $k=0.5$ to the whole interval $x\in [-\pi, \pi]$. The glueing of both solutions occurs at $x =b$ and $x= -b$ with $b \approx 1.85407$ (color online).}}\label{fig7} 
\end{figure}

\begin{proposition}
\label{3profile}
 Let  $L, c_1, c_2$ be fixed positive constants. We consider $a=0$ in \eqref{trav22}. Then, the degenerate  single-lobe kink state $(\pi, \phi_{2,0})$ belongs to $D_Z$ if and only if $Z>0$ with $Z=\frac{2}{\pi c_2}$ (see  Figure \ref{fig8}).
 \end{proposition}
\begin{proof}
 From $a=0$ and  $\phi_1(L)=\phi_{2,0}(L)=\pi$ we get that $\phi_1(x)\equiv \pi$ is the center solution for \eqref{1lobe}. Thus,  $(\pi, \phi_{2,0}) \in D_Z$ if and only if $Z>0$ with $Z=\frac{2}{\pi c_2}$.
\end{proof}

\begin{figure}[h]
 \centering
\includegraphics[angle=0,scale=0.55]{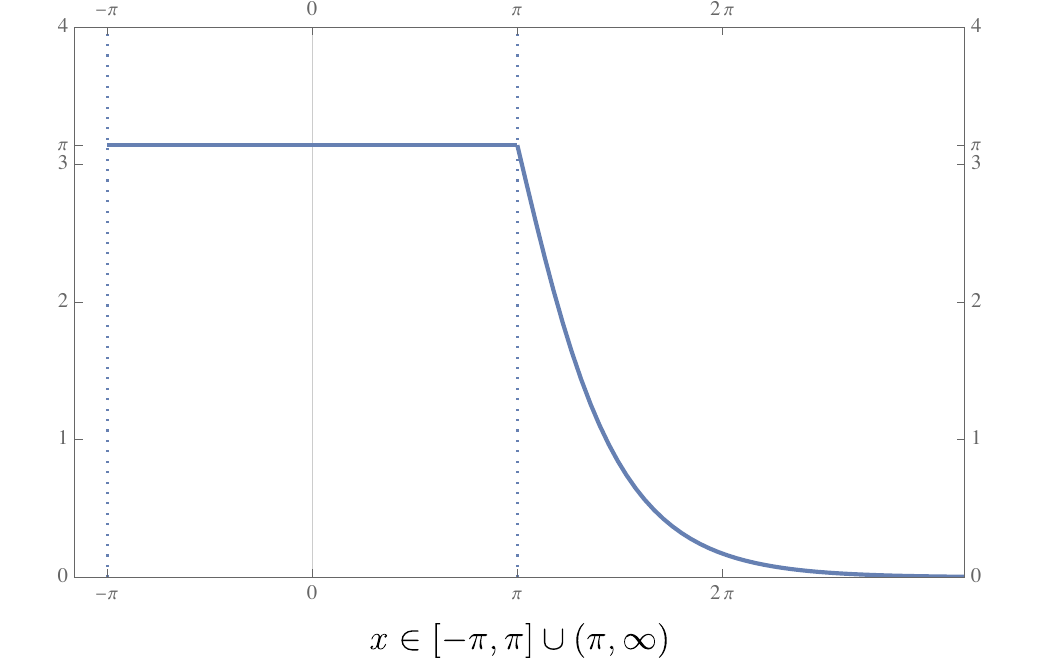}
\caption{\small{Degenerate single-lobe kink state  $(\pi, \phi_{2,0})$ for the sine-Gordon model.}} \label{fig8}
\end{figure}

\subsection{Existence of a single-lobe kink state} 

 In this subsection we determine conditions for the profile  $\Theta=(\phi_1, \phi_2)$ with  $\phi_2$ in \eqref{trav22} and $\phi_1$ in \eqref{formula1} or \eqref{La2} belongs to the $\delta$-interaction domain $D_Z$ in \eqref{Domain0}. Thus, from Propositions  \ref{1profile} and \ref{2profile} we need to determine one expression for  the shift-value  $a=a(k)$ associated to the kink-profile $\phi_2=\phi_{2,a}$. From Proposition \ref{Z2}, we know that  {\it a priori}  the strength $Z$ satisfies $Z\in (-\infty, \frac{2}{\pi c_2})$.

 \begin{proposition}
 \label{exis1}
 Let  $L, c_1$ fixed positive constants.  For $a<0$ we consider the kink-soliton profile $\phi_{2,a}$  in \eqref{trav22} with $\phi_{2,a}(L)>\pi$  and the single-lobe profile $\phi_{1,k}$ in \eqref{formula1}.  
 
 We have a single-lobe kink state, namely $(\phi_{1,k},\phi_{2,a})\in  D_Z$,  in the following two cases:
 \begin{enumerate}
 \item[(1)] Suppose that $\frac{L}{c_1}>\frac{\pi}{2}$. Then for $k_0\in (0,1)$ such that $K(k_0)=\frac{L}{c_1}$, we require,
  \begin{enumerate}
 \item[(i)] for $Z>0$ and $\frac{c_1}{2c_2}\geqq \tanh(\frac{L}{c_1})$ we need to have $Z\in \big(0, \frac{4}{\pi c_1}\big[\frac{c_1}{2c_2}-k_0\big ]\big)$;
  \item[(ii)] for $Z>0$ and $k_0<\frac{c_1}{2c_2}< \tanh(\frac{L}{c_1})$ we need to have $Z\in \big(0, \frac{4}{\pi c_1}\big[\frac{c_1}{2c_2}-k_0\big]\big)$;
 \item[(iii)] for  $Z<0$ and  $\frac{c_1}{2c_2}<k_0$ we need to have $Z\in \big(\frac{4}{\pi c_1}\big[\frac{c_1}{2c_2}-k_0\big], 0\big)$;
 \item[(iv)]  for  $Z<0$ and  $\frac{c_1}{2c_2}=k_0$ there is $m_0<0$ such that we need to have $Z\in (m_0, 0)$;
  \item[(v)]  for  $Z<0$ and  $k_0<\frac{c_1}{2c_2}< \tanh(\frac{L}{c_1})$ there is $m_1<0$ such that we need to have $Z\in (m_1, 0)$; and,
 \item[(vi)]  for $Z=0$ there is  a single-lobe kink state if and only if $\frac{c_1}{2c_2}\in (k_0, \tanh(L/c_1))$.   
   \end{enumerate}
  \item[(2)] Suppose that $\frac{L}{c_1}\leqq \frac{\pi}{2}$. Then we require,
  \begin{enumerate}
 \item[(i)]  $Z\in (0, \frac{2}{\pi c_2})$;
  \item[(ii)]  for $Z<0$ and $\frac{c_1}{2c_2}< \tanh(\frac{L}{c_1})$ there is $m_2<0$ such that we need to have $Z\in (m_2,0)$;
   \item[(iii)] $Z=0$ and $\frac{c_1}{2c_2}<  \tanh(\frac{L}{c_1})$.
 \end{enumerate}
 \end{enumerate}
 
 There is not  single-lobe kink state, namely $(\phi_{1,k},\phi_{2,a})\notin  D_Z$,  in the following  two cases:
 \begin{enumerate}
  \item[(3)] Suppose that $\frac{L}{c_1}>\frac{\pi}{2}$. Then for $k_0\in (0,1)$ such that $K(k_0)=\frac{L}{c_1}$ we have the following conditions:
  \begin{enumerate}
 \item[(i)] for  $Z>0$ and  $\frac{c_1}{2c_2}\leqq k_0$ there is not a single-lobe kink state,
 \item[(ii)] for  $Z<0$ and  $\frac{c_1}{2c_2}\geqq \tanh(\frac{L}{c_1})$ there is not a single-lobe kink state,
\end{enumerate} 
 
  \item[(4)] Suppose $\frac{L}{c_1}\leqq \frac{\pi}{2}$. Then we have the following conditions:
 \begin{enumerate}
  \item[(i)]  for $Z<0$ and $\frac{c_1}{2c_2}\geqq  \tanh(\frac{L}{c_1})$ there is not a single-lobe kink state,
  \item[(ii)]  for $Z=0$ and $\frac{c_1}{2c_2}\geqq  \tanh(\frac{L}{c_1})$ there is not a single-lobe kink state.
 \end{enumerate}
 \end{enumerate} 
\end{proposition}
\begin{remark} 
 The proof of Proposition \ref{exis1} shows formulae for $m_0, m_1, m_2$. Since $\tanh(K(k))>k$ for all $k\in (0,1)$, then in the case of   single-lobe kink state existence $(1)(i)$ above we necessarily have $\frac{c_1}{2c_2}-k_0>0$.
\end{remark}

As the proof of Proposition \ref{exis1} is  a bit technical, we will divide it into two lemmas, where the first of these shows that we can to glue the subluminal-profile $\phi_{1,k}$ in \eqref{formula1} with a kink-soliton profile $\phi=\phi_{2,a}$, $a=a(k)$,  at the vertex $\nu=L$, namely,  $\phi_{1,k}(L)=\phi_{2,a}(L)$ and with $k$ satisfying {\it a priori} $K(k)>\frac{L}{c_1}$. Here, we  establish an explicit formula for the  shift parameter $a=a(k)$.  In the second Lemma, we determine the correct values of $k$ and $Z$ such that $(\phi_{1,k}, \phi_{2, a(k)})\in D_Z$, namely, such that we have  the $\delta$-coupling relation $
2\phi'_{1,k}(L)=\phi'_{2,a}(L)+Z\phi_{2,a}(L)$. 

\begin{lemma}
 \label{Pexis1}
 Let  $L, c_1, c_2$ fixed positive constants.  For $a<0$ we consider the kink-soliton profile $\phi_{2,a}$  in \eqref{trav22} with $\phi_{2,a}(L)>\pi$  and the single-lobe profile $\phi_{1,k}$ in \eqref{formula1}.  Then,  $\phi_{1,k}(L)=\phi_{2,a}(L)$ for $a=a(k)$ given by
 \begin{equation}
\label{0form-a}
\sech\Big(\frac{a(k)}{c_2}\Big)=\frac{k'}{\dn(L/c_1;k)},
\end{equation}
where $k'=\sqrt{1-k^2}$.
\end{lemma}

\begin{proof} Let $a<0$ and with $k$ satisfying {\it a priori} $K(k)>\frac{L}{c_1}$. By formulae \eqref{trav22},  \eqref{formula1}, and continuity condition $\phi_{1,k}(L)=\phi_{2,a}(L)$, we obtain the following first formula for the shift parameter $a=a(k)$,
\begin{equation}
\label{1form-a}
e^{-\frac{a}{c_2}}=\frac{\dn(L/c_1;k)+k \cn(L/c_1;k)}{k'},
\end{equation}
where $\dn(\cdot;k)$ and $\cn(\cdot;k)$ represent the Jacobian elliptic functions of dnoidal and cnoidal type, respectively (see Figure \ref{fig5}). Here, $k'$ is defined by $k'=\sqrt{1-k^2}$. In the following, we show formula in \eqref{1form-a}. Indeed, for $\beta\equiv \arctan(e^{-\frac{a}{c_2}})$ we get from \eqref{profile} that 
$$
\cos(4\beta)=-1+ 2k^2sn^2\Big(\frac{L}{c_1} +K; k\Big).
$$
 Then, from $\cos(4\beta)=2\cos^2(2\beta)-1$ follows $k^2 \sn^2(\frac{L}{c_1} +K;k)=\cos^2(2\beta)$. Thus, since $2\beta\in [\frac{\pi}{2}, \pi]$ and $\frac{L}{c_1} +K< 2K$ we get $\cos(2\beta)=-k \sn(\frac{L}{c_1} +K;k)$ (see, for instance, Figure \ref{fig6a}). Moreover, from $\cos^2(2\beta)+\sin^2(2\beta)=1$ and $1-k^2 \sn^2 = \dn^2$ (see \cite{ByFr71}) we also obtain $\sin(2\beta)=\dn(\frac{L}{c_1} +K;k)$. Therefore,
\begin{equation}
\label{form-a2}
e^{-\frac{a}{c_2}} = \tan(\beta)=\frac{1-\cos(2\beta)}{\sin(2\beta)}=\frac{1+k\sn(\frac{L}{c_1} +K;k)}{\dn(\frac{L}{c_1} +K;k)}=\frac{\dn(L/c_1;k)+k \cn(L/c_1;k)}{k'},
\end{equation}
where we are using the relations $\sn(u+K)=\cn(u)/\dn(u)$ and  $\dn(u+K)=k'/\dn(u)$ (cf. Byrd and Friedman \cite{ByFr71}). 
 
 Now, from \eqref{1form-a}  we obtain the following  more manageable relationship between $a=a(k)$ and $k$,
\begin{equation}
\label{2form-a}
\sech\Big(\frac{a(k)}{c_2}\Big)=\frac{k'}{\dn(L/c_1;k)},
\end{equation}
and so we get the property that the  shift-value mapping $k\to a(k)$ will be smooth. {\it Formula \eqref{2form-a} shows that always we can glue a librational profile with a kink-soliton profile because $k'\leqq \dn(x;k)$ for all $x$ and $k$}.
\end{proof}

\begin{lemma}
 \label{exis2}
 Let  $L, c_1, c_2$ fixed positive constants.  For $a<0$ we consider the kink-soliton profile $\phi_{2,a}$  in \eqref{trav22} with $\phi_{2,a}(L)>\pi$  and the single-lobe profile $\phi_{1,k}$ in \eqref{formula1}.  Then,  we have  the $\delta$-coupling  relation 
$$
2\phi'_1(L)=\phi'_2(L)+Z\phi_2(L),
$$ 
with $\phi_1=\phi_{1,k}$, $\phi_2=\phi_{2, a(k)}$, and $a(k)$ satisfying relation in \eqref{2form-a}, provided that the conditions in items (1) and (2) of Proposition \ref{exis1}  hold.  Moreover,  there is not  single-lobe kink state for the sine-Gordon model on a tadpole graph, namely $(\phi_{1,k},\phi_{2,a})\notin  D_Z$,  provided that the conditions in (3) and (4) of Proposition \ref{exis1}  hold. 
\end{lemma}

The proof of Lemma \ref{exis2} is a bit technical,  so we leave this to be proven in Appendix.

In the sequel we formally prove that, for the Neumann-Kirchhoff condition in \eqref{Domain0}, namely, $Z=0$, we get a single-lobe kink state for the sine-Gordon model with a profile on $[-L, L]$ given by formula in \eqref{La2}. For $Z\neq 0$ we found 
technical difficulties in determining the existence of these profiles (see Remark \ref{3pi}).

\begin{proposition}\label{exis3}
Let  $L, c_1$ be fixed positive constants with $\frac{L}{c_1}>\frac{\pi}{2}$. For the subluminal profile $\phi_{1,k}$ in \eqref{La2} with $k\in (0,k_0)$ and $K(k_0)=\frac{L}{c_1}$, there is a shift-value $a = a(k)>0$ associated to the kink-profile $\phi_{2,a}$  in \eqref{trav22} such that $(\phi_{1,k},\phi_{2,a})\in  D_Z$, with $Z=0$, if and only if $c_2$ satisfies
\[
\frac{c_1}{2k_0}< c_2.
\]
 Moreover, $a(k)>0$ is defined  uniquely by relation in \eqref{2form-a}.  For $Z<0$,  there is not a single-lobe kink state with a first-component profile $\phi_{1,k}$ in \eqref{La2} with $c_2$ satisfying $\frac{c_1}{2c_2}\geqq k_0$.  
 \end{proposition}
 
\begin{proof}
We show initially that for $x\in (b, L]$,  $2K<\frac{x}{c_1}+K< 3K$ for every admissible $k$. Indeed, suppose $\beta\in (b,L]$ such that $2K\geqq \frac{\beta}{c_1}+K$ then since $b=c_1K$ follows immediately $2K>\frac{b}{c_1}+K=2K$. Next, for all $k\in (0,k_0)$ we have $\frac{L}{c_1}+K(k)< 3K(k)$, in fact, suppose there is $p\in (0,k_0)$ such that $\frac{L}{c_1}+K(p)> 3K(p)$ then there is $\beta\in (0, p)$ satisfying $\frac{L}{c_1}+K(\beta)= 3K(\beta)$. Thus, from relations \eqref{GF} and \eqref{La2} we get
\begin{equation}
-\frac{c_1^2}{2}[\phi_{1,\beta}'(L)]^2=\cos(\phi_{1,\beta}(L))-\cos(\phi_{1,\beta}(0))=-2\beta^2\cn^2\Big(\frac{L}{c_1}+K(\beta);\beta\Big)=0.
\end{equation}
Therefore, $\phi_{1,\beta}'(L)=0$, which is clearly not possible. We note from the profile of the elliptic function $\sn(\cdot)$ that  $-1 \leqq \sn(\frac{x}{c_1}+K)\leqq 0$ for $x\in (b, L]$ and for $x=L$,  $K<\frac{L}{c_1}<2K$ and hence $\sn(\frac{L}{c_1})>0$ and $\cn(\frac{L}{c_1})<0$.

Next, to show the continuity  condition  $\phi_{1,k}(L)=\phi_{2,a}(L)$ for $a=a(k)>0$ we shall follow the same strategy as in the proof of Lemma \ref{Pexis1}. Indeed, we get the first same expression for the shift $a$ as in \eqref{1form-a} by considering again $\beta\equiv \arctan(e^{-\frac{a}{c_2}})$ but with $\beta\in (0,\frac{\pi}{4}]$. Thus,  from relation   $k^2\sn^2(\frac{L}{c_1} +K;k)=\cos^2(2\beta)$ we
get $\cos(2\beta)=-k\sn(\frac{L}{c_1} +K;k)$ and $\sin(2\beta) = \dn(\frac{L}{c_1} +K;k)$. Thus, we obtain  \eqref{1form-a} with $a>0$ and a similar relation as  in \eqref{2form-a}. Next, we determine the right $k\in (0,k_0)$ such that $2\phi_{1,k}'(L)=\phi'_{2, a(k)}(L)$. Initially we get
\begin{equation}\label{limits2}
 \lim_{k\to 0} \sech\Big(\frac{a(k)}{c_2}\Big)= 1,\quad \lim_{k\to k_0} \sech\Big(\frac{a(k)}{c_2}\Big)=1.
 \end{equation} 
Therefore, $\lim_{k\to 0} a(k)=0$ and $\lim_{k\to k_0} a(k)=0$ ($a(k_0)=0$). We note that, compared to the negative case of $a$, $a=a(k)$ is no longer a strictly increasing function of $k$. In Figure \ref{figj(k)} we show a generic profile of $a(k)$ given by formula in \eqref{a(k)} and we notice that there is always a unique value $k_{\mathrm{crit}}\in (0, k_0)$ such that $a'(k_{\mathrm{crit}})=0$.

\begin{figure}[h]
 \centering
\includegraphics[angle=0,scale=0.55]{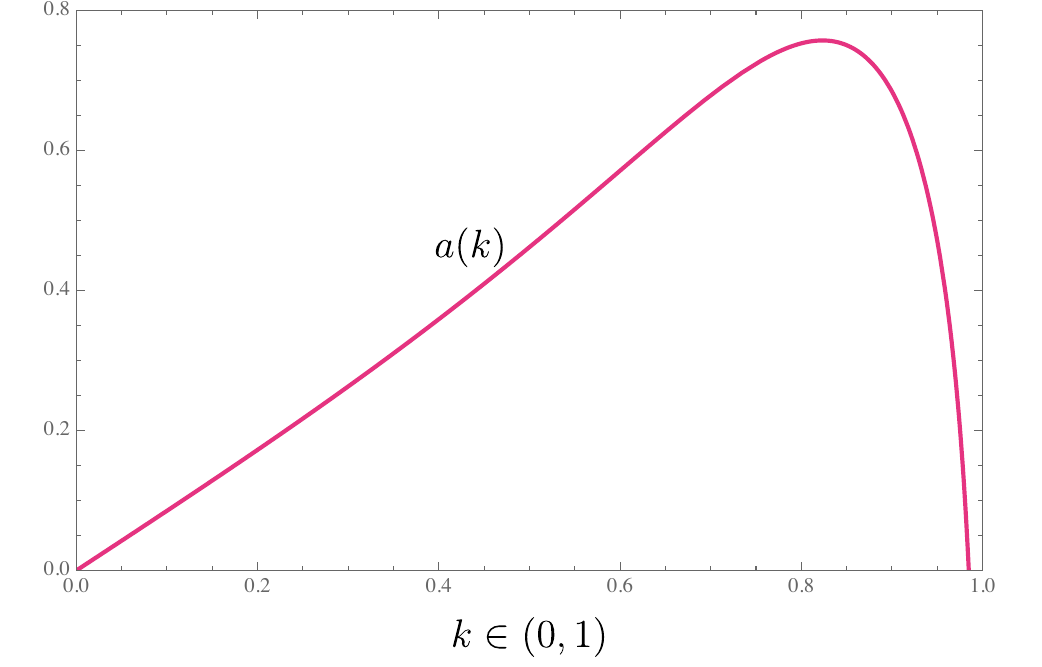}
\caption{\small{Graph of the function $a = a(k)$  for $k\in (0, k_0)$, with $c_1=1$, $L=\pi$ and $k_0 \approx 0.984433$, $K(k_0)=\pi$.} } \label{figj(k)}
\end{figure}

Now, as relations for $\phi_{1,k}'(L)$ and $\phi_{2,a}'(L)$ in \eqref{form-k1} and \eqref{form-k2}, respectively, holds true, we obtain $2\phi_{1,k}'(L)=\phi'_{2, a(k)}(L)$ if and only if $k$ satisfies the following  relation
\begin{equation}\label{k-value}
k\sn\Big(\frac{L}{c_1};k\Big)=\frac{c_1}{2c_2}.
 \end{equation} 
We now study the mapping $k\in (0,k_0)\to F(k)=k\sn\Big(K(k_0);k\Big)$: 
\begin{equation}\label{F}
F(0)=0,\quad \lim_{k\to k_0} F(k)=k_0 \sn\Big(K(k_0);k_0\Big)=k_0, \quad F(k)\leqq k_0.
 \end{equation} 
Therefore, for $c_2$ satisfying $\frac{c_1}{2k_0}\geqq c_2$ there is not  single-lobe kink state for the sine-Gordon model on a tadpole graph. Now, we analyze  the case for  $c_2$ satisfying $\frac{c_1}{2k_0}< c_2$. We note that in general the mapping $k\to F(k)$ has a profile with several oscillations (for instance, for large $L$) with finite simple zeros $\{0, r_1, ..., r_n\}\subset [0, k_0)$ and for $k\in (r_n, k_0)$ this mapping is strictly increasing, so from \eqref{F} there is a unique $k_{c_2}\in (r_n, k_0)$ such that $F(k_{c_2})=\frac{c_1}{2c_2}$. Therefore, we obtain the existence of single-lobe kink profile $(\phi_{1,k_{c_2}},\phi_{2,a(k_{c_2})})\in  D_0$ for the sine-Gordon model. 

Let $Z<0$, then from equation $H(k)=Z$ in \eqref{H} with $H$ defined in \eqref{form-k4} for $k\in (0,k_0)$, we need to have {\it a priori} that $\frac{c_1}{2c_2}<k\sn(L/c_1;k)$, for some $k\in (0,k_0)$. Therefore, for $\frac{c_1}{2c_2}\geqq k_0$ there is not a single-lobe kink state. This finishes the proof.
\end{proof}

\begin{remark}
\label{3pi}
A few observations are in order.
\begin{enumerate}
\item[(a)]  Let us illustrate the case $L=3\pi$, $c_1=1$, with $k_0$ satisfying $K(k_0)=3\pi$ ($k_0 \approx 1$).  Figure \ref{fig19} shows the profile of $F(k)=k\sn\big(K(k_0);k\big)$ for $k\in (0,k_0)$. 
\begin{figure}[h]
\centering
\includegraphics[angle=0,scale=0.5]{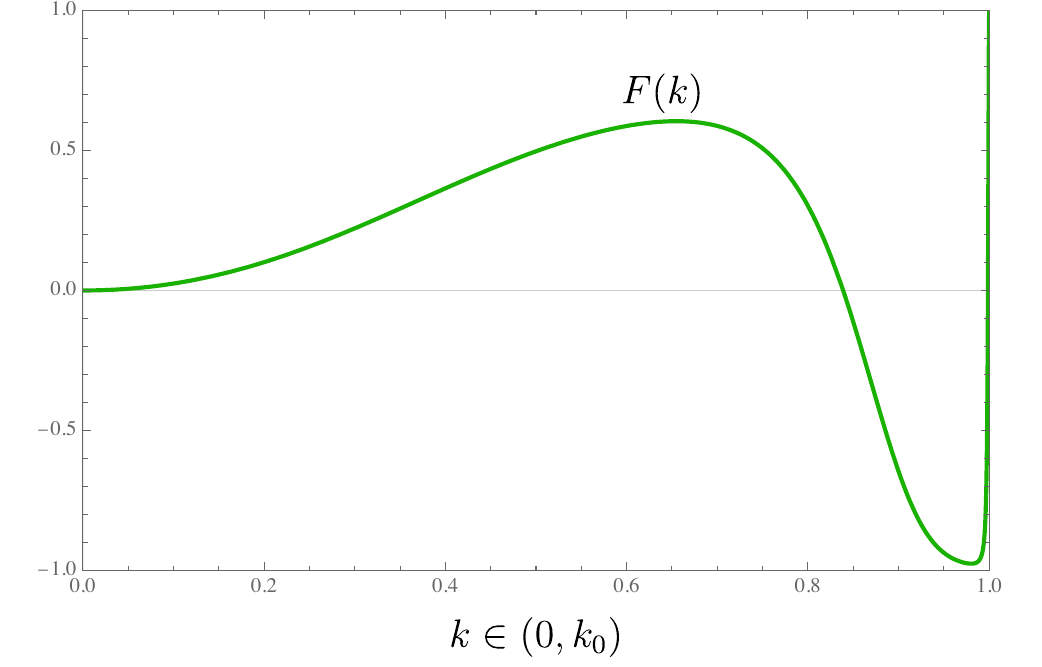}
\caption{\small{$F(k)=k\sn(K(k_0);k)$ for $k\in (0,k_0)$, $K(k_0)=3\pi$.}} \label{fig19}
\end{figure}

Now, $F(k)=0$ if and only if $r_1\approx 0.84$ and $r_2\approx 0.9987$ (see Figures \ref{fig20} and \ref{fig21}). By chosing $c_2=2$, we have that $F(k)=\frac14$ if and only if $k_1\approx 0.32$, $k_2 \approx 0.81$ and $k_3 \approx 0.999$. But the right $k_i$ that produces a single-lobe kink profile ($\phi_{1,k}'(x)<0$ for all $x\in [0, 3\pi]$) will be exactly $k_3 \approx 0.999$. 
\begin{figure}[h]
\begin{center}
\subfigure[]{\label{fig20}\includegraphics[scale=0.45]{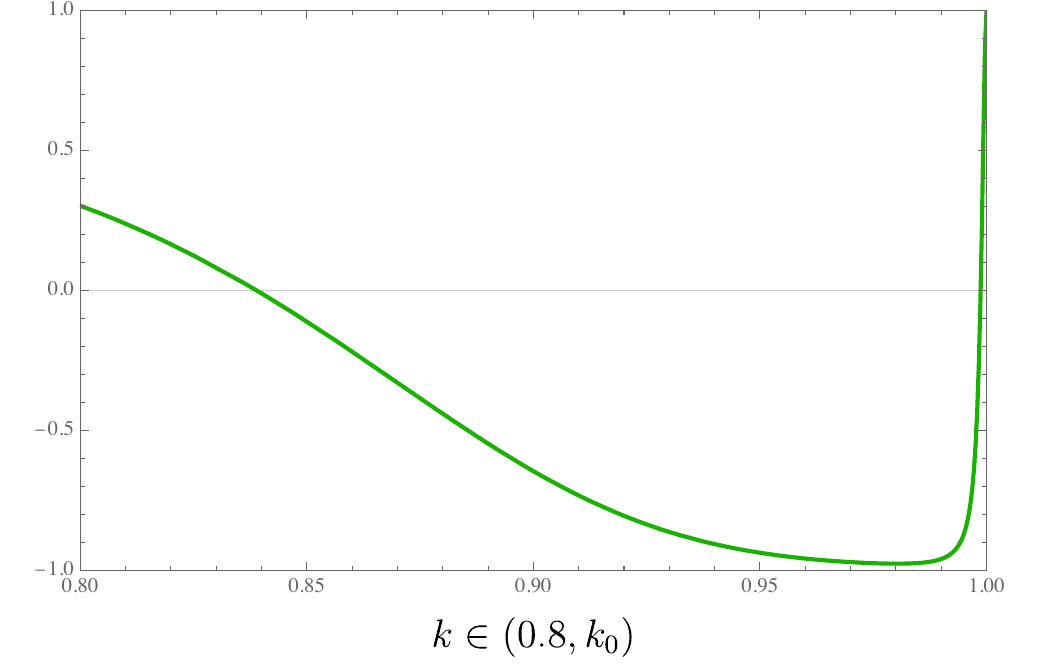}}
\subfigure[]{\label{fig21}\includegraphics[scale=0.45]{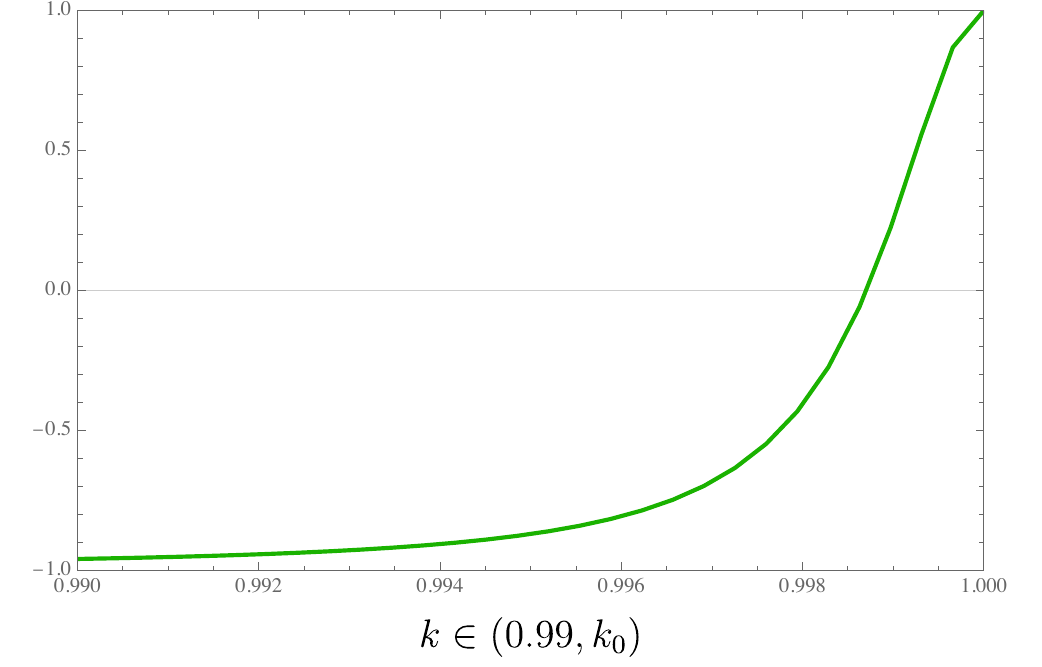}}
\caption{\small{$F(k)=k\sn(K(k_0);k)$ for $k\in (0.8,k_0)$ (panel (a)) and for $k\in (0.99,k_0)$ (panel (b)); here $K(k_0)=3\pi$.  (color online). }
}
\label{fig2021}
\end{center}
\end{figure}


%
In this way, Figures \ref{fig22}, \ref{fig23} and \ref{fig24}  show the profile of $\phi_{1,k_i}(x)$ in \eqref{La2} with $x\in [K(k_i), 3\pi]$.
\begin{figure}[h]
\begin{center}
\subfigure[]{\label{fig22}\includegraphics[scale=0.45]{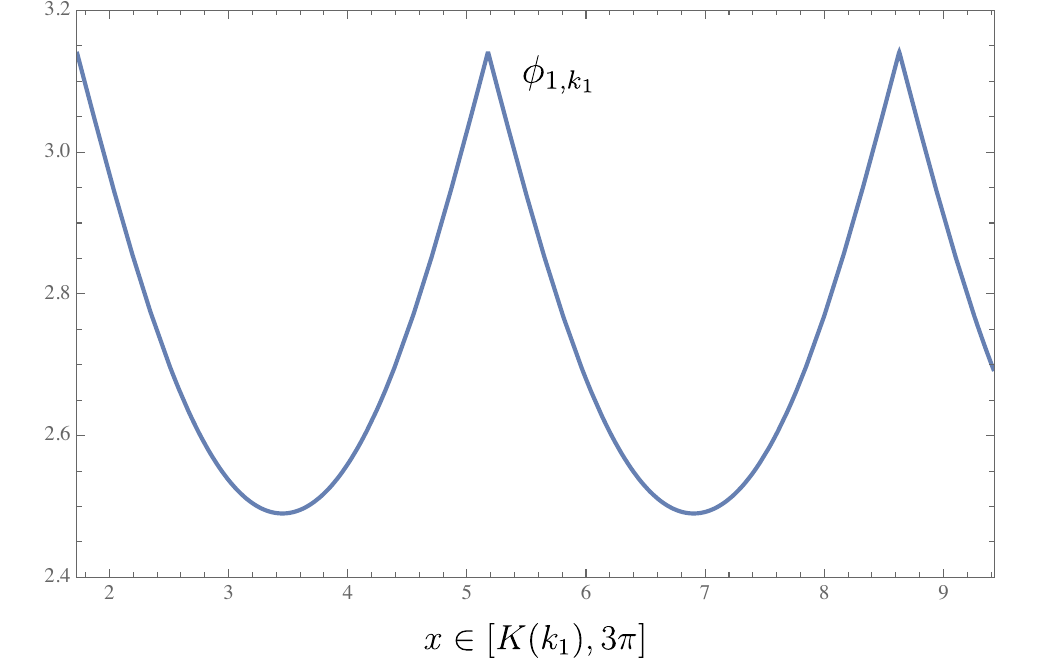}}
\subfigure[]{\label{fig23}\includegraphics[scale=0.45]{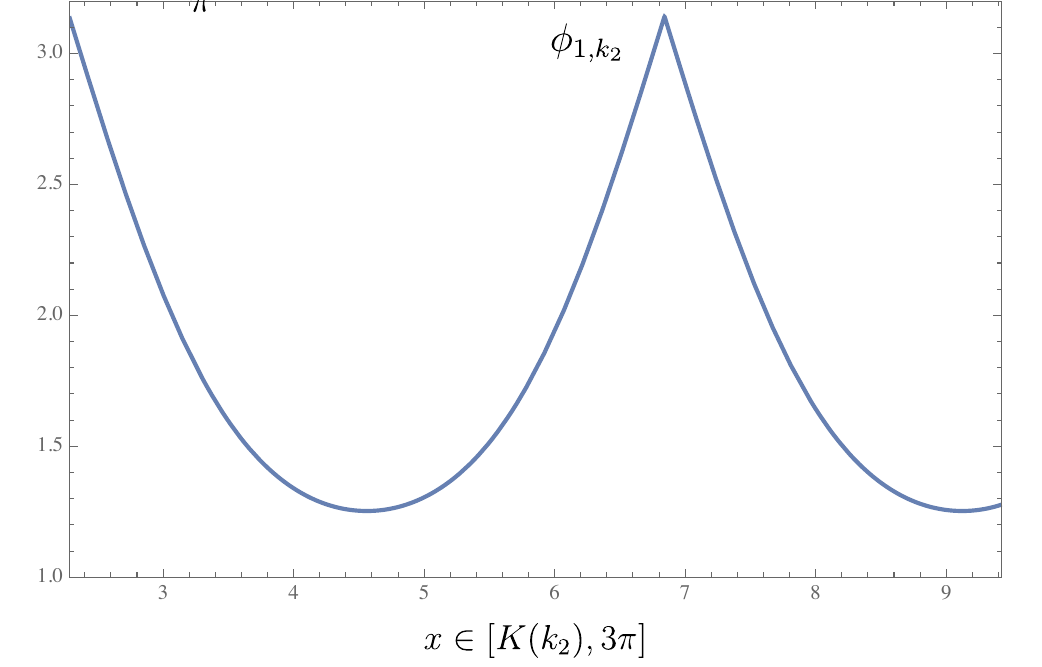}}
\subfigure[]{\label{fig24}\includegraphics[scale=0.45]{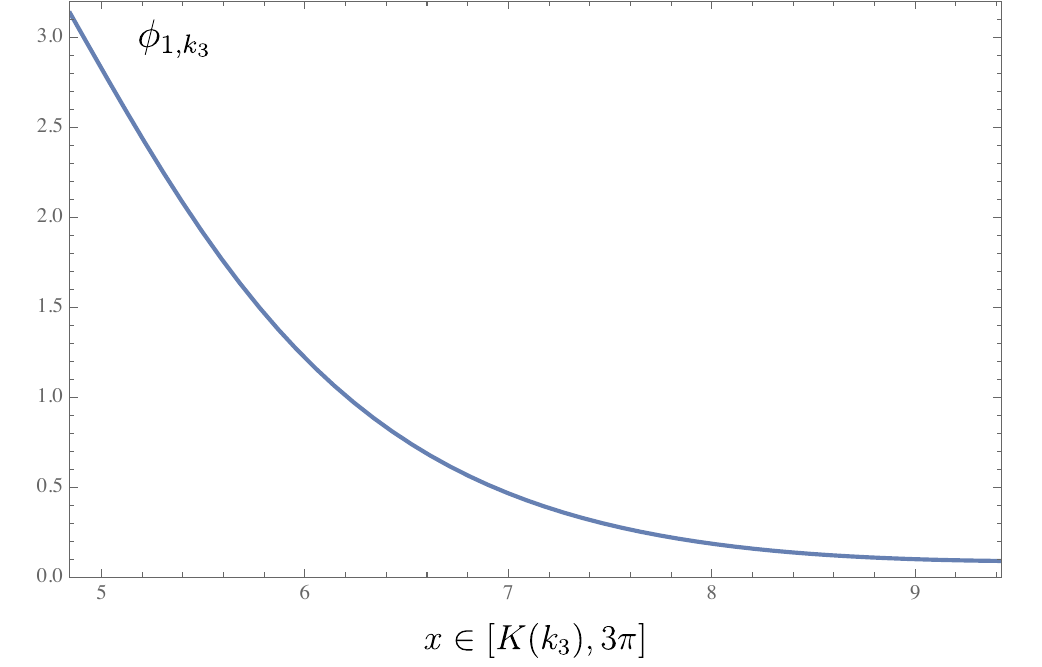}}
\caption{\small{Graph of $\phi_{1,k}$ (in solid blue line) for $k=k_1$ (panel (a)), $k=k_2$ (panel (b)) and $k=k_3$ (panel (c)) in the domains $x \in [K(k_i),3\pi]$, $i =1,2,3$ (color online). }
}
\label{fig222324}
\end{center}
\end{figure}

\item[(b)] For the case $Z<0$  and $c_2$ such that $k_0> \frac{c_1}{2c_2}$, some technical problems arise in studying equation $H(k)=Z$ in \eqref{H} to obtain a general picture of existence of single-lobe kink profile  as was that obtained in  Proposition \ref{exis1}. The problem here is that  the mapping $k\to k\sn(L/c_1;k)$, $k\in (0,k_0)$, may have  a profile with several oscillations that do not allow us to clearly study the mapping $H$ in \eqref{form-k4}. A similar situation is obtained for $Z>0$.
\end{enumerate}
\end{remark}

\section{Spectral structure associated to positive single-lobe kink states} 
\label{secspecstudy}
   
Let us consider $\Theta=(\phi_1,\phi_2)\in D_Z$, an arbitrary positive single-lobe kink state for \eqref{trav21}, and the family of self-adjoint operators $(\mathcal L_Z, D_Z)$,
\begin{equation}
\label{sg5}
\mathcal L_Z=\Big (\Big(-c_j^2\frac{d^2}{dx^2} + \cos(\phi_j)\Big )\delta_{j,k}\Big),\quad 1\leqq j,k\leqq 2, \;\;c_j>0.
\end{equation}
In this section we characterize the structure of the Morse and nullity indices of the diagonal Schr\"odinger operators $\mathcal L_Z$. 

 We start with the following ``moving framework''. For  $(f,g)\in  D_Z$ consider $h(x)=g(x+L)$ for $x>0$. Then $h(0)=g(L)$ and $h'(0)=g'(L)$. Therefore, the eigenvalue problem $\mathcal L_{Z} (f,g)^\top=\lambda (f,g)^\top$ is equivalent to the following system
  \begin{equation}\label{La0}
 \left\{ \begin{array}{ll}
\mathcal L_1f(x)=\lambda f(x),\quad\quad x\in (-L,L),\\
 \mathcal L_{a}h(x)=\lambda h(x), \quad\quad x\in (0, \infty),\\
(f,h)\in D_{Z, 0},
  \end{array} \right.
 \end{equation} 
where
  \begin{equation}\label{L}
\mathcal L_{1}=-c_1^2\partial_x^2+ \cos(\phi_1),\;\;\;\; \mathcal L_{a}\equiv -c_2^2\partial_x^2+ \cos(\psi_a)
 \end{equation} 
 with 
 \begin{equation}\label{akink}
 \psi_{a}(x)= \phi_2(x+L)=4\arctan\Big( e^{-\frac{1}{c_2}(x+a)}\Big), \qquad x>0,
  \end{equation}
 representing the kink-soliton profile for any shift-value $a\in \mathbb R$. Here, 
 \begin{equation}\label{D0}
   D_{Z,0}=\big\{(f,h)\in X^2(-L,L): f(L)=f(-L)=h(0),\; \text{and},\; f'(L)-f'(-L)=h'(0)+Zh(0)\big\}
   \end{equation}
with $X^n(-L, L)\equiv H^n(-L,L)\oplus H^n(0, \infty)$, $n\in \mathbb N$. Naturally, $(\phi_1, \psi_{a})\in D_{Z,0}$.  For convenience of notation, we will use $\mathcal L\equiv ( \mathcal L_{1}, \mathcal L_{a})$.

In what follows, we establish our main strategy in order to study the eigenvalue problem in \eqref{La0}.  More precisely, we will reduce the eigenvalue problem for $(\mathcal L,  D_{Z,0})$  into two classes of eigenvalue problems, one for $\mathcal L_1$ with periodic boundary conditions on $[-L,L]$, and one for the operator $\mathcal L_{a}$ with $\delta$-type boundary conditions on the half-line $(0,\infty)$. Thus, we employ the so-called  {\it {splitting eigenvalue method}} on a tadpole graph (see Angulo \cite{Angloop, Angtad}).   

 \begin{lemma}[splitting eigenvalue method]
 \label{split}  
 Suppose that $(f,g)\in D_{Z,0}$ with $g(0)\neq 0$, satisfies the relation $\mathcal L(f,g)^\top=\lambda (f,g)^\top$ for $\lambda < 0$. Then, we obtain the following two eigenvalues problems
 \begin{equation}\label{La1}
 (PBP)\left\{ \begin{array}{ll}
\mathcal L_1f(x)=\lambda f(x),\quad x\in (-L,L),\\
f(-L)=f(L),\\
 f'(-L)=f'(L),
  \end{array} \right. \quad
 (\delta BP)  \left\{ \begin{array}{ll}
\mathcal L_a g(x)=\lambda g(x),\quad\quad x\in (0,\infty),\\
g'(0+)=-Zg(0+).
  \end{array} \right.  
  \end{equation} 
\end{lemma}
\begin{proof} 
Following the analysis by Angulo \cite{Angloop, Angtad},  for $(f, g)\in D_{Z,0}$ and $g(0)\neq 0$, we have
\[
f(-L)=f(L),\;\;f'(L)-f'(-L)=\Big[\frac{g'(0+)}{g(0+)} +Z\Big] f(L)\equiv \theta f(L),
\]
and so $f$ satisfies the following real coupled problem
\begin{equation}\label{RC0}
\left\{ \begin{array}{ll}
\mathcal L_{1}f(x)=\lambda f(x),\quad\quad x\in (-L,L),\\
 f(L)=f(-L),\\
f'(L)-f'(-L)=\theta f(L).
  \end{array} \right.
 \end{equation} 
Then, by using oscillation theory for real coupled problems (see Theorem 4.8.1 in Zettl \cite{Zettl05}), $\phi_1$ has an even single-lobe profile, Sturm comparison theorem, and $\lambda < 0$, we get $\theta=0$. This finishes the proof.

\end{proof}

\subsection{Morse index for $\mathcal{L}_Z$}

The purpose of this subsection is to characterize the number of negative eigenvalues for $\mathcal L\equiv ( \mathcal L_{1}, \mathcal L_{a})$ with $ \mathcal L_{1}, \mathcal L_{a}$ defined in \eqref{L} and determined by {\it a priori} $(\phi_1, \psi_a)$ positive single-lobe kink state, $\psi_a$ in \eqref{akink} and $a\in \mathbb R$. By the splitting eigenvalue result in Lemma \ref{split}, it arises naturally to study the Morse index for the
family of self-adjoint operators,
\begin{equation}
\label{La}
\mathcal L_a= -c_2^2\partial_x^2+ \cos(\psi_a),
\end{equation}
with $\delta$-type boundary conditions on $(0,\infty)$.

\begin{lemma}
 \label{delta} Consider the operator $\mathcal L_a$ in \eqref{La} on the domain
\[
W_Z=\{g\in H^2(0, \infty): g'(0+)=-Zg(0+)\}.
\]
Then, for $Z\in \mathbb R$ we have that $(\mathcal L_a, W_Z)$ represents a family of self-adjoint operators and that the Morse index of  $\mathcal L_a$, $n(\mathcal L_a)$, satisfies $n(\mathcal L_a)\leqq 1$ for every $Z$.
 \end{lemma}
\begin{proof} We will use the extension theory for symmetric operators (see Appendix D in \cite{Albe}). In fact, by considering the minimal symmetric Schr\"odinger operator,
\[
\mathcal{M}_0\equiv \mathcal L_a,\quad D(\mathcal{M}_0)=C_0^\infty (0,\infty),
\]
we have that the  closure of $\mathcal{M}_0$, denoted by $\mathcal{M}\equiv\mathcal{M}_0$, has deficiency indices $n_{\pm}(\mathcal{M})$, satisfying $n_{\pm}(\mathcal{M})=1$.  By Proposition \ref{11} in Appendix we obtain that the family $(\mathcal L_a, W_Z)$ represents all the self-adjoint extensions of the closed symmetric operator $(\mathcal M, D(\mathcal M))$. Next, we show that  $\mathcal{M}\geqq 0$. Indeed,  from \eqref{trav21} we obtain 
\begin{equation}\label{spec7}
\mathcal Mf(x)=-\frac{1}{\psi'_a(x)} \frac{d}{dx}\Big[c_2^2  (\psi'_a(x))^2 \frac{d}{dx}\Big(\frac{f(x)}{\psi'_a(x)}\Big)\Big], \qquad x>0.
\end{equation}
We note that always we have $\psi'_a\neq 0$ on $[0,\infty)$. Thus for $f\in C_0^\infty (0,\infty)$ we obtain
\begin{equation}\label{spec6}
\begin{split}
\langle \mathcal{M} f, f\rangle&=\int_0^{\infty}c_2^2(\psi'_a)^2\Big|\frac{d}{dx}\Big(\frac{f}{\psi'_a}\Big)\Big|^2dx+ c_2^2f(0)\Big[\frac{f'(0)\psi'_a(0)-f(0)\psi'_a(0)}{\psi'_a(0)}\Big].
\end{split}
\end{equation}
  Due to the condition $f(0)=0$ the non-integral term vanishes, and we get $\mathcal M\geqq 0$. Then, by  Proposition \ref{semibounded} (see Appendix),  we have all the  self-adjoint extensions $(\mathcal L_a, W_Z)$ of $\mathcal M$ satisfy $n(\mathcal L_a)\leqq 1$. This finishes the proof.
\end{proof}

\begin{remark} Notice that it is possible to have, in Lemma \ref{delta}, that $n(\mathcal L_a)=0$ for some $W_Z$-domain. Indeed, by considering $Z_a=- \psi''_a(0)/\psi'_a(0)$ we have $\psi'_a\in W_{Z_a}$ and $\mathcal L_a\psi'_a=0$. Thus, as $\psi'_a<0$ on $[0,\infty)$ follows from Sturm-Liouville oscillation theory on half-lines that zero is the smallest eigenvalue for $(\mathcal L_a, W_{Z_a})$ (note that the essential spectrum for $\mathcal L_a$ is $\sigma_{\mathrm{ess}}(\mathcal L_a)=[1/c_2^2, \infty)$).
\end{remark}

Upon application of Lemmata \ref{split} and \ref{delta} we obtain the following result.

 \begin{lemma}
 \label{morse1} 
 Consider the pair $(\mathcal L, D_{Z,0})$. Suppose that there are $(f_i,g_i)\in D_{Z,0}$ with $g_i(0)\neq 0$, $i=1,2$, such that $\mathcal L(f_i,g_i)^\top=\lambda_i (f_i,g_i)^\top$ with $ \lambda_i<0$. Then, $\lambda_1= \lambda_2$ and $g_1=\alpha g_2$ on $[0,\infty)$.
\end{lemma}
\begin{proof} 
Suppose $ \lambda_1< \lambda_2$. By Lemma \ref{split} follows $\mathcal L_ag_i= \lambda_i g_i$, $g_i\in W_Z-\{0\}$. Therefore, $n(\mathcal L_a)\geqq 2$, which is a contradiction by Lemma \ref{delta}. Thus, $ \lambda_1= \lambda_2$ and by Lemma \ref{delta} we need to have that $g_1$ and $g_2$ are linearly dependent.
\end{proof}

\begin{remark} We note from Lemma \ref{split} that Lemma \ref{morse1} implies $\mathcal L_1f_i= \lambda_1 f_i$ with $f_i$ satisfying periodic boundary conditions; hence, it is possible for $\lambda_1$ to be a double eigenvalue. 
\end{remark}

\begin{theorem}
\label{morse2} 
Consider the pair $(\mathcal L, D_{Z,0})$ determined by {\it a priori} $(\phi_1, \psi_a)$ positive single-lobe kink state, $a\in \mathbb R$.  Then, the Morse index satisfies $n(\mathcal L)\leqq 1$. Therefore, we have also that for the pair  $(\mathcal L_Z, D_{Z})$, $n(\mathcal L_Z)\leqq 1$.
\end{theorem}
\begin{proof} The proof follows the strategy as in the proof of Theorems 3.1 and 3.2 in \cite{Angloop}. So, by the sake of completeness, we highlight the main points of the analysis. Indeed, without loss of generality, suppose $n(\mathcal L)= 2$  and consider $\lambda_0, \lambda_1$ being the negative eigenvalues (counting multiplicities).
  \begin{enumerate} 
 \item[(I)] Suppose initially $\lambda_0<\lambda_1$.  In the sequel we divide our analysis into several steps:
 \begin{enumerate}
\item[(A)] {\it{Perron-Frobenius property for  $(\mathcal L, D_{Z,0})$}}. For $\lambda_0=\text{inf}\; \sigma (\mathcal L)<0$, let $(\xi_{\lambda_0}, \zeta_{\lambda_0})$ be an associated eigenfunction to $\lambda_0$. Then, $\xi_{\lambda_0}$ and  $\zeta_{\lambda_0}$ can be chosen as positive functions so that, in addition, $\xi_{\lambda_0}$ is even. Indeed, we consider the quadratic form $\mathcal Q_{Z}$ associated to operator $\mathcal  L$ on  $D_{Z,0}$, namely, $\mathcal Q_{ Z}: D(\mathcal Q_{Z})\to \mathbb R$, with
\begin{equation}
\label{Q}
\mathcal Q_{ Z}(\xi, \zeta) =\int_{-L}^L (\xi')^2 + V_{\phi_1} \xi^2 dx + \int_{0}^{\infty} (\zeta')^2 + W_{\psi_a} \zeta^2 dx - Z|\zeta(0)|^2,
\end{equation}
$V_{\phi_1} = \cos(\phi_1) $, $W_{\psi_a}=\cos(\psi_a) $, and   $D(\mathcal Q_{Z})$ defined by 
 \begin{equation}\label{Qa}
   D(\mathcal Q_{Z})=\{(\xi,\zeta)\in X^1(-L,L): \xi(L)=\xi(-L)=\zeta(0)\}.
   \end{equation}
 \begin{enumerate}
\item[(1)] The profile $\zeta_{\lambda_0}$ is not identically zero: indeed, suppose that $\zeta_{\lambda_0}\equiv 0$; then $\xi_{\lambda_0}$ satisfies $\mathcal L_{1}\xi_{\lambda_0}(x)=\lambda_0 \xi_{\lambda_0} (x)$, $ x\in (-L,L)$, $ \xi_{\lambda_0}(L)=\xi_{\lambda_0}(-L)=0$ and $ \xi'_{\lambda_0}(L)=\xi'_{\lambda_0}(-L)$, and so from the Dirichlet condition and Oscillation Theorem of the Floquet theory, we need to have that $\xi_{\lambda_0}$ is odd. Then, by Sturm-Liouville theory there is an eigenvalue $\theta$ for $\mathcal L_{1}$, such that $\theta<\lambda_0$, with associated eigenfunction $ \chi>0$ on $(-L, L)$, and  $\chi(-L)=\chi(L)=0$. Now, let $\mathcal Q_{\mathrm{Dir}}$ be the quadratic form associated to $\mathcal L_{1}$ with Dirichlet domain, namely, $\mathcal Q_{\mathrm{Dir}}: H^1_0 (-L, L)\to \mathbb R$ and 
\begin{equation}\label{Q_d}
\mathcal Q_{\mathrm{Dir}}(f)=\int_{-L}^L (f')^2 + V_{\phi_1} f^2 dx.
 \end{equation} 
Then, $\mathcal Q_{\mathrm{Dir}}(\chi)=\mathcal Q_{Z}(\chi, 0)\geqq \lambda_0\|\chi\|^2$ and so, $\theta\geqq \lambda_0$. This is a contradiction. 

 \item[(2)]  $\zeta_{\lambda_0}(0)\neq 0$:  suppose  $\zeta_{\lambda_0}(0)=0$ and we consider the odd-extension $\zeta_{\mathrm{odd}}$ for $\zeta_{\lambda_0}$, and the even-extension $\psi_{\mathrm{even}}$ of the kink-profile $\psi_a$ on all the line. Then, $\zeta_{\mathrm{odd}}\in H^2(\mathbb R)$ and $\psi_{\mathrm{even}}\in H^2(\mathbb R-\{0\})\cap H^1(\mathbb R)$. Next, we consider the unfold operator 
\begin{equation}\label{Leven}
\widetilde{\mathcal L}=-c_2^2\partial_x^2 + \cos(\psi_{\mathrm{even}}),
 \end{equation} 
 on the $\delta$-interaction domain 
\begin{equation}\label{Ddelta}
   D_{\delta, \gamma}=\{f\in H^2(\mathbb R-\{0\})\cap H^1(\mathbb R): f'(0+)-f'(0-)=\gamma f(0)\},
   \end{equation}
for any $\gamma \in \mathbb R$. Then, from the extension theory for symmetric operators we have that the family $(\widetilde{\mathcal L}, D_{\delta, \gamma})_{\gamma \in \mathbb R}$ represents all the self-adjoint extensions of the symmetric operator $(\mathcal N_0, D(\mathcal N_0))$ defined by $
\mathcal N_0= \widetilde{\mathcal L}$, $D(\mathcal N_0)=\{f\in H^2(\mathbb R): f(0)=0\}$, 
because the deficiency indices of $(\mathcal N_0, D(\mathcal N_0))$ are given by $n_{\pm}(\mathcal N_0)=1$ (see Albeverio {\it et al.} \cite{Albe} and Proposition \ref{11} in Appendix below). Now, the even tail-profile $\psi_{\mathrm{even}}$  satisfies $\psi_{\mathrm{even}}'(x)\neq 0$ for all $x\neq 0$, and so from the relation
\begin{equation}\label{M0}
\mathcal N_0 f=-\frac{1}{\psi'_{\mathrm{even}}}\frac{d}{dx}\Big[c_2^2(\psi'_{\mathrm{even}})^2 \frac{d}{dx}\Big ( \frac{f}{\psi'_{\mathrm{even}}}\Big)\Big],\quad x\in \mathbb R-\{0\},
\end{equation}
we can see easily that $\langle \mathcal N_0 f, f\rangle\geqq 0$ for all $f\in D(\mathcal N_0)$. Then, from Proposition \ref{semibounded} in Appendix, we obtain that the Morse index for the family $(\widetilde{\mathcal L}, D_{\delta, \gamma})$ satisfies $n(\widetilde{\mathcal L})\leqq 1$, for all $\gamma \in \mathbb R$. Next, since $\zeta_{\mathrm{odd}}\in D_{\delta, \gamma}$, for any $\gamma$, and $\widetilde{\mathcal L} \, \zeta_{\mathrm{odd}}=\lambda_0 \zeta_{\mathrm{odd}}$ on $\mathbb R$, we have $n(\widetilde{\mathcal L})= 1$ and  $\lambda_0$ will be the smallest negative eigenvalue for $\widetilde{\mathcal L}$ on $\delta$-interactions domains. Then, by Theorem \ref{PFpro} in  Appendix  (Perron Frobenius property for $\widetilde{\mathcal L}$), $\zeta_{\mathrm{odd}}$ needs to be positive which is a contradiction. Therefore, $\zeta_{\lambda_0}(0)\neq 0$.

  \item[(3)] $\zeta_{\lambda_0}$ and  $\xi_{\lambda_0}$ are positive profiles: from $\zeta_{\lambda_0}(0)\neq 0$ and Lemma \ref{split}, the pair $(\zeta_{\lambda_0}, \lambda_0)$ satisfies the $(\delta BP)$-boundary problem in \eqref{La1} and so by Lemma \ref{delta} we obtain $n(\mathcal L_a)=1$. Therefore, by oscillation theory we can choose $\zeta_{\lambda_0}>0$ on $(0, \infty)$ (indeed, we can use the Perron-Frobenius property in Theorem \ref{PFpro} for the unfold operator $(\widetilde{\mathcal L}, D_{\delta, -Z})$ in \eqref{Leven2}-\eqref{Ddelta2} and also deduce the positive property of $\zeta_{\lambda_0}$). Next, we study $\xi_{\lambda_0}$. From Lemma \ref{split}, the pair $(\xi_{\lambda_0}, \lambda_0)$ satisfies the $(P BP)$-boundary problem in \eqref{La1}. We denote by  $\eta_0$ the first eigenvalue (simple)  for (PBP) in  \eqref{La1}, then $\lambda_0\geqq \eta_0$. {\it We are going to show that} $\lambda_0= \eta_0$. Indeed, we consider the quadratic form, $\mathcal Q_{\mathrm{per}}: H_{\mathrm{per}}^1(-L,L)\to \mathbb R$,
\begin{equation}\label{QRC}
\mathcal Q_{\mathrm{per}}(h)=\int_{-L}^L \big( (h')^2 + V_{\phi_1} h^2 \big) \, dx.
 \end{equation} 
Next, for $h\in H_{\mathrm{per}}^1(-L,L)$ fixed, define $\chi=\nu \zeta_{\lambda_0}$ with $\nu\in \mathbb R$ being chosen such that $\chi(0)=\nu \zeta_{\lambda_0}(0)=h(L)$. Then, $(h,\chi)\in D(\mathcal Q_{ Z})$ in \eqref{Qa}. Now, by using that $\mathcal L_{a} \zeta_{\lambda_0}=\lambda_0 \zeta_{\lambda_0}$ on $(0, \infty)$, we get immediately that
\begin{equation}\label{QQ}
\begin{aligned}
\mathcal Q_{\mathrm{per}}(h)
&= \mathcal Q_{ Z} (h,\chi)- \langle \mathcal L_{a} \chi, \chi\rangle = \mathcal Q_{ Z} (h,\chi) - \lambda_0\|\chi\|^2\\
&\geqq \lambda_0 [\|h\|^2+ \|\chi\|^2]-\lambda_0  \|\chi\|^2=\lambda_0  \|h\|^2.
\end{aligned} 
 \end{equation} 
Then, $\eta_0\geqq \lambda_0 $ and so $\eta_0= \lambda_0 $. Therefore, from Floquet theory follows that we can choose $\xi_{\lambda_0}$ positive on $[-L,L]$. Even more, since  $V_{\phi_1}$ is even, we obtain that $\xi_{\lambda_0}$ is even.
 \end{enumerate}

 \item[(B)] Let  $(f_{{1}}, g_{1}) \in D_{Z, 0}$ be  an associated eigenfunction for $\lambda_{{1}}$.  We shall prove that $g_{1}(0+)\neq 0$.
 \begin{enumerate}
\item[(1)] Suppose that $g_{{1}}\equiv 0$: then $f_{{1}}(-L)=f_{{1}}(L)=0$, $f'_{{1}}(-L)=f'_{{1}}(L)$, and so  $f_1$ will be odd (see step (1) in item (A) above). Now, our single-lobe profile  $\phi_1$ satisfies   $\mathcal L_1 \phi'_1(x)=0$,  for $x\in (-L,L)$, $\phi'$ is odd, $\phi'(x)>0$ for $x\in [-L, 0)$,  thus since $\lambda_{1}<0 $ we obtain  from the Sturm oscillation theorem that there is $r\in (-L, 0)$ such that $\phi'(r)=0$, it which is a contradiction. Then, $g_{1}$ is non-trivial. 

\item[2)] Now suppose $g_{1}(0+)=0$:  we consider the  odd-extension $g_{1, \mathrm{odd}}\in H^2(\mathbb R)$ of $g_{1}$ on all the line and the extension operator $\widetilde{\mathcal L}$ in \eqref{Leven} on the $\delta$-interaction domains $D_{\delta, \gamma}$ in \eqref{Ddelta}. Then, $g_{1, \mathrm{odd}}\in D_{\delta, \gamma}$ for any $\gamma$ and by Perron-Frobenius property (Theorem \ref{PFpro}) we will obtain  $n( \widetilde{\mathcal L})\geqq 2$, it which is a contradiction because $n( \widetilde{\mathcal L})\leqq 1$ for all $\gamma$ (see item (A)-(2) above).  
\end{enumerate}  
\item[(C)] By items (A),  (B) and Lemma \ref{morse1} we have that $\lambda_{1}=\lambda_{0}$, which is a contradiction.
\end{enumerate}   
\item[(II)] Suppose $\lambda_0$ is a double eigenvalue: $\lambda_0=\lambda_1$. Then by Perron-Frobenius property in item (I)-(A) there holds $\xi_{\lambda_0}, \zeta_{\lambda_0}>0$. Now, for $(f_1, g_1) \in D_{Z, 0}$  an associated eigenfunction for $\lambda_1$, we have $(f_1, g_1) \perp (\xi_{\lambda_0}, \zeta_{\lambda_0})$. Next, by  item (I)-(B),  Lemma \ref{split} and $\lambda_0=\eta_0$, being $\eta_0$ the first eigenvalue (simple) for (PBP) in  \eqref{La1}, we have that $f_1=r\xi_{\lambda_0}$ and $g_1=s \zeta_{\lambda_0}$ with $r=s\neq 0$ by continuity at zero. Hence, we arrive to a contradiction from the orthogonality property of the eigenfunctions. This finishes the proof of the Theorem.
\end{enumerate}  
\end{proof}

Next, we will give sufficient conditions for obtaining $n(\mathcal L_Z)= 1$, with $\mathcal L_Z$ defined in \eqref{sg5} and $(\phi_1, \phi_2)\equiv (\phi_{1,k},  \phi_a)$, with the  single-lobe state $\phi_{1,k}$  in \eqref{formula1} or \eqref{La2}. Initially, we consider the mapping $\mathcal T(\theta)=\theta \cos\theta - \sin \theta$ for $\theta\in [0,2\pi]$. Then $\mathcal T(\theta)\leqq 0$ for $\theta\in [0, \theta_0]$ where $\theta_0\approx 4.4934$ is the unique zero of $\mathcal T$ on the interval $(0, 2\pi)$. We note that $\theta_0\in (\pi, \frac{3\pi}{2})$ (see Figure \ref{fig9}). 

\begin{figure}[h]
 \centering
\includegraphics[angle=0,scale=0.55]{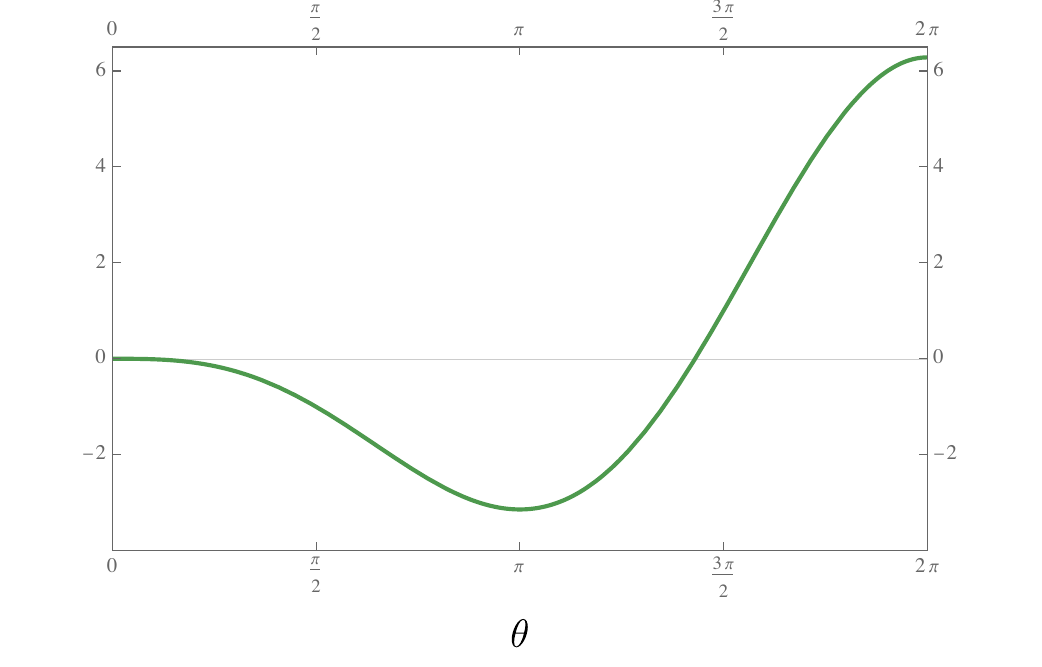}
\caption{\small{Graph of $\mathcal T(\theta)=\theta \cos \theta -\sin \theta$, $\theta\in [0,2\pi]$.} } \label{fig9}
\end{figure}

Thus, with the notations above we get the following results.
\begin{lemma}
\label{ineq} 
Consider the single-lobe state $\phi_{1,k}$ determined by $k\in (0,1)$ in \eqref{formula1}  or \eqref{La2}. Then, for $k$ such that 
\[
k^2\leqq \frac{1+ \cos(\theta_0)}{2}\approx 0.3914,
\]
we have  that $\phi_{1,k}(0)\in (\pi, \theta_0]$ and, therefore, 
\begin{equation}
\label{sineq}
\mathcal T( \phi_{1,k}(x))= \phi_{1,k}(x) \; \cos(\phi_{1,k}(x))- \sin(\phi_{1,k}(x))\leqq 0,\qquad \text{for all}\;\; x\in [-L,L].
\end{equation}
Moreover, for the center solution $\phi_1(x)\equiv \pi$, relation in \eqref{sineq} is trivial.
\end{lemma}
\begin{proof} By \eqref{formula1} or \eqref{La2} and $\sn^2(K(k);k)=1$, for all $k$, we have $\phi_{1,k}(0)=2\pi - \arccos(-1+2k^2)\leqq \theta_0$ if and only if 
\[
\cos (2\pi-\arccos(2k^2-1))=2k^2-1\leqq \cos \theta_0. 
\]
We recall that $\phi_{1,k}(0)\in (\pi, 2\pi)$. Lastly, since $\phi_{1,k}(x)\in (0, \phi_{1,k}(0)]$ for all $x\in [-L,L]$, we get immediately \eqref{sineq}. This finishes the proof.
\end{proof}

\begin{theorem}
\label{Morse=1} 
Consider  the  single-lobe state $\phi_{1,k}$  in \eqref{formula1} or \eqref{La2} determined  by $k\in (0,1)$ such that 
\[
k^2\leqq \frac{1+ \cos(\theta_0)}{2} \approx 0.3914,
\]
and such that for this value of $k$ we have the existence of a single-lobe kink state $(\phi_{1,k}, \phi_{2,a})\in D_Z$, with $a=a(k)$. Then for  $(\mathcal L_Z, D_{Z})$ defined in \eqref{sg5} with $(\phi_1, \phi_2)\equiv (\phi_{1,k}, \phi_{2,a})$, we have that the Morse index satisfies $n(\mathcal L_Z)=1$. In the case of the  degenerate single-lobe kink state $(\pi, \phi_{2,0})\in D_Z$ with $Z=\frac{2}{\pi c_2}$ we also have that $n(\mathcal L_Z)=1$.
 \end{theorem}

\begin{proof} For the   single-lobe kink state $(\phi_{1,k}, \phi_{2,a})\in D_Z$, we have for $x>L$ that $\phi_{2,a}(x)\leqq \phi_{2,a}(L)=\phi_{1,k}(L)<\theta_0$. Then, for all $x>L$ we obtain $\mathcal T(\phi_{2,a}(x)) = -\sin(\phi_{2,a}(x))+ \phi_{2,a}(x) \cos(\phi_{2,a}(x))< 0$. Therefore, by Lemma \ref{ineq} we obtain
$$
\langle \mathcal L_Z (\phi_{1,k}, \phi_{2,a})^\top, (\phi_{1,k}, \phi_{2,a})^\top\rangle=\int_{-L}^L \mathcal T( \phi_{1,k}(x))\phi_{1,k}(x)dx +\int_L^{\infty} \mathcal T(\phi_{2,a}(x))\phi_{2,a} (x)dx <0.
$$
 Then from minimax principle we  arrive at $n(\mathcal L_Z)\geqq 1$ (similarly for the degenerate case  $(\pi, \phi_{2,0})$). Thus, by Theorem \ref{morse2} we finish the proof.
\end{proof}

\begin{remark}
\label{smallest} 
\begin{enumerate}
\item[(1)] Notice that by Theorem \ref{Morse=1} and item (A) in the proof of Theorem \ref{morse2}, we obtain that the unique negative eigenvalue for $(\mathcal L_Z, D_{Z})$ coincides with the smallest negative eigenvalue for the eigenvalues problems (PBP) and ($\delta$BP) in \eqref{La1}.
\item[(2)] For $k^2>\frac{1+\cos(\theta_0)}{2}$ our study does not provide an accurate picture of stability properties of the possible profiles $(\phi_{1,k}, \phi_{2,a})$. 
\end{enumerate}
\end{remark}

In what follows we exactly determine which single-lobe kink state $(\phi_{1,k}, \phi_{2,a})\in D_Z$, characterized by 
Propositions \ref{exis1} and \ref{exis3}, will induce that the Morse index of $\mathcal L_Z$ is equal to one. As we will see the value of $L$, $k$, $c_i$ and the sign of $Z$ play a fundamental role.

\begin{theorem}
\label{exemplos} 
Consider  the  single-lobe state $\phi_{1,k}$  in \eqref{formula1} such that the single-lobe kink state $(\phi_{1,k}, \phi_{2,a})\in D_Z$, $a=a(k)<0$ (see formula \eqref{2form-a}), it is determined by Proposition \ref{exis1}. Then the following conditions on $L$,  $c_i$, $Z$, and with $k$ satisfying 
\[
k^2\leqq \frac{1+ \cos(\theta_0)}{2}\equiv k^2_\ell  \approx 0.3914,
\]
imply that the operator $\mathcal L_Z$ in \eqref{sg5} associated to $(\phi_{1,k}, \phi_{2,a(k)})$ satisfies $n(\mathcal L_Z)=1$. To state the conditions we denote $\mathcal I=K(k_\ell)\approx 1.77160$ and for $L, c_1$ fixed but arbitrary we obtain:
\begin{enumerate}
\item[(1)] Let $\frac{L}{c_1}>\frac{\pi}{2}$ and $k_0\in (0,1)$ such that $K(k_0)=\frac{L}{c_1}$. Then for $\frac{L}{c_1}<\mathcal I$ we consider the family of solutions
\[
k \in (k_0, k_\ell)\to (\phi_{1,k}, \phi_{2,a}),
\]
with $a=a(k)$ determined by formula in \eqref{2form-a}. Then for $H(k)$ defined in \eqref{form-k4} we have:
\begin{enumerate}
\item[(i)] For $Z>0$ with $Z\in (H(k_\ell), H(k_0))$, $H(k_0)=\frac{4}{\pi c_1}\big[\frac{c_1}{2c_2}-k_0\big]$, and  with $c_2$ satisfying $\frac{c_1}{2c_2}\geqq \tanh(\frac{L}{c_1})$  we have $n(\mathcal L_Z)=1$.
\item[(ii)] let  $Z>0$ and  $g(k)=k \sn(L/c_1; k)$, $k\in (k_0, 1)$, suppose that $c_2$ satisfies $g(k_\ell)<\frac{c_1}{2c_2}< tanh(\frac{L}{c_1})$. Then for $Z\in (H(k_\ell), H(k_0))$ we have $n(\mathcal L_Z)=1$.
\item[(iii)] Let $Z<0$ and suppose that $c_2$ satisfies $\frac{c_1}{2c_2}<k_0$. Then for $Z\in (H(k_0), H(k_\ell))$ we have $n(\mathcal L_Z)=1$.
\item[(iv)]  Let $Z<0$ and suppose that $c_2$ satisfies $\frac{c_1}{2c_2}=k_0$. Then for $Z\in (H(k_\ell), 0)$ we have $n(\mathcal L_Z)=1$.
\item[(v)]  Let  $Z<0$, $g(k)=k \sn(L/c_1; k)$, $k\in (k_0, 1)$ and suppose that $c_2$ satisfies $g(k_\ell)>\frac{c_1}{2c_2}>k_0$. Then for $Z\in (H(k_\ell), 0)$ we have $n(\mathcal L_Z)=1$.
\item[(vi)] Let  $Z=0$ and suppose that $c_2$ satisfies $\frac{c_1}{2c_2} \in (k_0, k_\ell \sn(L/c_1; k_\ell))$. Then  $n(\mathcal L_0)=1$.
\end{enumerate}
\item[2)] Let $\frac{L}{c_1}\leqq \frac{\pi}{2}$ and consider the family of solutions
\[
k\in (0, k_\ell)\to (\phi_{1,k}, \phi_{2,a}),
\]
with $a=a(k)$ determined by formula in \eqref{2form-a}. Then, for $H(k)$ defined in \eqref{form-k4} we have the following:
 \begin{enumerate}
 \item[(i)] For  $Z>0$ we have two cases:
 \begin{enumerate}
\item[(a)] Suppose $c_2$ satisfies $\frac{c_1}{2c_2}\geqq \tanh(\frac{L}{c_1})$. Then  for $Z\in (H(k_\ell), \frac{2}{\pi c_2})$ we have $n(\mathcal L_Z)=1$.
 \item[(b)] Suppose $c_2$ satisfies $\frac{c_1}{2c_2}< \tanh(\frac{L}{c_1})$ and let $\beta\in (0,1)$ be such that $g(\beta)=\frac{c_1}{2c_2}$. Then if $k_\ell<\beta$ we have for all $Z\in (H(k_\ell), \frac{2}{\pi c_2})$ that $n(\mathcal L_Z)=1$.
\end{enumerate}
 \item[(ii)]  Let  $Z<0$ and suppose that $c_2$ satisfies $k_\ell \sn(L/c_1; k_\ell)> \frac{c_1}{2c_2}$. Then for $Z\in (H(k_\ell), 0)$ we have $n(\mathcal L_Z)=1$.
\item[(iii)] Let  $Z=0$ and suppose that $c_2$ satisfies $\frac{c_1}{2c_2} \in (0, k_\ell \sn(L/c_1; k_\ell))$. Then  $n(\mathcal L_0)=1$.
\end{enumerate}
\end{enumerate}
\end{theorem}
\begin{proof}  By Proposition \ref{exis1} and its proof we have the following cases:
\begin{enumerate}
\item[(1)] Suppose $\frac{L}{c_1}>\frac{\pi}{2}$:
\begin{enumerate}
\item[(i)] For $Z>0$ and $Z\in (H(k_\ell), H(k_0))$ it follows from the strictly-decreasing property of the mapping $k\in (k_0, k_\ell)\to H(k)$ and the condition on $c_2$ that there is $k_Z\in (k_0, k_\ell)$ such that $ H(k_Z)=Z$. Therefore, $\Phi_{k_Z}=(\phi_{1,k_Z}, \phi_{2,a})\in D_Z$, $a=a(k_Z)$ and the operator $\mathcal L_Z$ associated to $\Phi_{k_Z}$ satisfies $n(\mathcal L_Z)=1$.
\item[ii)] Let $Z>0$. By the conditions on $c_2$ and $g(k_0)=k_0<g(k_\ell)$ we get $\frac{c_1}{2c_2}\in (k_0, \tanh(\frac{L}{c_1}))$ and $k_\ell\in (k_0, \gamma)$ with $g(\gamma)=\frac{c_1}{2c_2}$ and $H(\gamma)=0$. Therefore, there is $k_Z\in (k_0, k_\ell)$ such that $ H(k_Z)=Z$ and the profile $\Phi_{k_Z}=(\phi_{1,k_Z}, \phi_{2,a})\in D_Z$, $a=a(k_Z)$. Hence, $\mathcal L_Z$ associated to $\Phi_{k_Z}$ satisfies $n(\mathcal L_Z)=1$.
\item[(iii)] Let $Z<0$. By the condition on $c_2$ and the strictly-increasing property of the mapping $k\in (k_0, k_\ell)\to H(k)$ there is $k_Z\in (k_0, k_\ell)$ such that $ H(k_Z)=Z$. Thus, we conclude that $n(\mathcal L_Z)=1$.
\item[(iv)]  Let $Z<0$.  By the condition on $c_2$ we obtain $H(k_0)=0$. Thus, as $H(k_\ell)\geqq m_0 = \min_{k\in (k_0, 1)} H(k)$ we have for $Z\in (H(k_\ell), 0)$ that there is $k_Z\in (k_0, k_\ell)$ such that $ H(k_Z)=Z$. Then,  $n(\mathcal L_Z)=1$.
\item[(v)]  Let $Z<0$.  By the condition on $c_2$ and $k \sn(L/c_1; k)<\tanh(\frac{L}{c_1})$, we obtain $\frac{c_1}{2c_2}\in \big(k_0, \tanh(\frac{L}{c_1})\big)$. Next, let $\beta\in (k_0, k_\ell)$ be such that $g(\beta)=\frac{c_1}{2c_2}$ and, thus, $H(\beta)=0$. Then for $Z\in (H(k_\ell), 0)$ there is $k_Z\in (\beta, k_\ell)$ such that $ H(k_Z)=Z$. We conclude that $n(\mathcal L_Z)=1$.
\item[(vi)] Let  $Z=0$.  By the condition on $c_2$ there is $k\in (k_0, k_\ell)$ such that  $k \sn(L/c_1; k)=\frac{c_1}{2c_2}$. Then $H(k)=0$ and  
the operator $\mathcal L_0$ associated to $(\phi_{1,k}, \phi_{2,a(k)})$  satisfies $n(\mathcal L_0)=1$.
\end{enumerate}
\item[2)] Now suppose that $\frac{L}{c_1}\leqq \frac{\pi}{2}$: 
\begin{enumerate}
\item[(i)] For $Z>0$ we have the following cases:
 \begin{enumerate}
\item[(a)] Suppose that $c_2$ satisfies $\frac{c_1}{2c_2}\geqq \tanh(\frac{L}{c_1})$. Then $H(k)>0$ for all $k\in (0,1)$ with $H(0)=\frac{2}{\pi c_2}$,  $H(1)=0$ and $H'(k)<0$. Thus for $Z\in (H(k_\ell), \frac{2}{\pi c_2} )$ there is $k_Z\in (0, k_\ell)$ such that $ H(k_Z)=Z$. We conclude that $n(\mathcal L_Z)=1$. 
\item[(b)] Suppose that $c_2$ satisfies $\frac{c_1}{2c_2}< \tanh(\frac{L}{c_1})$ and let $\beta\in (0,1)$ be such that $g(\beta)=\frac{c_1}{2c_2}$. Then $H(\beta)=0$ and $H(k)>0$ for all $k\in (0,\beta)$ with $H(0)=\frac{2}{\pi c_2}$ and $H'(k)<0$. Thus, if $k_\ell<\beta$ we have for all $Z\in (H(k_\ell), \frac{2}{\pi c_2})$ that $n(\mathcal L_Z)=1$.
\end{enumerate}
\item[(ii)]  For $Z<0$ and $c_2$ satisfying $k_\ell \sn(L/c_1; k_\ell)> \frac{c_1}{2c_2}$, let $\beta\in (0,1)$ be such that $g(\beta)=\frac{c_1}{2c_2}$. Then, $H(\beta)=H(1)=0$, $H(k)<0$ for $k\in (\beta, 1)$ and for $Z\in (H(k_\ell), 0)$ we have $n(\mathcal L_Z)=1$.
 \item[(iii)] For $Z=0$ and $c_2$ satisfying $\frac{c_1}{2c_2} \in (0, k_\ell \sn(L/c_1; k_\ell))$, we obtain  $k\in (0, k_\ell)$ such that $k \sn(L/c_1; k)=\frac{c_1}{2c_2}$ (we recall that $g(k)=k \sn(L/c_1; k)$ is strictly-increasing) and hence $H(k)=0$. Then $n(\mathcal L_0)=1$.
\end{enumerate}
\end{enumerate}
This finishes the examination of all the cases and the proof.
\end{proof}

\begin{remark}
\label{nopuede} 
For  $\frac{L}{c_1}>\frac{\pi}{2}$ and $k_0\in (0,1)$ such that $K(k_0)=\frac{L}{c_1}$, we obtain that for $\frac{L}{c_1}\geqq I=K(k_\ell)$ that $k_0>k_\ell$ and for the family $k\in (k_0, 1)\to  (\phi_{1,k}, \phi_{2,a})$ we can not apply Theorem \ref{Morse=1}.
\end{remark}

\begin{theorem}\label{exemplos2}  Let $L, c_1$ be arbitrary but fixed with $\frac{L}{c_1}>\frac{\pi}{2}$ and $k_0$ such that $K(k_0)=\frac{L}{c_1}$. Let $c_2$ satisfy
\[
\frac{c_1}{2k_0}< c_2.
\]
Then, for $k\in (0, k_0)$, by Proposition \ref{exis3} there exists a shift-value $a=a(k)>0$ (see Formula \eqref{2form-a}), such that the  single-lobe kink state satisfies $(\phi_{1,k}, \phi_{2,a})\in D_0$, with $\phi_{1,k}$  in \eqref{La2}. Then the following restrictions on $L$,  $c_1$ and with $k$ satisfying 
\[
k^2\leqq \frac{1+ \cos(\theta_0)}{2}\equiv k^2_\ell  \approx 0.3914,
\]
guarantee that the operator $\mathcal L_0$ in \eqref{sg5} associated to $(\phi_{1,k}, \phi_{2,a(k)})$  satisfies $n(\mathcal L_0)=1$. In order to state these conditions, let  $\mathcal I=K(k_\ell)\approx 1.77160$; then we have:
\begin{enumerate}
\item[(1)] For  $\frac{L}{c_1}< \mathcal I$, the self-adjoint operator $\mathcal L_0$ in \eqref{sg5} determined by the family of solutions $ k\in (0, k_0)\to (\phi_{1,k}, \phi_{2,a})$ satisfies $n(\mathcal L_0)=1$.
\item[(2)] For  $\frac{L}{c_1}> \mathcal I$, the self-adjoint operator $\mathcal L_0$ in \eqref{sg5} determined by the family of solutions $ k\in (0, k_\ell)\to (\phi_{1,k}, \phi_{2,a})$ satisfies $n(\mathcal L_0)=1$.
\end{enumerate}
\end{theorem}
\begin{proof} 
The proof is an immediate consequence of Theorem \ref{Morse=1}.
\end{proof}

\subsection{Trivial kernel for $\mathcal L_{Z}$}

The purpose of this subsection is to provide sufficient conditions for $\mathcal L\equiv ( \mathcal L_{1}, \mathcal L_{a})$ to have a trivial kernel,  with $ \mathcal L_{1}, \mathcal L_{a}$ defined in \eqref{L} and  $(\phi_1, \psi_a)$ being an {\it a priori}  positive single-lobe kink  with 
\[
\psi_a(x)= 4\arctan\Big[e^{-\frac{1}{c_2} (x+a)}\Big], \qquad x>0.
\]

\begin{theorem}\label{ker} 
Let  $Z\in (-\infty, \frac{2}{\pi c_2})$ be an admissible strength value to obtain a single-lobe kink state $(\phi_{1}, \psi_{a})\in D_{Z,0}$ and 
define the quantity
\[
\alpha_a\equiv \frac{\psi''_{a}(0)}{\psi'_{a}(0)}=-\frac{1}{c_2} \tanh\Big(\frac{a}{c_2}\Big).
\]
Then, the kernel associated to $\mathcal L=(\mathcal L_{1}, \mathcal L_{a})$ on $D_{Z,0}$ is trivial in the following cases:
\begin{enumerate}
\item[(1)] for $\alpha_a\neq -Z$, or
\item[(2)] for $\alpha_a= -Z$ in the case of admissible  $Z$ satisfying $Z\leqq 0$. 
\end{enumerate}
In particular, 
\begin{enumerate}
\item[(i)] for the case $Z\neq 0$ and $a<0$, the kernel of $\mathcal L$ is trivial;
\item[(ii)] for the case $Z<0$ and $a>0$, the kernel of $\mathcal L$ is trivial; and,
\item[(iii)] for the case $Z=0$, the kernel of $\mathcal L$ is trivial for any $a\in \mathbb R-\{0\}$.
\end{enumerate}
In the case of a degenerate single-lobe kink state, $(\pi, \phi_{2,0})\in D_Z$ with $Z=\frac{2}{\pi c_2}$, we also have that $\ker(\mathcal L_Z)$ is trivial.
\end{theorem}

Next, as an application of Theorem \ref{ker} and from Theorems \ref{exemplos} and \ref{exemplos2}, we establish our main result about the Morse and nullity indices for $\mathcal L_Z$. 
\begin{theorem}
\label{MK} 
Let  $Z\in (-\infty, \frac{2}{\pi c_2})$ be an admissible strength value to obtain a single-lobe kink state $(\phi_{1,k}, \psi_{2,a})\in D_{Z}$, given by  Propositions \ref{exis1} and \ref{exis3}. Then, under the restrictions on $L, c_1, c_2, k$ in Theorem \ref{exemplos} and \ref{exemplos2}, we obtain $n(\mathcal L_Z)=1$ and $\ker(\mathcal L_Z)=\{\bf{0}\}$. In the case of a degenerate single-lobe kink state $(\pi, \phi_{2,0})\in D_Z$ with $Z=\frac{2}{\pi c_2}$, we also have that $\ker(\mathcal L_Z)$ is trivial and $n(\mathcal L_Z)=1$.
\end{theorem}
\begin{proof} 
The proof is immediate from Theorems \ref{exemplos} and \ref{exemplos2}, and items $(i)$ and $(iii)$ in Theorem \ref{ker}.
\end{proof}

\begin{proof}[Proof of Theorem \ref{ker}]
The proof follows a similar  analysis as in Angulo \cite{Angloop, Angtad}. By convenience of the reader we
provide the main point in the analysis. Thus, let $(f, h)\in D_{Z,0}$ such that $\mathcal L(f, h)^\top= {\bf{0}}$. Then  $\mathcal L_ah(x)=0$ for  $x>0$. Next, from classical Sturm-Liouville theory on half-line (see \cite{BerShu91}) and $\mathcal L_a\psi'_{a}=0$, we have $h=c\psi'_{a}$ on $(0,\infty)$ for $c\in \mathbb R$. Now, we consider the following cases: 

\begin{enumerate}
\item[(I)]  Suppose that $c=0$. This implies that $h\equiv 0$ and hence $f$ satisfies $\mathcal L_1f=0$ with  Dirichlet-periodic boundary conditions on $[-L,L]$, $f(L)=f(-L)=0$ and $f'(L)=f'(-L)$. Suppose that $f\neq 0$:
\begin{enumerate}
\item[(a)] From Sturm-Liouville theory for Dirichlet conditions we have that $f$ is even or odd (as zero is a simple eigenvalue and the potential $V_{\phi_1}$ is even).
\item[(b)] Suppose that $f$ is odd: then, since $\mathcal L_1\phi'_1=0$ and $\phi'_1$ is odd, there is $\gamma\in \mathbb R$ such that $f=\gamma\phi'_1$ (classical ODE analysis for second-order differential equations). Therefore, $0=f(L)=\gamma\phi'_1(L)$ implies $\gamma=0$ (because $\phi'_1(L)\neq 0$). Then, $f\equiv 0$.
\item[(c)] Suppose that $f$ is even: then from oscillation results of the Floquet theory, we need to have that zero it is not the first eigenvalue for $\mathcal L_1$. Therefore, $f$ needs to change of sign and the number of zeros of $f$ on $[-L, L)$ is {\it even}. Hence $f$ has an odd number of zeros on  $(-L, L)$ (as $f(-L)=0$). Then, $f(0)=0$ and $f'(0)=0$. Hence, $f\equiv 0$.
\end{enumerate}
Thus, from items $(a), (b)$ and $(c)$ we obtain $f\equiv 0$ and $\ker(\mathcal L)=\{\bf{0}\}$.
\item[(II)] Suppose that $\alpha_a\neq -Z$. Then $c=0$ and by the former  item (I) it follows that $\ker\mathcal L)=\{\bf{0}\}$. Indeed, if $c\neq 0$ then $h(0)\neq 0$ and from Lemma \ref{split} we get the eigenvalue problems in \eqref{La1}. In particular, $h$ and  $Z$ satisfy $h'(0)=-Zh(0)$ and hence $Z+\frac{\psi''_a(0)}{\psi'_a(0)}=0$, which is a contradiction.
\item[(III)]  Suppose that $\alpha_a=-Z$ with $Z\leqq 0$. Then $c=0$ and so by the previous item (I) it follows that $\ker(\mathcal L)=\{\bf{0}\}$. Suppose that $c\neq 0$; without loss of generality we can assume that $c<0$. Hence, $h(x)>0$ on $[0, \infty)$ and by the splitting eigenvalue result in Lemma \ref{split} we obtain that $(f, h)$ satisfies the separated eigenvalues problems (PBP)-($\delta$BP) in \eqref{La1} with $\lambda=0$. The $\delta$-interaction condition (also by hypotheses) implies 
\begin{equation}
\label{Z}
Z=-\frac{\psi''_a(0)}{\psi'_a(0)}.
 \end{equation} 
Next, we denote by $\{\eta_n\}$, $n\in \mathbb N_0$, the set of eigenvalues for the periodic-problem associated to $\mathcal L_1$ on $[-L, L]$, and by $\{\mu_n\}$, $n\in \mathbb N_0$, the set of eigenvalues for the Dirichlet-problem associated to $\mathcal L_1$ on $[-L, L]$. Then, by classical Sturm-Liouville theory (see Theorem 4.8.1 in \cite{Zettl05}) we have that  $\eta_0$ is   simple and, in particular, we have the following distribution of eigenvalues,
\begin{equation}\label{inequal2}
 \eta_0<\mu_0<\eta_1\leqq \mu_1\leqq \eta_2<\mu_2<\eta_3.
\end{equation}

Now we prove that $\eta_1=0$ in \eqref{inequal2} and that it is simple. Indeed, we examine all the cases:
 \begin{enumerate}
\item[(a)]  Suppose that $0> \mu_1$. We know that $\mathcal L_1 \phi_1'=0$, $\phi'_1$ is odd and strictly decreasing on $[-L, 0]$. Now, since the eigenfunction $\chi$ associated to $\mu_1$ is odd and $\chi(-L)=\chi(0)=0$, it follows from the Sturm Comparison Theorem (cf. \cite{Teschl12,Zettl05}) that $\phi'_1$ needs to have one zero on $(-L, 0)$. But this is impossible and we have a contradiction.
 \item[(b)]  Suppose that $\mu_1=0$. Then we have the existence of one odd-eigenfunction $\chi$  for $\mu_1$ ($\mathcal L_1 \chi=0$,  $\chi(-L)=\chi(L)=0$). Hence, by classical arguments from the theory of second-order differential equations, we get $\chi= \gamma \phi'_1$ and  $0=\chi (L)=\gamma \phi'_1(L)$. Hence, $\gamma=0$ and $\chi\equiv 0$. Again, this is not possible and we reach a contradiction.
 \item[(c)] From the previous items (a) and (b) it follows that $\mu_1>0$ and, by \eqref{inequal2} and Lemma \ref{split} ($\mathcal L_1 f=0$),  we need to have $\eta_1=0$ and that the zero is simple. Moreover, the eigenfunction $f$ is even (for $f$ being  odd, we have $f(L)=-f(-L)=-f(L)$ and so $h(0)=f(L)=0$), which is a contradiction.
 \end{enumerate}
 
Next, by Floquet theory, $f$ has exactly two different zeros $-\beta, \beta$ ($\beta>0$) on $(-L,L)$. Therefore,  $f(0)<0$, $f(\pm L)=h(0)>0$). Now, we consider the Wronskian function (constant) of $f$ and $\phi'_1$ (we recall $\mathcal L_1 f=\mathcal L_1 \phi'_1=0$), 
\[
W(x)=f(x)\phi''_1(x)-f'(x)\phi'_1(x)\equiv C,\quad \text{for all}\;\; x\in [-L,L].
\]
 Then, $C=f(0)\phi''_1(0)>0$. Now, by \eqref{Z}, $f'(L)=0$, $f(L)=h(0)$, and from \eqref{trav21} we obtain
\begin{equation}
\label{final}
C=f(L)\phi''_1(L)=h(0)\psi''_{a}(0)=-cZ[\psi'_{a}(0)]^2\leqq 0,
\end{equation}
with  $Z\leqq 0$. Therefore we get a contradiction by assuming that $c\neq 0$.
\end{enumerate}
The statements $(i)-(iii)$ are immediate from the explicit formula for $\alpha_a$. This finishes the proof.
\end{proof}

Next, we establish our main result about the  instability of a family of single-lobe kink states for the sine-Gordon model on a tadpole graph.
\begin{theorem}
\label{unstable} 
Under the restrictions on $Z, L, c_1, c_2$ and $k$ given by Propositions \ref{exis1} and \ref{exis3} and by Theorems \ref{exemplos}, \ref{exemplos2} and \ref{MK}, every member of the smooth family of single-lobe kink states, $k\to (\phi_{1,k}, \phi_{2,a(k)})\in D_{Z}$,  is spectrally and nonlinearly unstable for the sine-Gordon model \eqref{sg2} on a tadpole graph. The degenerate single-lobe kink state $(\pi, \phi_{2,0})\in D_Z$, with $Z=\frac{2}{\pi c_2}$ and 
\[
\phi_{2,0}(x) = 4 \arctan \big[e^{-\frac{1}{c_2}(x-L)}\big], \qquad x\in [L, \infty),
\]
is also spectrally and nonlinearly unstable for the sine-Gordon model \eqref{sg2} on a tadpole graph.
\end{theorem}

\begin{proof}  
The proof of the linear instability property of the single-lobe kink states $(\phi_{1,k}, \phi_{2,a(k)})$ is immediate from Theorem \ref{MK} and the instability criterion in Theorem \ref{crit}. Now, since the mapping data-solution for the sine-Gordon model on the energy space $\mathcal E(\mathcal G)\times L^2(\mathcal G)$ is at least of class $C^2$ (indeed, it is smooth) by Theorem \ref{IVP}, it follows that the linear instability property of $(\phi_{1,k}, \phi_{2,a(k)})$ is in fact of nonlinear type in the $\mathcal E(\mathcal G)\times L^2(\mathcal G)$-norm.  The reader is referred to \cite{AC2, ALN08, HPW82} for further information. This finishes the proof.
\end{proof}

\vskip0.1in

 \noindent
{\bf Data availability} 
\vskip0.1in
No data was used for the research described in the article.
 
 \vskip0.1in

\section*{Acknowledgements}
J. Angulo was supported in part by Universal/CNPq project and by CNPq/Brazil Grant. The work of R. G. Plaza was partially supported by DGAPA-UNAM, program PAPIIT, grant IN-104922 and by SECIHTI, grant CF-2023-G-122. R. G. Plaza is grateful to the IME-USP for their hospitality and to the FAPESP, Brazil (processo 2024/15816-0) for their financial support during a research stay in November 2024 when this work was carried out.

\appendix
\section{Appendix}
\label{secApp}

In this section we formulate some tools from the extension theory of symmetric operators by Krein and von Neumann which are suitable for our needs (see, for instance, Naimark  \cite{Nai67, Nai68} and Reed and Simon \cite{RS2} for further information).  In particular, we establish for $(\widetilde{\mathcal L}, D_{\delta, \gamma})_{\gamma \in \mathbb R}$ in \eqref{Leven} and \eqref{Ddelta}, a Perron-Frobenius result to be used in the  accurate Morse index estimate  $n(\mathcal L_Z)\leqq 1$ for $(\mathcal L_Z, D_Z)$ in \eqref{sg5} and \eqref{Domain0}.  
 \subsection{Extension theory for the Laplacian operator on a tadpole graph}
The following three results from the  extension theory of symmetric operators are classical and can be found in \cite{Nai67, Nai68, RS2}.
 
\begin{theorem}[von-Neumann decomposition]
\label{d5} 
Let $A$ be a closed, symmetric operator, then
\begin{equation}
\label{d6}
D(A^*)=D(A)\oplus\mathcal N_{-i} \oplus\mathcal N_{+i}.
\end{equation}
with $\mathcal N_{\pm i}= \ker (A^*\mp iI)$. Therefore, for $u\in D(A^*)$ and $u=x+y+z\in D(A)\oplus\mathcal N_{-i} \oplus\mathcal N_{+i}$,
\begin{equation}\label{d6a}
A^*u=Ax+(-i)y+iz.
\end{equation}
\end{theorem}

\begin{remark} 
The direct sum in (\ref{d6}) is not necessarily orthogonal.
\end{remark}

\begin{proposition}\label{11}
	Let $A$ be a densely defined, closed, symmetric operator in some Hilbert space $H$ with deficiency indices equal  $n_{\pm}(A)=1$. All self-adjoint extensions $A_\theta$ of $A$ may be parametrized by a real parameter $\theta\in [0,2\pi)$ where
	\begin{equation*}
	\begin{split}
	D(A_\theta)&=\{x+c\phi_+ + \zeta e^{i\theta}\phi_{-}: x\in D(A), \zeta \in \mathbb C\},\\
	A_\theta (x + \zeta \phi_+ + \zeta e^{i\theta}\phi_{-})&= Ax+i \zeta \phi_+ - i \zeta e^{i\theta}\phi_{-},
	\end{split}
	\end{equation*}
	with $A^*\phi_{\pm}=\pm i \phi_{\pm}$, and $\|\phi_+\|=\|\phi_-\|$.
\end{proposition}

\begin{proposition}\label{semibounded}
Let $A$  be a densely defined lower semi-bounded symmetric operator (that is, $A\geq mI$)  with finite deficiency indices, $n_{\pm}(A)=k<\infty$,  in the Hilbert space ${H}$, and let $\widehat{A}$ be a self-adjoint extension of $A$.  Then the spectrum of $\widehat{A}$  in $(-\infty, m)$ is discrete and  consists of, at most, $k$  eigenvalues counting multiplicities.
\end{proposition}

\subsection{Perron-Frobenius property }

In what follows, we establish  the Perron-Frobenius property for the family of self-adjoint operators
$(\widetilde{\mathcal L}, D_{\delta, \gamma})$, where
\begin{equation}
\label{Leven2}
\widetilde{\mathcal L}=-c_2^2\partial_x^2 + \cos(\psi_{\mathrm{even}}),
 \end{equation} 
$\psi_{\mathrm{even}}$ is the even-extension of the kink-profile $\psi_a(x)=4 \arctan (e^{-\frac{1}{c_2}(x+a)})$, $x>0$, on all the line, and  
\begin{equation}\label{Ddelta2}
   D_{\delta, \gamma}=\{f\in H^2(\mathbb R-\{0\})\cap H^1(\mathbb R): f'(0+)-f'(0-)=\gamma f(0)\}
   \end{equation}
for any $\gamma \in \mathbb R$. So, we start with the following  remark:  by Weyl's essential spectrum theorem (cf. Reed and Simon \cite{RS4}), we have that the essential spectrum, $\sigma_{\mathrm{ess}}(\widetilde{\mathcal L})$, of $\widetilde{\mathcal L}$ satisfies $\sigma_{\mathrm{ess}}(\widetilde{\mathcal L})=[1, \infty)$ for any $\gamma$.

\begin{theorem}[Perron-Frobenius property]
\label{PFpro} 
Consider the family of self-adjoint operators defined in \eqref{Leven2}-\eqref{Ddelta2}, $(\widetilde{\mathcal L}, D_{\delta, \gamma})_{\gamma \in \mathbb R}$. For $\gamma$ fixed, assume that $\beta=\inf \sigma(\widetilde{\mathcal L})<1$ is the smallest eigenvalue. Then, $\beta$ is simple, and its corresponding eigenfunction $\zeta_\beta$ is positive (after replacing $\zeta_\beta$ by $-\zeta_\beta$ if necessary) and even.
\end{theorem}

\begin{proof} This result follows by a slight twist of standard abstract Perron-Frobenius arguments (see Proposition 2 in Albert \emph{et al.} \cite{ABH87}) applied to Schr\"odinger operators with point interactions. We recommend the reader to see the proof of Theorem 6.7 in Angulo \cite{Angloop}.
\end{proof}

\subsection{Proof of Lemma \ref{exis2}}

\begin{proof}
In the sequel we determine the correct values of $k$ and $Z$ such that $(\phi_{1,k}, \phi_{2, a(k)})\in D_Z$, namely, such that we have  the relation 
$$
2\phi'_1(L)=\phi'_2(L)+Z\phi_2(L),
$$ 
with $\phi_1=\phi_{1,k}$, $\phi_2=\phi_{2, a(k)}$, and $a(k)$ satisfying relation in \eqref{2form-a}. We divide this study into several cases.
 
 \medskip
 
 \noindent $\bullet$ \textbf{Case (1)}. Let $\frac{L}{c_1}> \frac{\pi}{2}$. Consider the unique $k_0\in (0,1)$ such that $K(k_0)=\frac{L}{c_1}$. Then for $k>k_0$ it follows that $K(k)>\frac{L}{c_1}$. Now, from properties of the Jacobi elliptic functions (see \cite{ByFr71}), we obtain from \eqref{2form-a} (the continuity condition at the vertex $\nu=L$) the following limits: 
\begin{equation}
\label{limits}
\begin{aligned}
\lim_{k\to 1} \sech\Big(\frac{a(k)}{c_2}\Big) &= \lim_{k\to 1}\frac{k'}{\dn(L/c_1;k)} =\frac{0}{\sech(L/c_1)}=0,\\
\lim_{k\to k_0} \sech\Big(\frac{a(k)}{c_2}\Big) &=\frac{k'_0}{\dn(K(k_0);k_0)}=\frac{k'_0}{k'_0}=1.
\end{aligned}
\end{equation}
Therefore, $\lim_{k\to 1} a(k)=-\infty$ and $\lim_{k\to k_0} a(k)=0$ ($a(k_0)=0$). Next, we will see that $k\in (k_0, 1) \to a(k)$ is a strictly decreasing mapping. Indeed, from the formula $\sech^{-1}(x) = \ln(\frac{1+\sqrt{1-x^2}}{x})$, $0<x\leqq 1$, and $\dn^2-k'^2=k^2\cn^2$ we get the following relation
\begin{equation}
\label{a(k)}
\frac{a(k)}{c_2}= - \ln\Big( \frac{\dn(L/{c_1}; k)+ k \cn(L/{c_1}; k)}{k'}\Big)\equiv j(k).
 \end{equation} 
Then, numerical simulations always show that  $j(k)$  in \eqref{a(k)} is strictly decreasing  for $k\in (k_0, 1)$ and so $a'(k)<0$, for $k\in (k_0, 1)$  (see Figure \ref{fig12n} for the graph of $j = j(k)$, $k \in (k_0, 1)$ with $c_1=1$, $L=\pi$ and $k_0 \approx 0.984432$).  We note that by using formula (710.53) in \cite{ByFr71} and deriving with regard to $k$ both sides of equality in \eqref{2form-a} we can also see that $a(k)$ is strictly decreasing after long calculations.

\begin{figure}[h]
 \centering
\includegraphics[angle=0,scale=0.55]{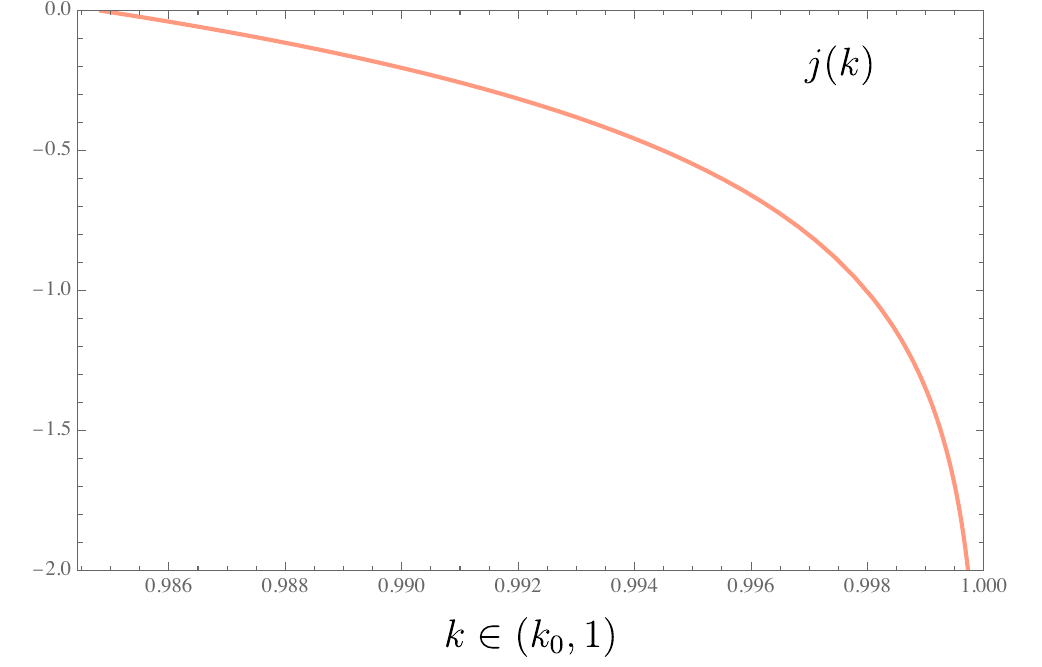}
\caption{\small{Graph of the function $j=j(k)$ in \eqref{a(k)} for $k\in (k_0, 1)$, with $c_1=1$, $L=\pi$ and where $k_0 \approx 0.984432$.} } \label{fig12n}
\end{figure}

Next, from \eqref{GF}, $\sn^2 + \cn^2=1$, $\frac{L}{c_1}+K\in [K, 2K]$ and $\cn(u+K)=-k' \sn(u)/\dn(u)$ we get for $\phi_1=\phi_{1,k}$,
\begin{equation}
\label{form-k1}
\phi'_1(L)=\frac{2k}{c_1}\cn(L/c_1+K;k)=-\frac{2kk' \sn(L/c_1;k)}{c_1\dn(L/c_1;k)}.
\end{equation}
From \eqref{trav22} and \eqref{2form-a} we get for $\phi_2=\phi_{2,a}$,
\begin{equation}
\label{form-k2}
\phi'_2(L)=-\frac{2}{c_2}\sech \Big(\frac{a}{c_2}\Big) =-\frac{2}{c_2}\frac{k'}{\dn(L/c_1;k)}.
\end{equation}
Thus, we obtain from $2\phi'_1(L)=\phi'_2(L)+Z\phi_2(L)$ the following relation 
\begin{equation}
\label{form-k3}
\frac{k'}{c_1\dn(L/c_1;k)}\Big[\frac{c_1}{2c_2}-k\sn(L/c_1;k)\Big]=Z\beta,
\end{equation}
with  $\beta\equiv \arctan(e^{-\frac{a}{c_2}})$. Therefore for $a=a(k)<0$ and $k\in (k_0, 1)$, $Z$ and $k$ are related by the equation
\begin{equation}
\label{form-k4}
Z=\frac{\sech(a/c_2)}{c_1\arctan(e^{-a/c_2})}\Big[\frac{c_1}{2c_2}-k\sn(L/c_1;k)\Big]\equiv H(k).
\end{equation}
Thus, the existence of single-lobe kink profile for the sine-Gordon model on a tadpole is reduced to solve the equation
\begin{equation}
\label{H}
H(k)=Z,
\end{equation}
for $Z\in (-\infty, \frac{2}{\pi c_2})$ fixed, and some $k\in (k_0,1)$ with  $a=a(k)$ solving \eqref{2form-a}. In other words, we will  apply the Intermediate Value Theorem.  Next, we consider the following cases:

\begin{enumerate}
\item[(I)] Let $Z>0$. Then, {\it a priori}, $k$ needs to satisfy $k\sn(L/c_1;k)<\frac{c_1}{2c_2}$ for solving \eqref{H}. Thus,  from basic properties of the Jacobi elliptic functions we get the following  limits
\begin{equation}
\label{$Z=0$}
\begin{aligned}
\lim_{k\to 1}k\sn(L/c_1;k) &= \tanh(L/c_1),\\
\lim_{k\to k_0} k\sn(L/c_1;k) &= k_0 \sn(K(k_0);k_0)=k_0.
\end{aligned}
\end{equation}
Moreover, the mapping 
\begin{equation}
\label{defofg}
k\in (k_0, 1) \to g(k)\equiv k \sn(L/c_1;k),
\end{equation}
is strictly increasing (we note that for $k<k_0$ it is possible to have $g(k)\leqq 0$ and with $g$ having some oscillations; see Figures \ref{fig12} and \ref{fig13}), and so we get  
$$
k_0=k_0 \sn(K(k_0);k_0)< k\sn(L/c_1;k)<\tanh(L/c_1), \;\;\text{for}\;\; k\in (k_0, 1). 
$$

\begin{figure}[h]
 \centering
\includegraphics[angle=0,scale=0.5]{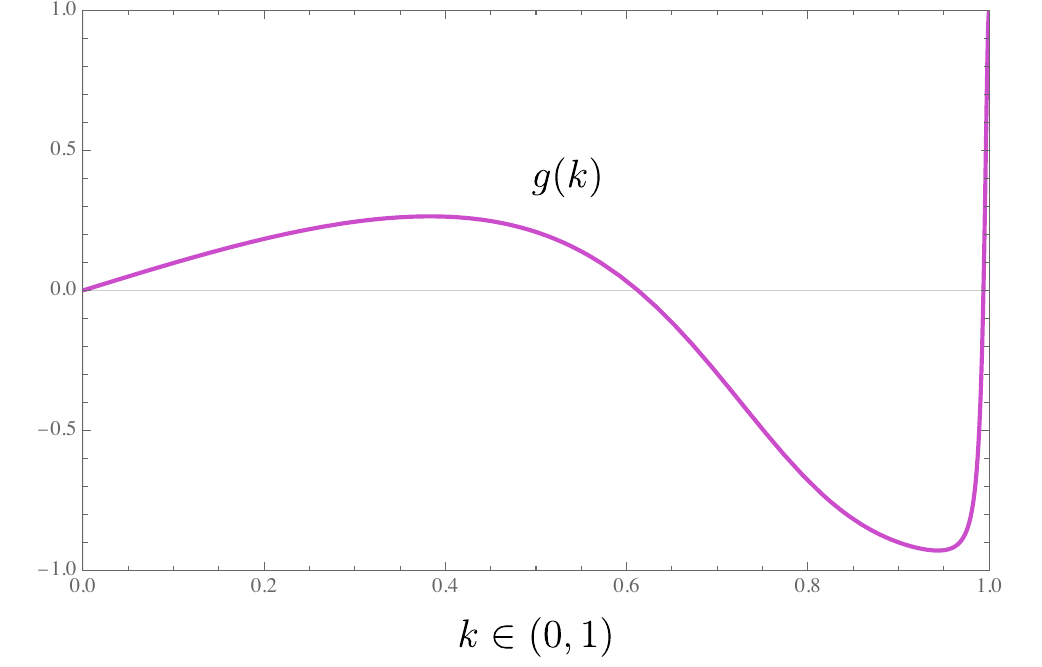}
\caption{\small{Graph of the function $g(k)$ defined on \eqref{defofg} for $k\in (0, 1)$, with $c_1=1$, $L=2.5\pi$.} } \label{fig12}
\end{figure}

\begin{figure}[h]
 \centering
\includegraphics[angle=0,scale=0.5]{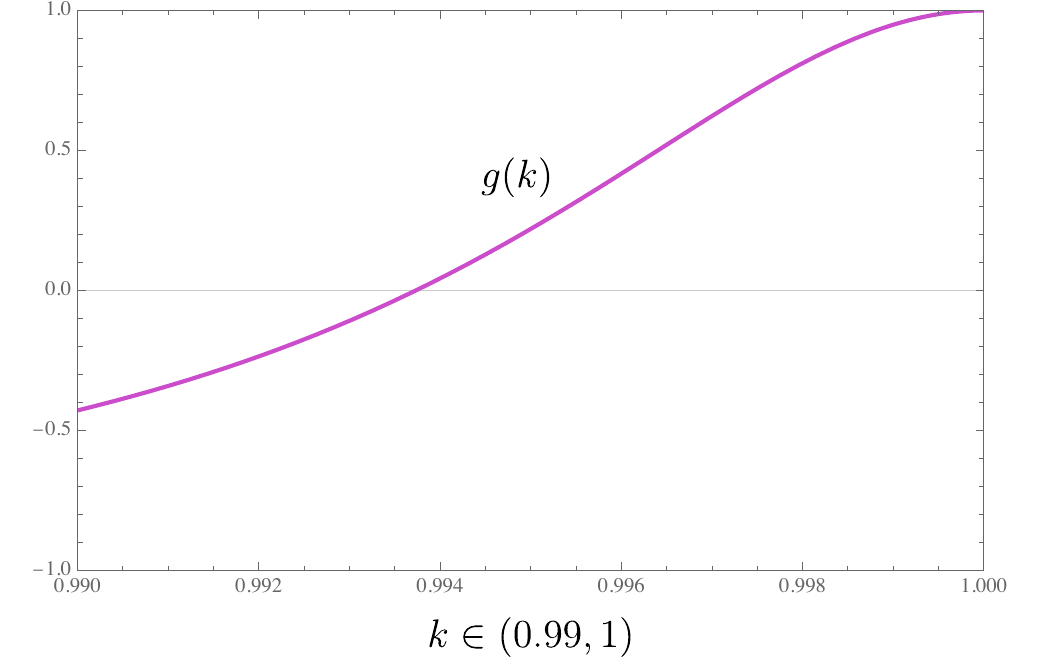}
\caption{\small{Graph of $g(k)$  for $k\in (0.99, 1)$, with $c_1=1$, $L=2.5\pi$ and $k^2_0 \approx 0.999998$.} } \label{fig13}
\end{figure}

Let us consider the following subcases:
\begin{enumerate}
\item[(i)] Suppose  $\frac{c_1}{2c_2}\geqq \tanh(L/c_1)$. Then $H(k)>0$ for $k\in (k_0, 1)$.  Now, from \eqref{limits} and properties of the Jacobi elliptic functions (cf. \cite{ByFr71}), we obtain
\begin{equation}
\label{Hk}
\begin{aligned}
\lim_{k\to 1}H(k) &= \frac{0}{\arctan(\infty)}\Big[\frac{c_1}{2c_2}-\tanh(L/c_1)\Big]=0,\\
 \lim_{k\to k_0} H(k) &=\frac{4}{\pi c_1}\Big[\frac{c_1}{2c_2}-k_0\sn(K(k_0);k_0)\Big]=\frac{4}{\pi c_1}\Big[\frac{c_1}{2c_2}-k_0\Big]>0.
\end{aligned}
\end{equation}
Thus, by analysis above ($a'(k)<0$ and $g'(k)>0$) we can see that  $H'(k)<0$ for $k\in  (k_0, 1)$, and therefore for $Z\in \big(0, \frac{4}{\pi c_1}\big[\frac{c_1}{2c_2}-k_0\big ]\big)$ there is a unique $k_Z\in (k_0, 1)$ with $H(k_Z)=Z$. This shows $(1)(i)$.

\item[(ii)] Suppose $k_0<\frac{c_1}{2c_2}< \tanh(L/c_1)$. Then, there is a unique $\gamma\in (k_0, 1)$ with $g(\gamma)=\frac{c_1}{2c_2}$ and such that for $k\in (k_0, \gamma)$ follows $H(k)>0$. From $ \lim_{k\to \beta} H(k) =0$ and \eqref{Hk}, we have  for $Z\in \big(0, \frac{4}{\pi c_1}\big[\frac{c_1}{2c_2}-k_0\big]\big)$ that there is a unique $k_Z\in (k_0, \beta)$ with $H(k_Z)=Z$. This shows $(1)(ii)$.
\end{enumerate}
\item[(II)] Now let $Z<0$. From \eqref{form-k4} follows that $k$ {\it a priori} needs to satisfy $\frac{c_1}{2c_2}<k\sn(L/c_1;k)<\tanh(L/c_1)$. Let us examine the following subcases:
\begin{enumerate}
\item[(iii)] Suppose $\frac{c_1}{2c_2}<k_0$. Then $g(k)>k_0$ for $k\in (k_0,1)$, and so $H(k)<0$ and $H'(k)>0$ for any $k\in (k_0,1)$. Therefore, for any $Z\in \big( \frac{4}{\pi c_1}\big[\frac{c_1}{2c_2}-k_0\big], 0\big)$ there is a unique $k_Z\in (k_0, 1)$ with $H(k_Z)=Z$. This shows $(1)(iii)$.
\item[(iv)]  Let  $\frac{c_1}{2c_2}= k_0$. Then  $\lim_{k\to k_0} H(k)=0$.  Let $m_0 = \min_{k\in (k_0, 1)} H(k)$, then for $Z\in [m_0, 0)$ there is at least  $k_Z\in (k_0, 1)$ such that $H(k_Z)=Z$. This proves $(1)(iv)$. 
\item[(v)] Next, for $k_0<\frac{c_1}{2c_2}<\tanh(L/c_1)$ we consider $\beta\in (k_0, 1)$ with $g(\beta)=\frac{c_1}{2c_2}$. Then, for every $k\in (\beta, 1)$ follows $g(k)> \frac{c_1}{2c_2}$ and so $H(k)<0$ for $k\in (\beta, 1)$. Since  $ \lim_{k\to \beta} H(k) =0$, we obtain for $Z\in (m_\beta, 0)$, $m_\beta = \min_{k\in (\beta, 1)} H(k)$,  that there is at least  $k_Z\in (\beta, 1)$ such that $H(k_Z)=Z$. This shows $(1)(v)$
\end{enumerate}
\item[(III)] Now suppose that  $Z=0$. From \eqref{form-k4} we need to have $H(k)=0$ and so $\frac{c_1}{2c_2}=k\sn(L/c_1;k)$ is true if and only if   $\frac{c_1}{2c_2}\in (k_0, \tanh(L/c_1))$ and in this case there is a unique $k\in (k_0, 1)$ such that $H(k)=0$. This shows $(1)(vi)$.
\end{enumerate}

\noindent $\bullet$ \textbf{Case (2)}. Now consider the case $\frac{L}{c_1}\leqq  \frac{\pi}{2}$. Then $K(k)> \frac{L}{c_1}$ for all $k\in (0,1)$. Thus, since $\dn(\cdot;0)=1$ we get from \eqref{2form-a} and \eqref{limits} again  $\lim_{k\to 1} a(k)=-\infty$ and $\lim_{k\to 0} a(k)=0$. Moreover,  the mapping $k\in (0,1)\to g(k)= k\sn(L/c_1;k)$ is strictly increasing (see Figure \ref{fig13}) with $k\sn(L/c_1;k)>0$, $\lim_{k\to 1}k\sn(L/c_1;k)=\tanh(\frac{L}{c_1})$ and $\lim_{k\to 0}k\sn(L/c_1;k)=0$.  

\begin{figure}[h]
 \centering
\includegraphics[angle=0,scale=0.5]{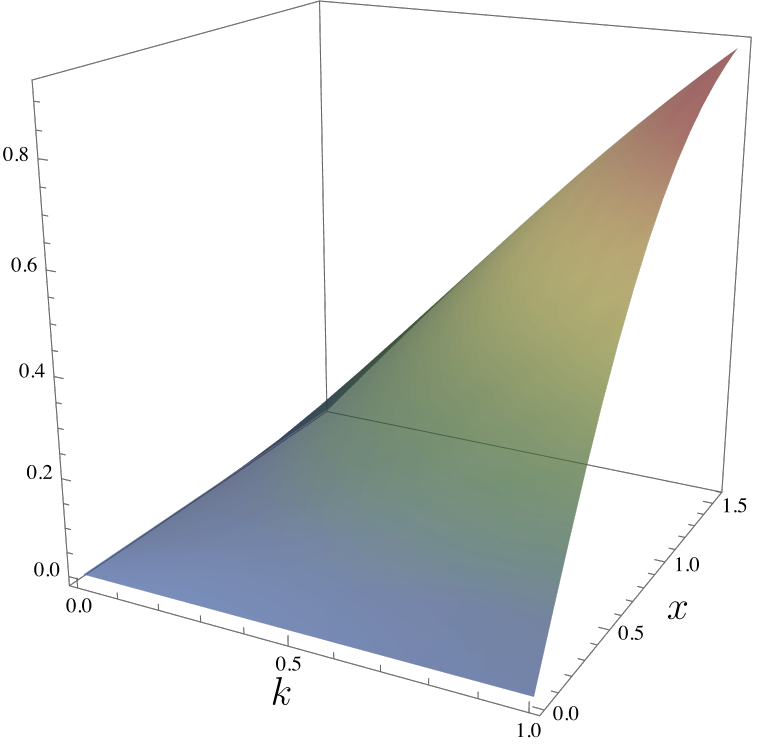}
\caption{\small{Graph of $G(k,x) = k\sn(x;k)$ for $k\in (0, 1)$ and $x\in (0,\frac{\pi}{2})$.}} \label{fig12bis}
\end{figure}

Once again, we split the analysis into subcases:
\begin{enumerate}
\item[(I)] Let $Z>0$. Then we have:
\begin{enumerate}
\item[(i)] For $\frac{c_1}{2c_2}\geqq \tanh(\frac{L}{c_1})> k\sn(L/c_1;k)$ it follows that $H(k)>0$ for $k\in (0,1)$ and $ \lim_{k\to 0} H(k) = \frac{2}{\pi c_2}$, $ \lim_{k\to 1}H(k)=0$,  and $H'(k)<0$ for all $k$. Then,  for  every $Z\in (0, \frac{2}{\pi c_2})$  there is a unique $k_Z\in (0, 1)$ with $H(k_Z)=Z$.
\item[(ii)] For  $\frac{c_1}{2c_2}< \tanh(\frac{L}{c_1})$ consider $\beta\in (0,1)$ with $g(\beta )= \frac{c_1}{2c_2}$. Then, for all $k\in (0, \beta)$ we obtain $H(k)>0$. From  $ \lim_{k\to 0} H(k) = \frac{2}{\pi c_2}$ and  $ \lim_{k\to \beta}H(k)=0$, we obtain for  every $Z\in (0, \frac{2}{\pi c_2})$  that there is a unique $k_Z\in (0, \beta)$ with $H(k_Z)=Z$.
\end{enumerate}
This proves $(2)(i)$.
\item[(II)] Let $Z<0$. For $\frac{c_1}{2c_2}< \tanh(\frac{L}{c_1})$ we consider $\beta\in (0,1)$ with $g(\beta )= \frac{c_1}{2c_2}$, then $H(k)<0$ for all $k\in (\beta,1)$. From  $ \lim_{k\to 1} H(k) = 0$ and  $ \lim_{k\to \beta}H(k)=0$, for  $Z\in (p_\beta, 0)$, $p_\beta = \min_{k\in (\beta, 1)} H(k)$,   there is at least  $k_Z\in (\beta, 1)$ such that $H(k_Z)=Z$. This proves $(2)(ii)$.
\item[(III)] Let $Z=0$. For $\frac{c_1}{2 c_2}< \tanh(\frac{L}{c_1})$, we consider the unique $k\in (0,1)$ with $g(k )= \frac{c_1}{2c_2}$. Then $H(k)=0$. This shows $(2)(iii)$.
\end{enumerate}

\noindent $\bullet$ \textbf{Case (3)}. Suppose $\frac{L}{c_1}>\frac{\pi}{2}$ and we consider  $k_0\in (0,1)$ such that $K(k_0)=\frac{L}{c_1}$.
  \begin{enumerate}
 \item[(i)] If $Z>0$ and  $\frac{c_1}{2c_2}\leqq k_0$ then $ \frac{c_1}{2c_2}-g(k)<0$ and so we do not haven a solution for $H(k)=Z$. Therefore there is not a single-lobe kink state. This proves $(3)(i)$.
 \item[(ii)] If $Z<0$ and $\frac{c_1}{2c_2} \geqq \tanh(\frac{L}{c_1})$ then $\frac{c_1}{2c_2} -g(k)>0$ and  obviously there is not a single-lobe kink state. This shows $(3)(ii)$.
\end{enumerate}

\noindent $\bullet$ \textbf{Case (4)}. Finally, suppose that $\frac{L}{c_1}\leqq \frac{\pi}{2}$. Then we have the following conditions:
 \begin{enumerate}
  \item[(i)]  For $Z<0$ and $\frac{c_1}{2c_2}\geqq  \tanh(\frac{L}{c_1})$ we have $\frac{c_1}{2c_2} -g(k)>0$ and, clearly, there is not a single-lobe kink state. This proves $(4)(i)$
  \item[ii)]  For $Z=0$ and $\frac{c_1}{2c_2}\geqq  \tanh(\frac{L}{c_1})$ there holds $\frac{c_1}{2c_2} -g(k)>0$ and, clearly, there is not a single-lobe kink state. This shows $(4)(ii)$.
\end{enumerate} 

The Lemma is now proved.
\end{proof}

\def\cprime{$'\!\!$} \def\cprimel{$'\!$}

 \end{document}